\documentclass[11pt]{amsart}
\usepackage{amsmath, amssymb, amscd, cancel, graphicx, soul, stmaryrd, accents, bbm}
\pdfoutput=1
\usepackage[colorlinks=true, urlcolor=NavyBlue, linkcolor=NavyBlue, citecolor=NavyBlue, bookmarks=false]{hyperref}

\usepackage[dvipsnames,usenames]{color}
\usepackage{mathdots}
\usepackage{float}
\usepackage{amsthm}
\usepackage{amsfonts}
\usepackage{enumerate}
\usepackage{paralist}
\usepackage{xypic}
\usepackage{subfigure}
\usepackage{hyperref}
\usepackage{color}
\usepackage{marginnote}
\usepackage[all]{xy}
\numberwithin{equation}{section}
\usepackage{tikz}
\usepackage{wrapfig}
\usepackage{caption}
\usepackage{pinlabel}
\usepackage{multirow}
\hypersetup{
linkbordercolor={1 0 0}, 
citebordercolor={0 1 0} 
}
\newtheorem{thm}{Theorem}[section]
\newtheorem{conj}[thm]{Conjecture}
\newtheorem{cor}[thm]{Corollary}
\newtheorem{lemma}[thm]{Lemma}
\newtheorem{prop}[thm]{Proposition}

\newtheorem{defn}[thm]{Definition}
\newtheorem{rmk}[thm]{Remark}

\newtheorem{example}[thm]{Example}

\newcommand{\Q}{\mathbb Q}

\begin{document}

\title[Taut Foliations, Positive 3-Braids, and the L-Space Conjecture]
{Taut Foliations, Positive 3-Braids, and the L-Space Conjecture}

\author{Siddhi Krishna}
\address{Department of Mathematics, Boston College\\ Chestnut Hill, MA 02467}
\email{siddhi.krishna@bc.edu}

\maketitle

\medskip

\noindent {\bf Abstract.} 
We construct taut foliations in every closed 3--manifold obtained by $r$--framed Dehn surgery along a positive 3--braid knot $K$ in $S^3$, where $r < 2g(K)-1$ and $g(K)$ denotes the Seifert genus of $K$. This confirms a prediction of the L--space Conjecture. For instance, we produce taut foliations in every non--L--space obtained by surgery along the pretzel knot $P(-2,3,7)$, and indeed along every pretzel knot $P(-2,3,q)$, for $q$ a positive odd integer. This is the first construction of taut foliations for every non--L--space obtained by surgery along an infinite family of hyperbolic L--space knots. Additionally, we construct taut foliations in every closed 3--manifold obtained by $r$--framed Dehn surgery along a positive 1--bridge braid in $S^3$, where $r <g(K)$.

\section{Introduction}

The L-space Conjecture predicts a surprising relationship between Floer-homological, algebraic, and geometric properties of a closed 3-manifold $Y$:

\begin{conj}[The L-space Conjecture \cite{BoyerGordonWatson, Juhasz:Survey}] \label{conj:LSpace} 
Suppose $Y$ is an irreducible rational homology 3-sphere. Then the following are equivalent: 
\begin{enumerate}
\item $Y$ is a non-L-space (i.e.~ the Heegaard Floer homology of $Y$ is not ``simple"),
\item $\pi_1(Y)$ is left-orderable, and
\item $Y$ admits a taut foliation.
\end{enumerate}
\end{conj}

Work by many researchers fully resolves Conjecture \ref{conj:LSpace} in the affirmative for graph manifolds \cite{BoyerClay2, BoyerClay, BoyerGordonWatson, BrittenhamNaimiRoberts, ClayLidmanWatson, EisenbudHirschNeumann, HRRW,  LiscaStipsicz}. 
Combining results of Ozsv\'ath-Szab\'o, Bowden, and Kazez-Roberts proves that if $Y$ admits a taut foliation, then $Y$ is a non-L-space \cite{OSz:HolDisks, Bowden:Approx, KazezRoberts}. Here, we investigate the converse.

One strategy for producing non-L-spaces is via Dehn surgery. A non-trivial knot $K \subset S^3$ is an \textbf{L-space knot} if \textit{some} non-trivial surgery along $K$ produces an L-space. Lens spaces are prominent examples of L-spaces, so any knot with a non-trivial surgery to a lens space (notably Berge knots \cite{Berge}) is an L-space knot. Berge-Gabai knots are the subclass of 1-bridge braids in $S^3$ admitting lens space surgeries \cite{Gabai:1BridgeBraids, Berge}, yet \textit{every} 1-bridge braid is an L-space knot \cite{GLV:11LSpace}.

In fact, if $K$ is an L-space knot, \textit{infinitely} many surgeries along $K$ yield L-spaces. In particular, for any $K$ realized as the closure of a positive braid, the set of L-space surgery slopes is either $[2g(K)-1, \infty) \cap \Q$, or the empty set \cite{Livingston:TauInvariant, OSz:LensSpaceSurgeries, KMOSz, Rasmussen2}. Thus, $r$-framed Dehn surgery along \textit{any} non-trivial knot realized as a positive braid closure yields a non-L-space for all $r < 2g(K)-1$. This viewpoint guides our treatment of Conjecture \ref{conj:LSpace}, which predicts these manifolds admit taut foliations.

\begin{thm}\label{thm:main}
Let $K$ be a knot in $S^3$, realized as the closure of a positive 3-braid. Then for every $r <2g(K)-1$, the knot exterior $X_K := S^3 - \accentset{\circ}{\nu}(K)$ admits taut foliations meeting the boundary torus $T$ in parallel simple closed curves of slope $r$. Hence the manifold obtained by $r$-framed Dehn filling, $S^3_r(K)$, admits a taut foliation.
\end{thm}

\begin{rmk} Theorem \ref{thm:main} can be reformulated as follows: for $K$ and $r$ as above, the manifold $S_r^3(K)$ admits a taut foliation, such that the core of the Dehn surgery is a closed transversal. 
\end{rmk}

A 3-stranded twisted torus knot is a knot obtained as the closure of $(\sigma_1 \  \sigma_2)^q (\sigma_2)^{2s}$, where $q$ and $s$ are positive integers, and $\sigma_1, \sigma_2$ are the standard Artin generators. Vafaee proved every 3-stranded twisted torus knot is an L-space knot \cite{Vafaee:TwistedTorusKnots}. Moreover, if an L-space knot admits a presentation as a 3-braid closure, then $K$ is a twisted torus knot \cite{LeeVafaee}. Thus, hyperbolic L-space knots are abundant among positive 3-braid closures. Applying Theorem \ref{thm:main} yields:

\begin{cor}
In Conjecture \ref{conj:LSpace}, (1) $\iff$ (3) holds for all Dehn surgeries along an infinite family of hyperbolic L-space knots. \hfill $\Box$
\end{cor}

Baker-Moore, strengthening results of Lidman-Moore, proved that the only L-space Montesinos knots are the pretzel knots $P(-2,3,q)$, for $q \geq 1, \ q$ odd \cite{LidmanMoore:PretzelKnots, BakerMoore:Montesinos}. These knots are realized as closures of positive 3-braids (see Figure \ref{fig:pretzelknotisotopy}). Applying Theorem \ref{thm:main}, we deduce:

\begin{cor}\label{thm:pretzelknots}
Let $K$ be an L-space Montesinos knot in $S^3$. Then for any $r$-framed surgery on $K$, the surgered manifold $Y = S_r^3(K)$ is a non-L-space $\iff$ $Y$ admits a taut foliation. \hfill $\Box$
\end{cor}

We note that Delman-Roberts recover Corollary \ref{thm:pretzelknots} in forthcoming work \cite{DelmanRoberts}.

\begin{example} 
\textup{The Fintushel-Stern pretzel knot $P(-2,3,7)$ is a hyperbolic knot in $S^3$ admitting lens space surgeries \cite{FintushelStern}, hence is an L-space knot. It can be realized as a positive 3-braid closure in $S^3$ (see Figure \ref{fig:pretzelknotisotopy}). In Section \ref{section:example}, we explicitly construct the family of taut foliations meeting the boundary torus $T$ in all rational slopes $r < 2g(K)-1 = 9$.}
\end{example}

Tran, generalizing work of Nie \cite{Nie:LOPretzelKnots}, showed that for any $K$ in an infinite subfamily $\mathcal{F}$ of 3-stranded twisted torus knots, and $r \geq 2g(K)-1$, $\pi_1(S^3_r(K))$ is not left-orderable \cite{Tran}. The L-space pretzel knots comprise a proper subset of $\mathcal{F}$. We conclude:

\begin{cor}
Suppose $Y$ is obtained by $r$-framed Dehn surgery along $K$ in $S^3$, for $K$ a 3-stranded twisted torus knot in $\mathcal{F}$, and $r \in \Q$. Then
\begin{center}
$\pi_1(Y)$ is not left-orderable $\impliedby$ $Y$ is an L-space $\iff$ $Y$ does not admit a taut foliation.
\end{center}
That is, $(2) \implies (1) \iff (3)$ of Conjecture \ref{conj:LSpace} holds for manifolds obtained by Dehn surgeries along knots in $\mathcal{F}$. 
\end{cor}

Our methods for proving Theorem ~\ref{thm:main} are constructive. Inspired by work of Roberts \cite{Roberts:Part1, Roberts:Part2}, we build \textbf{sink disk free} branched surfaces in fibered knot exteriors. By Li \cite{TaoLi:SinkDisk, TaoLi:BoundarySinkDisk}, these branched surfaces carry essential laminations. We first extend these laminations to taut foliations in knot exteriors, and then to taut foliations in surgered manifolds. 

Conjecture \ref{conj:LSpace} predicts Theorem \ref{thm:main} holds for any knot $K$ realized as a positive braid closure, on any number of strands. Any such $K$ is fibered; applying \cite{Roberts:Part2}, $S^3_r(K)$ admits a taut foliation for any $r <1$. An adaptation of our techniques partially closes the gap between Roberts' result and the prediction for 1-bridge braids in $S^3$:

{\begin{thm}\label{thm:1bridgebraids}
Let $K$ be any (positive) 1-bridge braid in $S^3$, i.e.~ $K$ is a knot in $S^3$, realized as the closure of a braid $\beta$ on $w$ strands, where $$\beta = (\sigma_b \sigma_{b-1} \ldots \sigma_2 \sigma_1) (\sigma_{w-1} \sigma_{w-2} \ldots \sigma_2 \sigma_1)^t$$ for $ w \geq 3, 1 \leq b \leq w-2, t \geq 1$. Then for every $r<g(K)$, the knot exterior $X_K := S^3 - \accentset{\circ}{\nu}(K)$ admits taut foliations meeting the boundary torus $T$ in parallel simple closed curves of slope $r$. Hence the manifold obtained by $r$-framed Dehn filling, $S^3_r(K)$, admits a taut foliation.
\end{thm}

\subsection{Organization}
In Section \ref{section:background}, we review the required background on branched surfaces and fibered knots. In Section \ref{section:example}, we establish the foundations for proving Theorem \ref{thm:main}. Along the way, we construct taut foliations for every $S_r^3(K)$, where $K = P(-2,3,7)$ and $r<9$. In Section \ref{section:3braids}, we prove Theorem \ref{thm:main}. In Section \ref{section:1bridgebraids}, we prove Theorem 
\ref{thm:1bridgebraids}. \\

\begin{wrapfigure}[1]{o}{.18\linewidth}
\labellist
\small
\pinlabel $v$ at 58 246
\pinlabel $w$ at 246 57 
\pinlabel {$\langle v, w \rangle = 1$} at 180 140
\endlabellist
\begin{minipage}[c][.1\linewidth]{20mm}
\centering
\includegraphics[scale=.30]{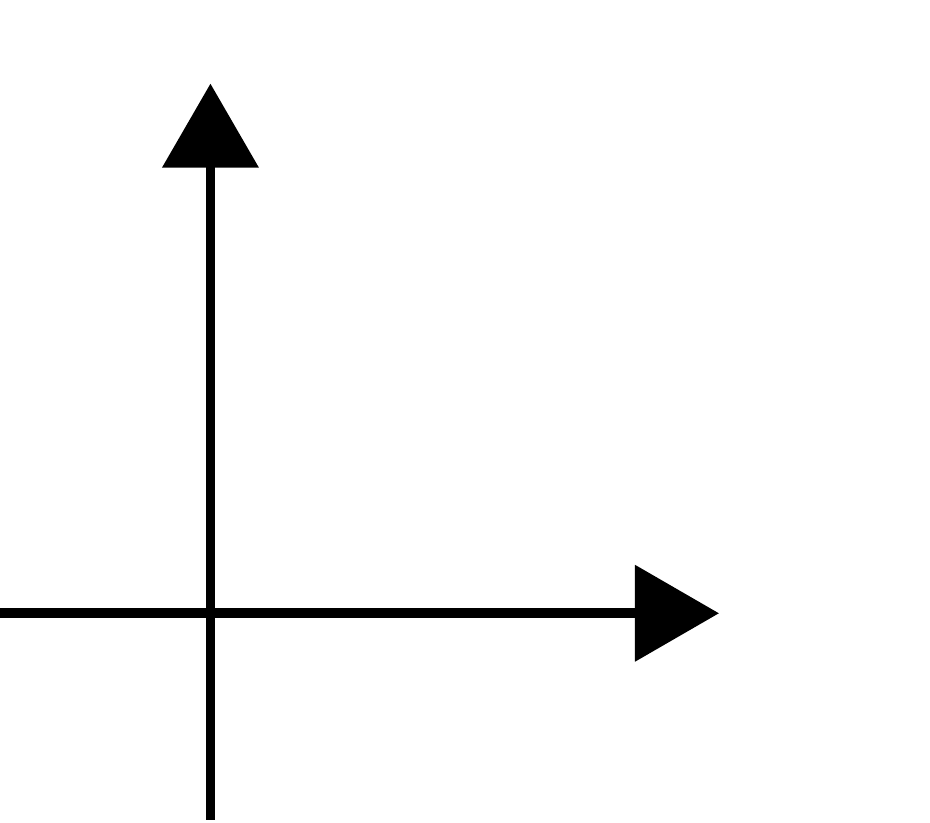}
\end{minipage}
\label{fig:intersection_number}
\end{wrapfigure}

\subsection{Conventions} \label{conventions}
\begin{itemize}
\item We work only with braid closures which are knots in $S^3$.
\item For any knot exterior $X_K$, $H_1(\partial X_K)$ is generated by \\the Seifert longitude $\lambda$ and the standard meridian $\mu$. 
\item Let $\langle \alpha, \beta \rangle$ denote the algebraic intersection number; following the sign convention above, we set $\langle \lambda, \mu \rangle = 1$. For any essential simple closed curve $\gamma$ on $T = \partial X_K$, the slope of $\gamma$ is determined by $\displaystyle \frac{\langle \gamma, \lambda \rangle}{\langle \mu, \gamma \rangle}$.
\item We use $\sigma_1, \sigma_2, \ldots, \sigma_{n-1}$ to represent the standard Artin generators for the $n$-stranded braid group. Strands are drawn vertically, oriented ``down", and enumerated from left-to-right. Given a braid diagram, we recover the braid word by reading $\beta$ from top-to-bottom.
\item The surface $F$ will always be orientable; in all figures of Seifert surfaces, only $F^{+}$ is visible.
\item If a properly embedded arc $\alpha$ lies on $F^{-}$, it is drawn with a \textcolor{blue}{blue} dotted line; if $\alpha$ lies on $F^{+}$, it is drawn with a \textcolor{RubineRed}{pink} solid line. A helpful mnemonic: ``\textbf{p}ink" and ``\textbf{p}lus" both start with ``\textbf{p}".

\item Given a fibered knot $K \subset S^3$ with fiber $F$ and monodromy $\varphi$, the knot exterior is a mapping torus $F \times [0,1] / \sim$, where $(x,0) \sim (\varphi(x), 1)$. Moreover, $\varphi \approx \mathbbm{1}$ in $\nu(\partial F)$. \\
\end{itemize}

\subsection{Acknowledgements} I am grateful to my advisor, Josh Greene, for his guidance, patience, and kindness. Thanks to Tao Li for answering countless questions, and John Baldwin for sharing his vision to extend Roberts' results. Additionally, thanks to Peter Feller, Kyle Hayden, and Patrick Orson for helpful conversations. Finally, we thank the referee for their careful reading and valuable suggestions.

\section{Background} \label{section:background}

\subsection{Branched Surfaces}
Our primary tool for constructing taut foliations are branched surfaces. For a detailed exposition on branched surfaces, see Floyd-Oertel \cite{FloydOertel}.

\begin{defn}
A \textbf{spine for a branched surface} is a 2-complex in a 3-manifold $M$, locally modeled by:
\end{defn}

\begin{figure}[h]\center
\includegraphics[scale=.3]{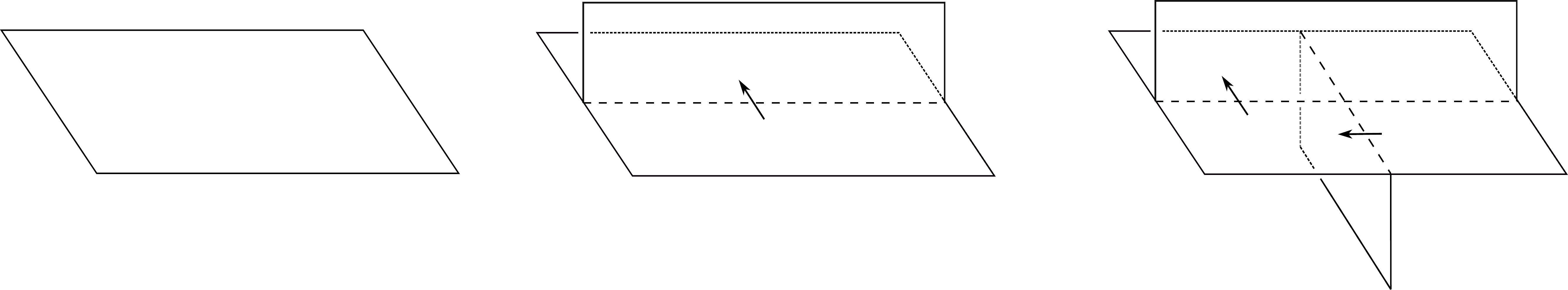}
\caption{Ignoring the arrows yields the local models for the spine of a branched surface.} 
\label{fig:spine}
\end{figure}

\begin{defn}
A \textbf{branched surface} $B$ in a 3-manifold $M$ is built by providing smoothing/cusping instructions for a spine. It is locally modeled by:
\end{defn}

\begin{figure}[h]\center
\includegraphics[scale=.3]{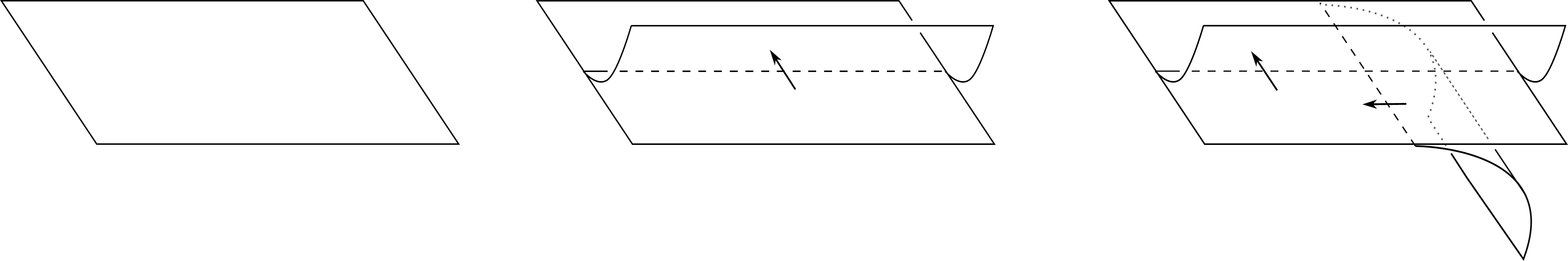}
\caption{The cusping instructions for the spine in Figure \ref{fig:spine} yield these local models.}
\label{fig:branched_surface}
\end{figure}

A branched surface is locally homeomorphic to a surface everywhere except in a set of properly embedded arcs and simple closed curves, called the \textbf{branch locus $\gamma$}. A point $p$ in $\gamma$ is called a \textbf{triple point} if a neighborhood of $p$ in $B$ is locally modeled by the rightmost picture of Figure \ref{fig:branched_surface}. A \textbf{branch sector} is a connected component of $\overline{B - \gamma}$ (the closure under the path metric). In this paper, all branched surfaces meet the boundary torus of $X_K$; it will do so in a train track. 

\begin{defn}
A \textbf{sink disk} \cite{TaoLi:SinkDisk} is a branch sector $S$ of $B$ such that (1) $S$ is homeomorphic to a disk, (2) $\partial S \cap \partial M = \varnothing$, and (3) the branch direction of every smooth arc or curve in its boundary points into the disk. A \textbf{half sink disk} \cite{TaoLi:BoundarySinkDisk} is a branch sector $S$ of $B$ such that (1) $S$ is homeomorphic to a disk, and (2) $\partial S \cap \partial M \neq \varnothing$, and (3) the branch direction of each arc in $\partial S - \partial M$ points into $S$. Note: $\partial S \cap \partial M$ may not be connected. When a branched surface $B$ contains no sink disk or half sink disk, we say $B$ is \textbf{sink disk free}. See Figure \ref{fig:sinkdiskdefn}.
\end{defn}

Thus, to prove a branched surface is sink disk free, we need only check that some cusped arc points out of each branch sector. Indeed, this is the heart of the proof of Theorem \ref{thm:main}.

\begin{figure}[h]\center
\includegraphics[scale=.6]{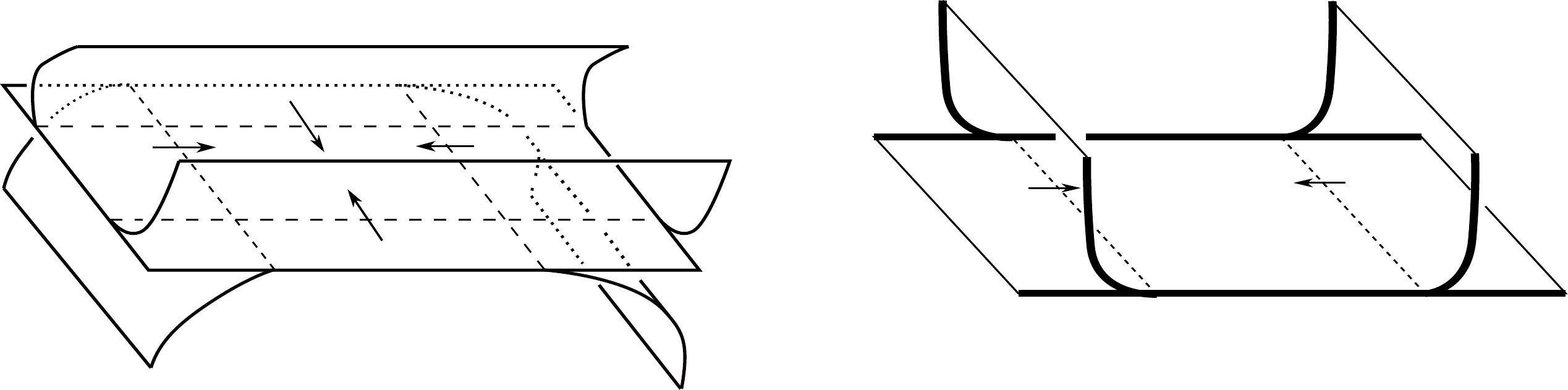}
\caption{On the left, the local model of a \textbf{sink disk}. On the right, the \textbf{bolded} lines lie on $\partial M \approx T^2$; this is the local model for a \textbf{half sink disk}. }
\label{fig:sinkdiskdefn}
\end{figure}

Gabai and Oertel prove a lamination $\mathcal{L}$ is essential if and only if $\mathcal{L}$ is carried by an essential branched surface $B$ \cite{GabaiOertel}. Li proves that for $B$ to carry an essential lamination, it suffices to be sink disk free:

\begin{thm}[Theorem 2.5 in \cite{TaoLi:BoundarySinkDisk}]\label{thm:taolisinkdisk}
Suppose $M$ is an irreducible and orientable 3-manifold whose boundary is an incompressible torus, and $B$ is a properly embedded branched surface in $M$ such that
\begin{enumerate}
\item[(1a)] $\partial_h(N(B))$ is incompressible and $\partial$-incompressible in $M - \text{int}(N(B))$
\item[(1b)] There is no monogon in $M - \text{int}(N(B))$
\item[(1c)] No component of $\partial_h N(B)$ is a sphere or a disk properly embedded in $M$
\item[(2)] $M - \text{int}(N(B))$ is irreducible and $\partial M - \text{int}(N(B))$ is incompressible in $M - \text{int}(N(B))$
\item[(3)] $B$ contains no Reeb branched surface (see \cite{GabaiOertel} for more details)
\item[(4)] $B$ is sink disk free
\end{enumerate}
Suppose $r$ is any slope in $\Q \cup \{\infty\}$ realized by the boundary train track $\tau_B = B \cap \partial X_K$. If $B$ does not carry a torus that bounds a solid torus in $M(r)$, the manifold obtained by $r$-framed Dehn filling, then (1) $B$ carries an essential lamination in $M$ meeting the boundary torus in parallel simple closed curves of slope $r$, and (2) $M(r)$ contains an essential lamination. 
\end{thm}

\begin{rmk}
Our version of Theorem \ref{thm:taolisinkdisk} differs mildly from the version in \cite{TaoLi:BoundarySinkDisk}. The discrepancy arises from our consideration of the lamination in $M$; this is not problematic, as the lamination in $M(r)$ meets the surgery torus in simple closed curves of slope $r$.
\end{rmk}

A branched surface satisfying conditions (1--4) in Theorem \ref{thm:taolisinkdisk} is called a \textbf{laminar branched surface}. To prove Theorem \ref{thm:main} for any positive 3-braid knot $K$, we construct a laminar branched surface $B$ and prove the boundary train track $\tau$ carries all rational slopes $r <2g(K)-1$. Applying Theorem \ref{thm:taolisinkdisk}, we deduce the existence of essential laminations in $X_K$, which we extend to taut foliations in $X_K$. 

\subsection{Product Disks} \label{section:productdisks}
Positive braid closures are fibered links \cite{Stallings:Fibered}. This statement can be proved concretely via disk decomposition \cite{Gabai:Fibered}. We recount the relevant details of Gabai's method. 

For $K \subset S^3$, let $F$ be a genus $g$ orientable Seifert surface for $K$. $F \times I$ is a genus $2g$ handlebody $H$, and $\partial H \approx F^+ \cup F^{-} \cup A$, where $A \approx K \times I$. This is an example of a \textbf{sutured manifold} with annular suture $A$, formally written as $(F \times I, \partial F \times I) \approx (F \times I, K \times I) \approx (M, \gamma)$.

A \textbf{product disk} is a disk $D^2$ in the \textbf{complementary sutured manifold} $(X_F, \partial F \times I)$, $X_F :=\overline{S^3 - (F \times I)}$, such that $\partial D^2 \approx S^1$ meets the suture $A$ exactly twice. Given a product disk in $X_F$, we can \textbf{decompose along it}, by cutting $X_F$ along $D$ and creating a new sutured manifold $M' \approx \overline{X_F - (D \times I})$. The sutures $\gamma$ of $M$ can be modified in one of two ways to form the sutures $\gamma'$ of $M'$: at the sites where $\gamma \cap \partial M'$, connect the ends of $\gamma \cap (\partial D \times (\pm 1))$ by diameters of $D\times \{\pm 1 \}$. Writing $(M, \gamma) \stackrel{D}{\leadsto} (M', \gamma')$ denotes a \textbf{(product) disk decomposition}. 

\begin{wrapfigure}{R}{.28\linewidth}
\labellist
\pinlabel {$\alpha$} at -50 850
\pinlabel {$\varphi(\alpha)$} at 750 650
\endlabellist
\begin{minipage}[c][.25\paperheight]{50mm}
\begin{center}
\includegraphics[scale=.12]{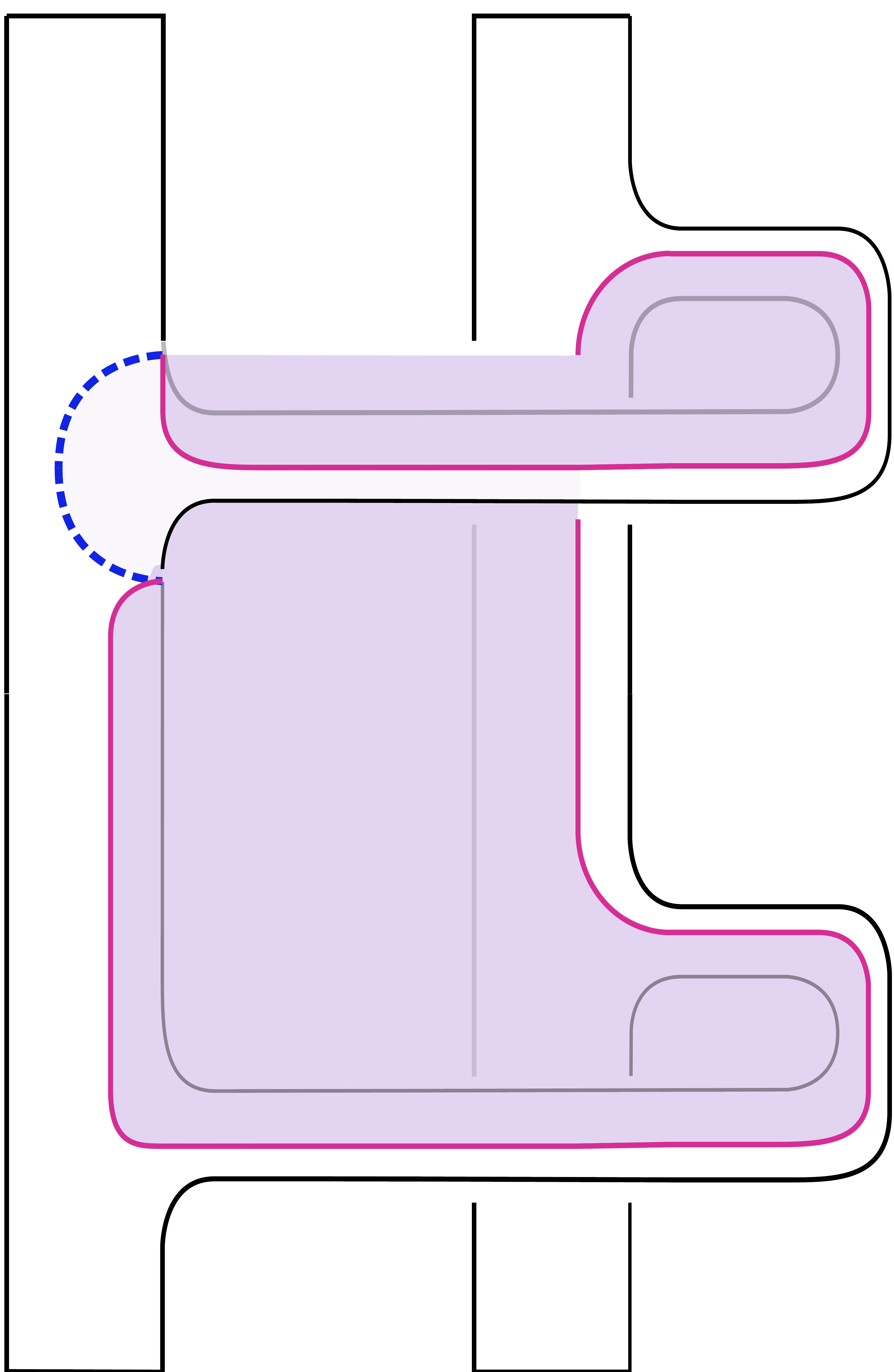}
\captionof{figure}{The product \\disk $D$ for a positive \\Hopf Band. We see \\ $\partial D|_{F^{+} \cup F^{-}} \approx \alpha \cup \varphi(\alpha)$,\\ where $\varphi$ is a positive \\Dehn twist about the \\core curve.}
\label{fig:eg_productdisk}
\end{center}
\end{minipage}
\end{wrapfigure}

\begin{thm}[Theorem 1.9 in \cite{Gabai:Fibered}]
A link $L \subset S^3$ is fibered with fiber surface $F$ if and only if a sequence of product disk decompositions, applied to $(X_F, \partial F \times I)$, terminates with a collection of product sutured balls $(B^3, S^1 \times I)$.
\end{thm}

When $K$ is a fibered knot in $S^3$, the sequence of product disk decompositions terminates with a single $(B^3, S^1 \times I)$. 

A sequence of disk decompositions to a product sutured ball not only certifies fiberedness, but also determines where the monodromy sends properly embedded arcs on $F$. Let $F$ be a fiber surface for $K \subset S^3$; thus, $(F \times I, A)$ is a trivial product sutured manifold. Heuristically, all the data pertaining to the monodromy of the fibered knot is captured by the complementary sutured manifold. In particular, let $\alpha$ be an essential properly embedded arc on $F^{-}$. Now, view $\alpha$ as an arc on $F^{-} \subset \partial (F \times I)$ with $\partial \alpha \subset \partial A$.  
Pushing $\alpha$ through the complementary sutured manifold $(X_F \approx F \times I, \partial F \times I)$ yields  a disk $D \approx \alpha \times I$, where $\partial D$ meets the suture twice, and $\overline{\partial D - A} = \alpha^{+} \sqcup \alpha^{-}$, with $\alpha^{\star} \subset F^{\star}$.  $D$ is a product disk, and $\varphi(\alpha^{-}) \approx \alpha^{+}$. See Figure \ref{fig:eg_productdisk} for an example.

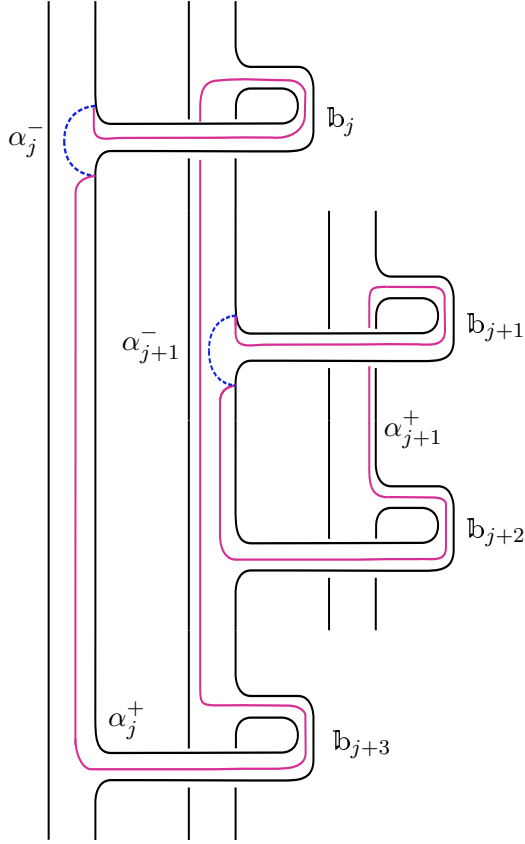
\begin{wrapfigure}{L}{.66\linewidth}
\labellist
\pinlabel {$\mathbbm{\alpha}_{j}^{-}$} at -60 2050
\pinlabel {$\mathbbm{\alpha}_{j}^{+}$} at 230 350
\pinlabel {$\mathbbm{\alpha}_{j+1}^{-}$} at 300 1450
\pinlabel {$\mathbbm{\alpha}_{j+1}^{+}$} at 1060 1200
\pinlabel {$\mathbbm{b}_{j}$} at 850 2100
\pinlabel {$\mathbbm{b}_{j+1}$} at 1300 1500
\pinlabel {$\mathbbm{b}_{j+2}$} at 1300 900
\pinlabel {$\mathbbm{b}_{j+3}$} at 910 280
\endlabellist
\begin{minipage}[c][.52\paperheight]{70mm}
\begin{raggedright}
\vspace{1.5cm}
\includegraphics[scale=.14]{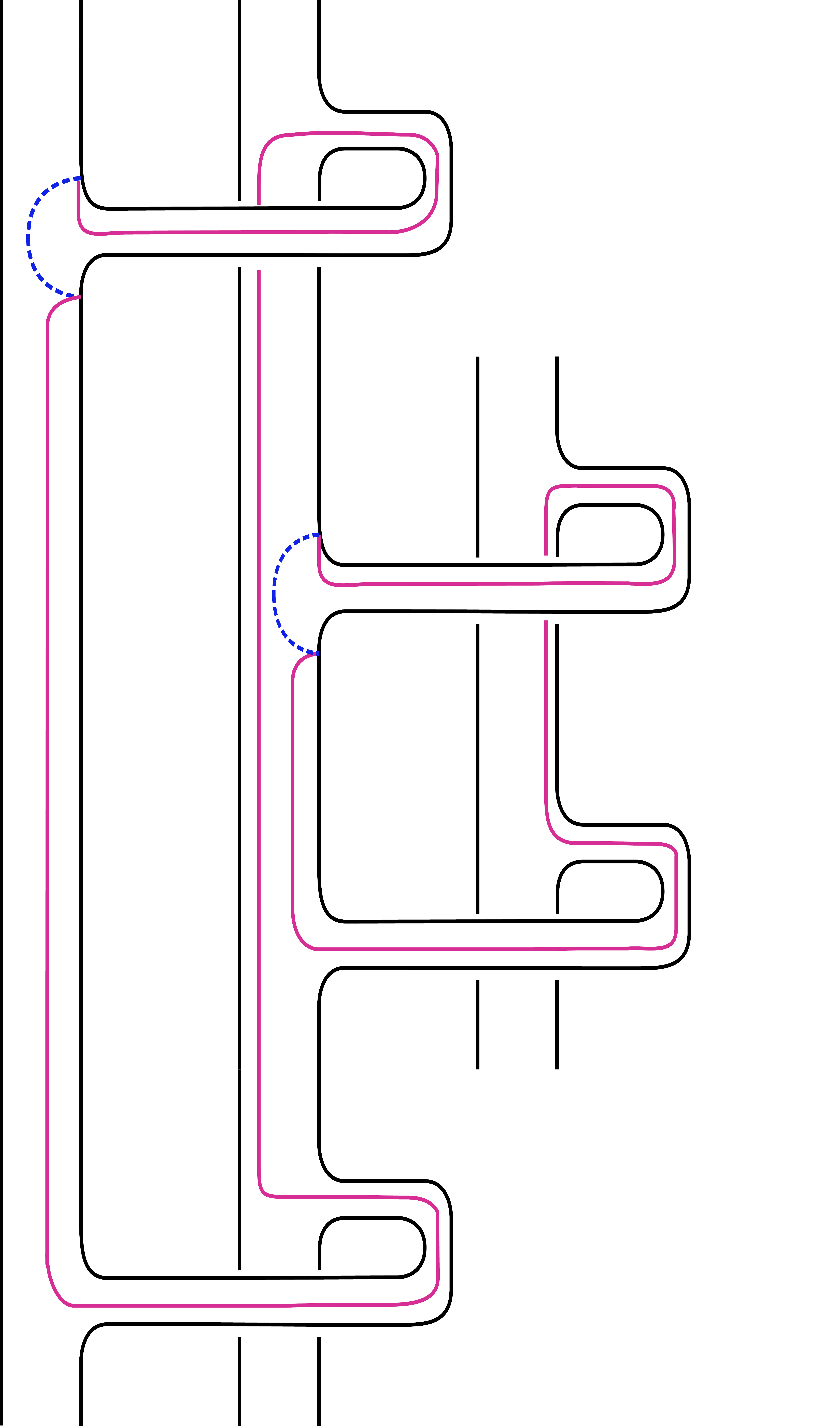}
\end{raggedright}
\captionof{figure}{There are two product disks \\ identified, $D_j$ and $D_{j+1}$. We have \\$\partial D_j \subset S_1 \cup S_2 \cup \mathbbm{b}_j \cup \mathbbm{b}_{j+3}$, and \\ $\partial D_{j+1} \subset S_2 \cup S_3 \cup \mathbbm{b}_{j+1} \cup \mathbbm{b}_{j+2}$. \\ The non-sutured portions of $\partial D_j$ \\ and $\partial D_{j+1}$ are $\alpha_j^- \cup \alpha_j^+$ and \\ $\alpha_{j+1}^- \cup \alpha_{j+1}^+$, respectively. \\ \\}
\label{fig:2productdisks}
\end{minipage}
\end{wrapfigure}

\begin{rmk}
Positive braid closures are obtained by a sequence of plumbings of positive Hopf bands. One can inductively apply Corollary 1.4 in \cite{Gabai:MurasugiSumII} to produce an explicit factorization of the monodromy in terms of Dehn twists.
\end{rmk}

\subsection{Constructing the Fiber Surface for Positive 3-braid Closures:} \label{FiberSurface3Braids}

Let $\beta$ be a positive 3-braid, where $\beta$ is not one of $\sigma_1^s, \sigma_2^s,$ or $\sigma_1 \sigma_2$. For such braids, conjugation and repeated applications of the braid relation $\sigma_2 \sigma_1 \sigma_2 = \sigma_1 \sigma_2 \sigma_1$ eliminate isolated instances of $\sigma_1$ \cite{Baader:PositiveBraidsSignature}. Thus, every such positive 3-braid can be written in the form
\begin{align} \label{3braidstandardform}
\beta = \sigma_1^{a_1} \sigma_2^{b_1} \sigma_1^{a_2} \sigma_2^{b_2} \ldots \sigma_1^{a_k} \sigma_2^{b_k}, \qquad \text{ where for all } i \leq k,\  2 \leq a_i \text{ and }1 \leq b_i \end{align}

Going forward, we assume all 3-braids are in this form.

\begin{defn}
Let $\beta$ be of the form described in (\ref{3braidstandardform}).  $\beta$ has \textbf{$k$ blocks}, where the $i^{\text{th}}$ block has the form $\sigma_1^{a_i} \sigma_2^{b_i}$. 
\end{defn}

\begin{defn}
Let $\hat{\beta}$ denote the closure of $\beta$, which is in the form specified by Equation \ref{3braidstandardform}. Define:
\begin{minipage}[r]{.05\paperheight}
\begin{center}
$$ \ \ \ \ c_1 := \sum_{i=1}^k a_i \qquad \qquad c_2 := \sum_{i=1}^k b_i$$
\end{center}
\end{minipage}
\end{defn}

Applying Seifert's algorithm to $\hat{\beta}$ yields Seifert disks $S_1, S_2, S_3$. Reading $\beta$ from left to right, each occurrence of $\sigma_i$ dictates the attachment of a positively twisted band between $S_i$ and $S_{i+1}$. 

\begin{defn}
For the $j^{\text{th}}$ letter $\sigma_i$ in the braid word $\beta$, denote the corresponding positively twisted band attached between $S_i$ and $S_{i+1}$ as $\mathbbm{b}_{j}$.
\end{defn}

The bands are attached from top to bottom; there are $c_1 + c_2$ bands attached in total. This is our fiber surface $F$ for $\hat{\beta}$. Following conventions established by Rudolph \cite{Rudolph:QPsliceness}, we only see $F^{+}$, the ``positive side" of $F$, in our figures.

\begin{defn}
The bands $\mathbbm{b}_{j}$ and $\mathbbm{b}_{k}$ are \textbf{of the same type} if they are both attached between the Seifert disks $S_i$ and $S_{i+1}$.
\end{defn}

It is straightforward to identify a collection of product disks for $F$: the boundary of a disk $D_j$ will be entirely contained in $\mathbbm{b}_{j}$, $\mathbbm{b}_k$ (the next band of the same type as $\mathbbm{b}_{j}$), and $S_i \cup S_{i+1} \cup A$ (where $S_i$ and $S_{i+1}$ are the Seifert disks to which $\mathbbm{b}_{j}$ and $\mathbbm{b}_{k}$ are attached). Decomposing $X_F$ along $c_1 + c_2 - 2$ disks results in a single product sutured ball. Since fiber surfaces are minimal genus Seifert surfaces, we conclude $\chi(F)=3 - (c_1 + c_2)$  and $2g(K)-1 = c_1 + c_2 - 3$.

\begin{defn}
Suppose a product disk has boundary contained in $\mathbbm{b}_{j}$ and $\mathbbm{b}_{k}$, which are bands of the same type with $j < k$. We refer to this disk as $D_{j}$. Furthermore, we denote the non-sutured portion of $\partial D_j$, $\overline{\partial D_{j}- A}$, by $\alpha_{j}^{+} \cup \alpha_{j}^{-}$, where $\alpha_j^\star \subset F^\star$.
\end{defn}

The product disk $D_j$ is completely determined by the arcs $\alpha_{j}^{-}$ and $\alpha_{j}^{+} \approx \varphi(\alpha_j^{-})$, so we use these arcs to identify product disks -- in particular, we will not include the interior of these disks in our figures. As in Figure \ref{fig:2productdisks}, we draw $\alpha_j^{\pm}$ on $F \times \left\{\frac{1}{2}\right\}$, not in $(X_F, K \times I)$.

\section{Foundations for Theorem \ref{thm:main}} \label{section:example}

This section provides the structure of proof of Theorem \ref{thm:main} and a series of important lemmas towards that end. We establish notation for constructing and analyzing branched surfaces in exteriors of positive 3-braid closures. The proof of Theorem \ref{thm:main}, in Section \ref{section:3braids}, requires analysis of 3 cases; we carry out the example of $P(-2,3,7)$ here alongside our preparatory material as motivation. This example already contains the richness of the several cases required to prove Theorem \ref{thm:main}.

We outline the construction of taut foliations in $S^3_r(K)$, $K$ realized as the closure of a positive 3-braid, $r \in (-\infty, 2g(K)-1)$: 

\begin{tabular}{ll}
\textit{Section \ref{subsection:step1}:} & Identify $c_1 + c_2-2$ disjoint product disks $\{D_j\}$ in $X_F$ \\
\textit{Section \ref{subsection:step2}:} & Isotope $\{D_j\}$ into a standardized position in $X_K$ \\
\textit{Section \ref{subsection:step3}:} & Build the spine of the branched surface in $X_K$ from a copy of the fiber surface $F$ \\
& and these standardized disks\\
\textit{Section \ref{subsection:step4}:} & Build the laminar branched surface $B$: \\
& \textit{Section \ref{subsection:step4a}:} Assign optimal co-orientations for the standardized $\{D_j\}$ \\
& \textit{Section \ref{subsection:step4b}:} Check $B$ is sink disk free \\
& \textit{Section \ref{subsection:step4c}:} Prove $B$ is a laminar branched surface \\
\textit{Section \ref{subsection:step5}:} & Construct taut foliations in $X_K$: \\
& \textit{Section \ref{subsection:step5a}:}  Show the boundary train track $\tau$ carries all slopes $(-\infty, 2g(K)-1)$ \\
& \textit{Section \ref{subsection:step5b}:}  Extend essential laminations to taut foliations in $X_K$ \\
& \textit{Section \ref{subsection:step5c}:} Produce taut foliations in $S^3_r(K)$ via Dehn filling
\end{tabular}

To begin our motivational example, we note that $P(-2,3,7)$ is the closure of a positive 3-braid. In particular, $P(-2,3,7) = \hat{\beta}$, for $\beta = \sigma_1^7\sigma_2^2 \sigma_1^2 \sigma_2$. 

\begin{figure}[h!]\center
\labellist
\pinlabel {$q$} at 68 72
\pinlabel {$q$} at 162 72
\pinlabel {$q$} at 255 86
\pinlabel {$q$} at 326 96
\endlabellist
\includegraphics[scale=1]{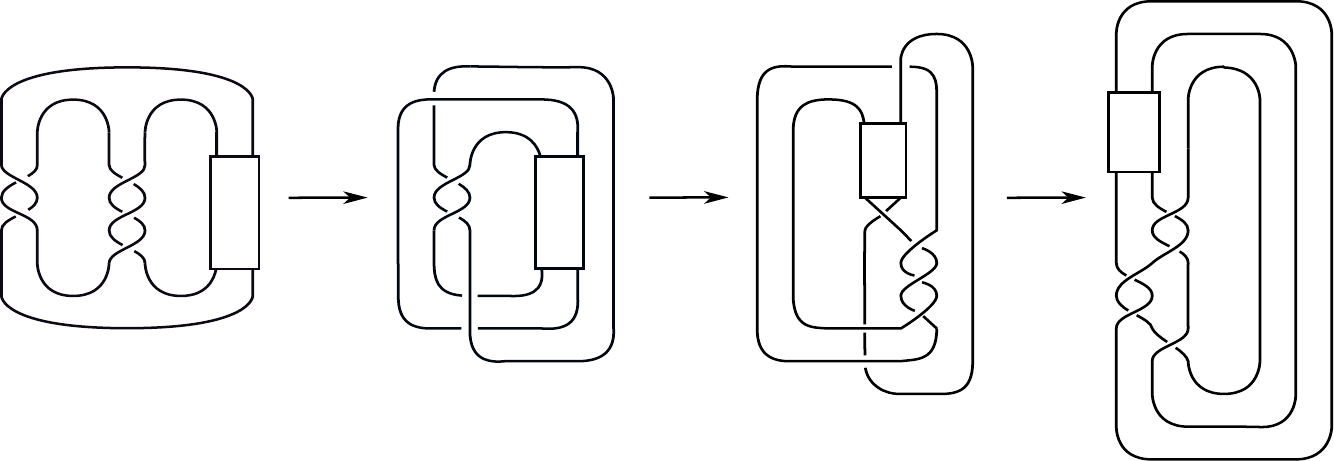}
\caption{An isotopy of $P(-2, 3, q)$, $q$ odd, $q\geq 1$ into the positive closed 3-braid $\hat{\beta}$, for $\beta = \sigma_1^q \sigma_2^2 \sigma_1^2 \sigma_2$.} 
\label{fig:pretzelknotisotopy}
\end{figure}
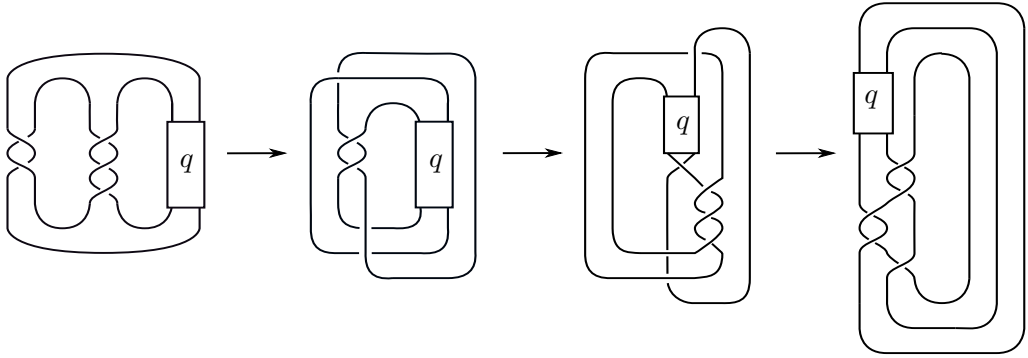

\subsection{Identify disjoint product disks $\{D_j\}$ in $X_F$.} \label{subsection:step1}

The setup in Section \ref{FiberSurface3Braids} supplies $c_1 + c_2 - 2$ product disks: take the product disks used to show $F$ is a fiber surface for $K$. 

Figure \ref{fig:fiber_surface_example} \ shows the fiber surface for $P(-2,3,7)$, and 10 product disks $\{D_1, \ldots D_{10}\}$. The disks $\{D_1, D_2, \ldots, D_7, D_{10}\}$ have boundaries contained in $\mathbbm{b}_{1} \cup \ldots \cup \mathbbm{b}_{7} \cup \mathbbm{b}_{10} \cup \mathbbm{b}_{11} \cup S_1 \cup S_2$; the disks $\{ D_8, D_9 \}$ have boundaries contained in $\mathbbm{b}_{8} \cup \mathbbm{b}_9 \cup \mathbbm{b}_{12} \cup S_2 \cup S_3$. The product disks $D_1, \ldots, D_{10}$ are disjoint in $X_F$, as $\alpha_1^-, \ldots, \alpha_{10}^-$ are pairwise disjoint.  

\subsection{Isotope $\{D_j\}$ into a standardized position in $X_K$} \label{subsection:step2}

The $c_1 + c_2 -2$ product disks found in Section \ref{subsection:step1} are contained in the surface exterior $X_F \approx \overline{X_K - (F \times [\frac{1}{4}, \frac{3}{4}])}$. Collapsing $F \times \left[ \frac{1}{4}, \frac{3}{4}\right]$ to $F \times \left\{\frac{1}{2}\right\}$ produces $c_1 + c_2 -2$ disks in $X_K$, with $\partial D_j \subset \left(F \times \{1/2\} \right) \cup \partial X_K$. 

Consider $(F \times \{\frac{1}{2}\} ) \cup (D_1 \cup \ldots \cup D_{c_1+c_2-2} )$ in $X_K$. This is the spine for a branched surface in $X_K$. For all $j \neq \ell$ and fixed $\star \in \{+, -\}$, the arcs $\alpha_{j}^{\star}$ and $\alpha_{\ell}^{\star}$ are disjoint on the fiber surface $F \times \{\frac{1}{2}\}$. However, for $j \neq \ell$, it is possible for $\alpha_{j}^{+}$ and $\alpha_{\ell}^{-}$ to intersect on $F \times \{\frac{1}{2}\}$; after smoothing, there will be many triple points, as in Figure \ref{fig:branched_surface}.

We want to simplify the forthcoming branched surface. To this end, we isotope the product disks $D_1, \ldots, D_{c_1+c_2-2}$ in $X_K$ such that the arcs $\{\alpha_j^{\pm}\}$ intersect minimally on $F \times \{\frac{1}{2}\}$. 

There are two types of intersection points between $\alpha_j^{+}$ and $\alpha_{\ell}^{-}$, $j \neq \ell$:
\begin{defn} 
A \textbf{Type 1 intersection point} arises from $\alpha_{j}^{+} \cap \alpha_{j+1}^{-}$, where $\mathbbm{b}_j$ and $\mathbbm{b}_{j+1}$ are bands of the same type.
A \textbf{Type 2 intersection point} arises from $\alpha_{j}^{+} \cap \alpha_{\ell}^{-}$, where $\mathbbm{b}_{j}$ and $\mathbbm{b}_\ell$ are bands associated to the last occurrences of $\sigma_1$ and $\sigma_2$ in the same block $\sigma_1^{a_i} \sigma_2^{b_i}$. 
\end{defn}

In Figure \ref{fig:fiber_surface_example}, we see nine triple points in the spine of $P(-2,3,7)$: there are eight Type 1 intersection points, and a single Type 2 intersection point. Lemma \ref{lemma:isotopy} will eliminate all Type 1 intersection points.

\begin{defn} 
Let $D_j$ be a product disk in the spine of a branched surface. A \textbf{spinal isotopy} $\iota_j: D_j \times [0,1] \to X_K$ is an isotopy of the disk $D_j$ in $X_K$ such that for all $t \in [0,1]$, 
\begin{itemize}
\item $\iota_j|_{\alpha_j^- \times \{t\}} = \mathbbm{1}$
\item $\iota_j(\alpha_j^+ \times \{t\}) \subset (F \times \{\frac{1}{2}\})^+$
\item $(\partial D \cap \partial X_K) \subset \partial X_K$
\item $\mathring{D} \subset X_K - (F \times \{\frac{1}{2}\})$
\end{itemize}
and $\iota_j(\alpha_j^+ \times \{1\}) \subset S_i$, where $i = 2,3$.
\end{defn}
Intuitively, allowing $\alpha_j^+$ to move freely along $F \times \{\frac{1}{2}\}$ guides an isotopy of $D_j$ in $X_K$.

\begin{defn}
An arc $\alpha_j^{+}$ is in \textbf{standard position} if it has been isotoped to lie entirely in a single Seifert disk $S_i, i = 2, 3$. A disk is in \textbf{standard position} if both $\alpha_j^+$ and $\alpha_j^{-}$ lie entirely in $S_1 \cup S_2 \cup S_3$.
\end{defn}

\begin{lemma} \label{lemma:isotopy}
There exists a sequence of $c_1+c_2-2$ spinal isotopies of the disks $D_1, \ldots, D_{c_1+c_2-2}$ putting all disks in standard position. Equivalently, there exists a splitting of the spine of the branched surface with no Type 1 intersection points, i.e. with $\alpha_1^+, \ldots, \alpha_{c_1+c_2-2}^+$ in standard position. 
\end{lemma}

\newpage
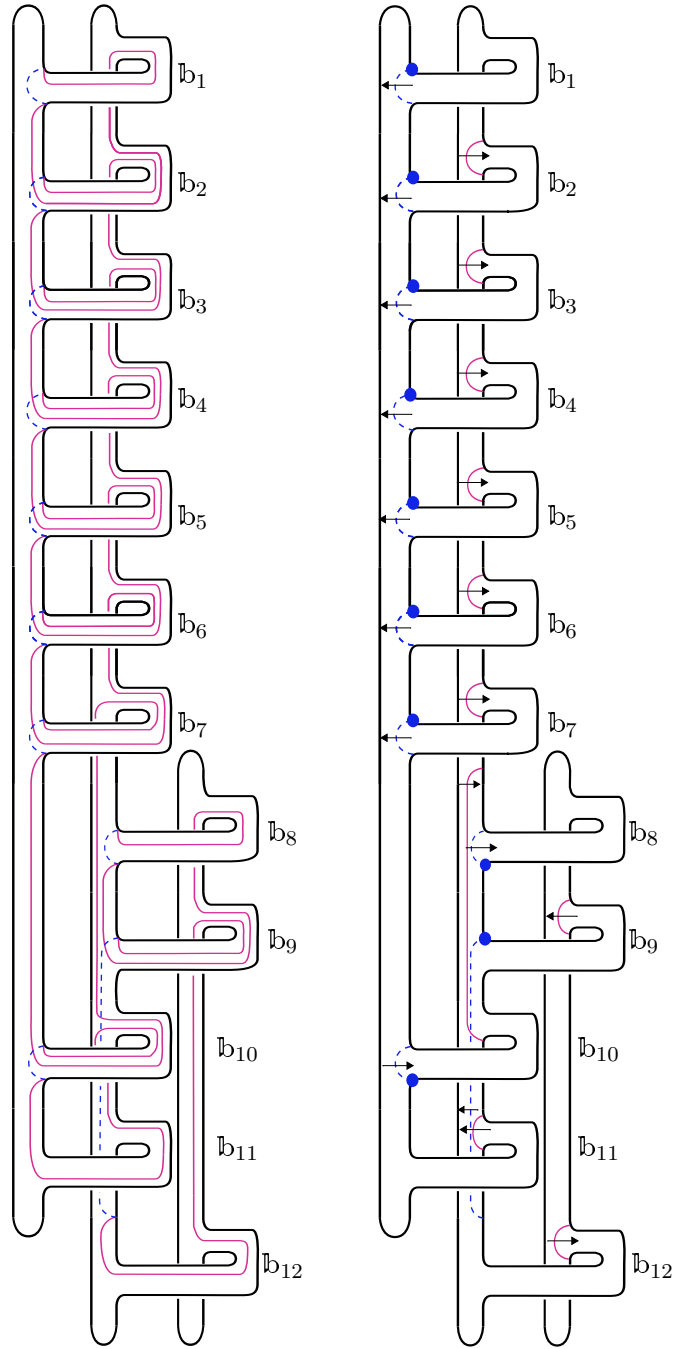
\begin{figure}[h!]\center
\labellist
\pinlabel {$\mathbbm{b}_{1}$} at 200 1410
\pinlabel {$\mathbbm{b}_{2}$} at 200 1290
\pinlabel {$\mathbbm{b}_{3}$} at 200 1160
\pinlabel {$\mathbbm{b}_{4}$} at 200 1050
\pinlabel {$\mathbbm{b}_{5}$} at 200 920
\pinlabel {$\mathbbm{b}_{6}$} at 200 800
\pinlabel {$\mathbbm{b}_{7}$} at 200 690
\pinlabel {$\mathbbm{b}_{8}$} at 300 570
\pinlabel {$\mathbbm{b}_{9}$} at 300 450
\pinlabel {$\mathbbm{b}_{10}$} at 250 330
\pinlabel {$\mathbbm{b}_{11}$} at 250 220
\pinlabel {$\mathbbm{b}_{12}$} at 300 90 
\pinlabel {$\mathbbm{b}_{1}$} at 610 1410
\pinlabel {$\mathbbm{b}_{2}$} at 610 1290
\pinlabel {$\mathbbm{b}_{3}$} at 610 1160
\pinlabel {$\mathbbm{b}_{4}$} at 610 1050
\pinlabel {$\mathbbm{b}_{5}$} at 610 920
\pinlabel {$\mathbbm{b}_{6}$} at 610 800
\pinlabel {$\mathbbm{b}_{7}$} at 610 690
\pinlabel {$\mathbbm{b}_{8}$} at 700 570
\pinlabel {$\mathbbm{b}_{9}$} at 700 450
\pinlabel {$\mathbbm{b}_{10}$} at 650 330
\pinlabel {$\mathbbm{b}_{11}$} at 650 220
\pinlabel {$\mathbbm{b}_{12}$} at 710 90 
\endlabellist
\includegraphics[scale=.349]{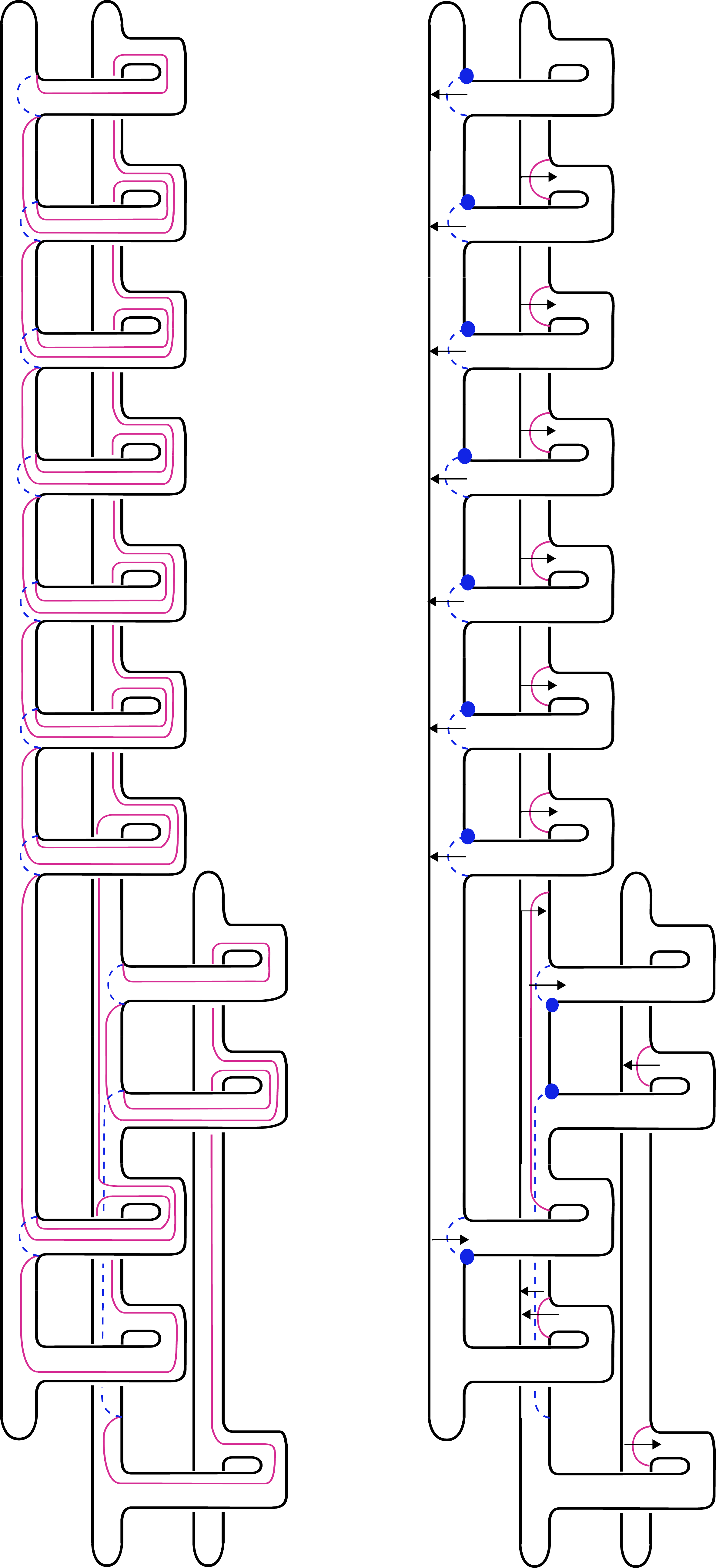}
\caption{On the left: the fiber surface and 10 product disks for $P(-2,3,7)$. On the right: the laminar branched surface for $P(-2,3,7)$ with cusping directions $(\leftarrow)^7(\rightarrow)(\leftarrow)(\rightarrow)(\ \ ) (\ \ )$.} 
\label{fig:fiber_surface_example}
\end{figure}
\newpage

\begin{proof} 
Scanning the diagram of $F \times \{\frac{1}{2}\}$ from bottom to top, find the first arc $\alpha_{s}^{+}$ encountered. The last letter of $\beta$ is $\sigma_2$, so $\alpha_s^{+} \subset \mathbbm{b}_s \cup \mathbbm{b}_{c_1+c_2} \cup S_2 \cup S_3$, with $s < c_1 + c_2$. If we allow \textit{free} isotopy of arcs in $F \times \frac{1}{2}$ (i.e. an isotopy $i_s$ of $\alpha_s^+$ where the endpoints of the arc can move along $\partial F$), $\alpha_s^{+}$ can be isotoped to lie entirely in $S_3$. Let $\iota_s$ be the spinal isotopy of $D_s$ in $X_K$ such that for all $t$, $\iota_s (\alpha_s^+ \times \{t\}) = i_s(\alpha_s^+ \times \{t\})$. Applying $\iota_s$ puts $D_s$ in standard position.

Continue scanning the diagram from bottom to top, and find the next arc $\alpha_r^{+}$ encountered. Apply the spinal isotopy $\iota_r$ of $D_r$ in $X_K$ such that $\iota_r|_{\alpha_{r}^{+} \times \{t\}}$ pushes $\alpha_r^{+}$ into standard position. After $c_1+c_2-2$ iterations of this procedure (finding the next arc $\alpha_m^{+}$ encountered, and putting the disk $D_m$ in standard position via $\iota_m$), all disks are standardized. A Type 1 intersection between $\alpha_t^{+}$ and $\alpha_{t+1}^{-}$ is eliminated by the isotopy $\iota_t$ standardizing $D_t$. 
\end{proof}

\begin{rmk}
The pre- and post- split spine have isotopic exteriors. 
\end{rmk}

For $P(-2,3,7)$, the arcs get isotoped in the following order: $$\alpha_9^{+}, \alpha_{10}^{+}, \alpha_7^{+}, \alpha_8^{+}, \alpha_6^{+}, \alpha_5^{+}, \alpha_4^{+}, \alpha_3^{+}, \alpha_2^{+}, \alpha_1^{+}$$
The result of applying Lemma \ref{lemma:isotopy} is seen in the right diagram in Figure \ref{fig:fiber_surface_example}. There is a single Type 2 intersection point between $\alpha_7^{+}$ and $\alpha_9^{-}$.

Going forward, all disks $D_j$ are in standard position, unless stated otherwise. We will \textbf{not} change our notation to indicate the disks are standardized. \\

\subsection{Build the spine of the branched surface} \label{subsection:step3}

The spine for the branched surface is built from $$\left(F \times \{1/2\}\right) \cup \left(\bigcup_{i=1}^{c_1+c_2-2} D_i\right)$$ 

For $P(-2,3,7)$, the spine for the branched surface is in Figure \ref{fig:fiber_surface_example}.

\subsection{Build the branched surface $B$} \label{subsection:step4}

To build the laminar branched surface, we need to assign co-orientations for the disks $D_j, \ 1 \leq j \leq c_1+c_2-2$, and verify these choices do not create sink disks. To achieve these goals, we study the branch locus and branch sectors. 

Lemma \ref{lemma:isotopy} simplified the branch locus: all arcs $\alpha_j^{\pm}$, $1 \leq j \leq c_1+c_2-2$ are now contained in $S_1 \cup S_2 \cup S_3$. Moreover, arcs $\alpha_{j}^{-}$ are isotopic to the co-cores of bands $\mathbbm{b}_{j}$, or would be if other bands were not obstructing the path of the lower endpoint.

For $P(-2,3,7)$,
\begin{itemize}
\item the arcs $\alpha_{1}^{-}, \ldots, \alpha_7^-, \alpha_{10}^-$, contained in $S_1$, are isotopic to the co-cores of the 1-handles $\mathbbm{b}_{1}, \ldots, \mathbbm{b}_{7}, \mathbbm{b}_{10}$ respectively. 
\item the arc $\alpha_8^{-}$ is isotopic to the co-core of $\mathbbm{b}_8$.
\item the arcs $\alpha_{1}^{+}, \ldots, \alpha_{6}^{+}, \alpha_{10}^{+}$ are isotopic to the co-cores of the 1-handles $\mathbbm{b}_{2}, \ldots \mathbbm{b}_7, \mathbbm{b}_{11}$, respectively, and are contained in $S_2$.
\item the $\alpha_{8}^{+}$ is isotopic to the co-core of $\mathbbm{b}_{9}$, and is contained in $S_3$.
\item the two arcs $\alpha_9^{-}$ and $\alpha_7^+$ are not isotopic to the co-cores of any bands.
\end{itemize}

Cusp directions for the disks have yet to be assigned. Nevertheless, we know the branch sectors for $B$ will fall into two categories: the sectors that lie in $F \times \{\frac{1}{2}\}$, and sectors arising from isotoped product disks. The former can be further refined into 3 categories: 
\begin{defn} \label{defn:typesofbranchsectors}
The \textbf{$S_i$ disk sector} is the connected component of a branch sector containing the Seifert disk $S_i$. A \textbf{band sector} is the connected component of a branch sector associated to a positively twisted band. The remaining branch sectors are \textbf{polygon sectors}; each lies in a single Seifert disk. 
\end{defn}

In particular, all polygon sectors lie in $S_2$. For $P(-2,3,7)$, there are 7 band sectors (the branch sectors containing $\mathbbm{b}_{2}, \ldots, \mathbbm{b}_{7} \cup \mathbbm{b}_9 \cup \mathbbm{b}_{10}$), and a pair of polygon sectors.

\subsubsection{\textbf{\textup{Assign optimal co-orientations to $\{D_j\}$}}} \label{subsection:step4a}

\begin{defn}
Let $\widehat{\alpha}_j^\star$ denote the cusp direction of $\alpha_j^\star$, for $\star \in \{ +, - \}$. 
\end{defn}

\begin{lemma} \label{lemma:cuspdirections}
Assigning a co-orientation to $D_j$ determines the cusp orientation to both $\alpha_{j}^{+}$ and $\alpha_j^{-}$. 
Moreover, if we orient the arcs $\alpha_j^{\pm}$ from the lower endpoint to the upper endpoint, the pairings $\langle \alpha_j^+, \widehat{\alpha}_j^+ \rangle$ and $\langle \alpha_j^-, \widehat{\alpha}_j^- \rangle$ have opposite signs. 
\end{lemma}
Heuristically: the induced cusp orientations of $\alpha_j^{+}$ and $\alpha_j^{-}$ ``point in opposite directions" when looking at $(F \times \{\frac{1}{2}\})^+$.

\begin{figure}[h!]\center
\labellist
\pinlabel {$F \times \frac{1}{2}$} at 400 130
\pinlabel {$F \times \frac{1}{2}$} at 155 130
\pinlabel {$D^2$} at 310 10 
\pinlabel {$D^2$} at 65 10 
\endlabellist
\includegraphics[scale=1]{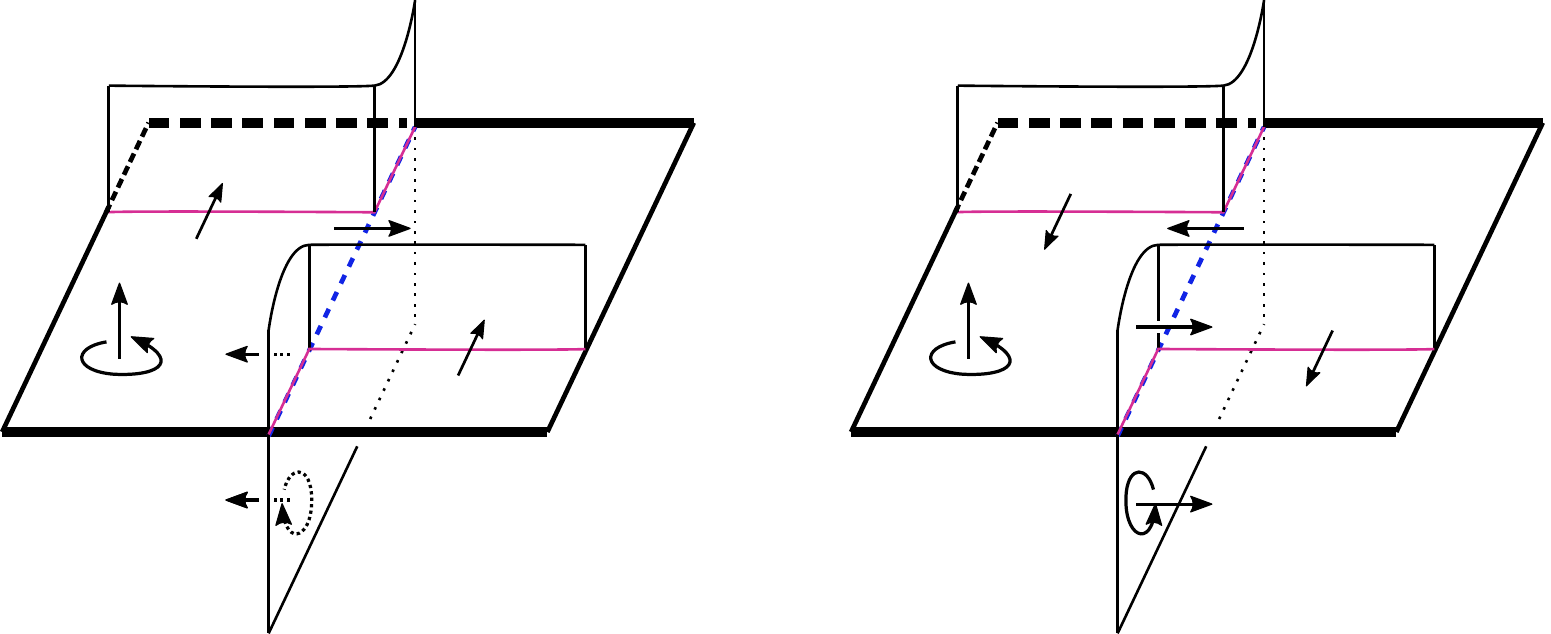}
\caption{In this local model, we have fixed a co-orientation on $F \times \{\frac{1}{2}\}$, and chosen different co-orientations on $D_j$ in the left and right figures. The correct cusping choices for $\alpha_j^\pm$ are provided. The \textbf{bolded} horizontal lines lie on $\partial X_K$.} 
\label{fig:cusping_a_product_disk}
\end{figure}
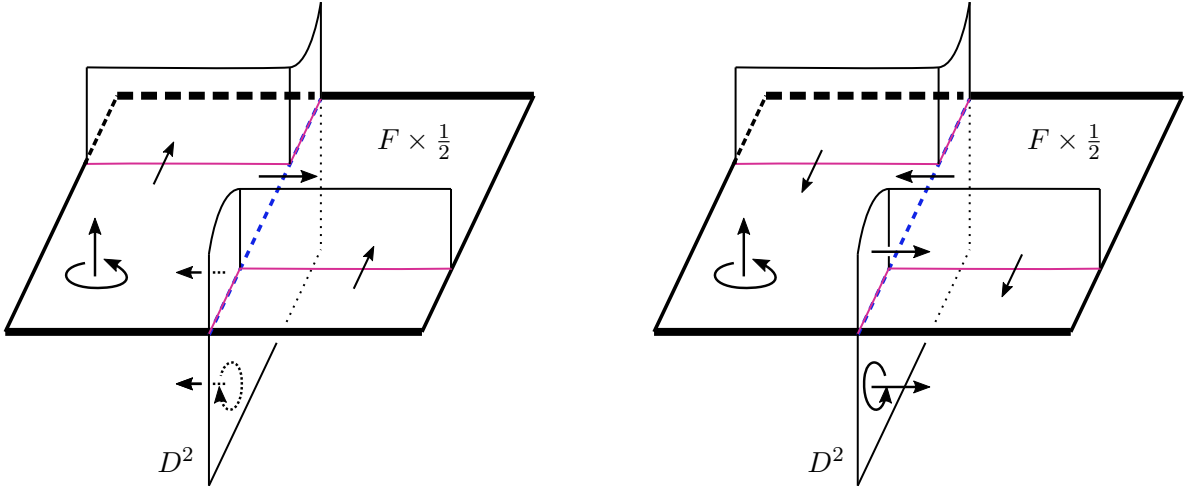

\begin{proof} 
For simplicity, assume the disk has yet to be standardized. Choose a co-orientation on $D_j$. Since $F$ is co-oriented, the correct smoothing choices for $\alpha_j^+$ and $\alpha_j^{-}$ ensure the co-orientations of $F$ and $D_j$ agree near the branch locus. The corresponding cusp directions for $\alpha_j^{\pm}$ can be determined immediately, as in the local model in Figure \ref{fig:cusping_a_product_disk}: if the cusp direction on $\alpha_j^{-}$ points to the right (resp. left) near $\partial X_K$, then the cusp direction on $\alpha_j^{+}$ points to the left (resp. right) near $\partial X_K$. Taking a global viewpoint as in Figure \ref{fig:global_cusping_product_disk}, orient the arcs $\alpha_j^\pm$ from the lower endpoint to the upper endpoint: the pairings $\langle \alpha_j^{\pm}, \widehat{\alpha}_j^{\pm} \rangle$ have opposite signs, and the cusp directions point in opposite directions when looking at $(F \times \{\frac{1}{2}\})^{+}$. Our isotopy $\iota_t$ of $D_t$ preserves the relative positions of the upper and lower endpoints of $\alpha_t^{+}$, so the lemma holds for standardized disks.
\end{proof}

\begin{rmk} 
In addition to establishing conventions about cusp directions, Figure \ref{fig:global_cusping_product_disk} indicates special (i.e. bolded) endpoints. The meaning of the bolding is postponed until Definition \ref{defn:endpoints} (it will become relevant when computing the slopes carried by the branched surface).
\end{rmk}

\begin{figure}[h!]\center
\labellist
\pinlabel {$\alpha_j^-$} at -40 400
\pinlabel {$\alpha_j^-$} at 600 400
\pinlabel {$\mathbbm{b}_j$} at 450 450
\pinlabel {$\mathbbm{b}_j$} at 1100 450
\pinlabel {$\mathbbm{b}_{j+1}$} at 470 150
\pinlabel {$\mathbbm{b}_{j+1}$} at 1120 150
\endlabellist
\includegraphics[scale=.25]{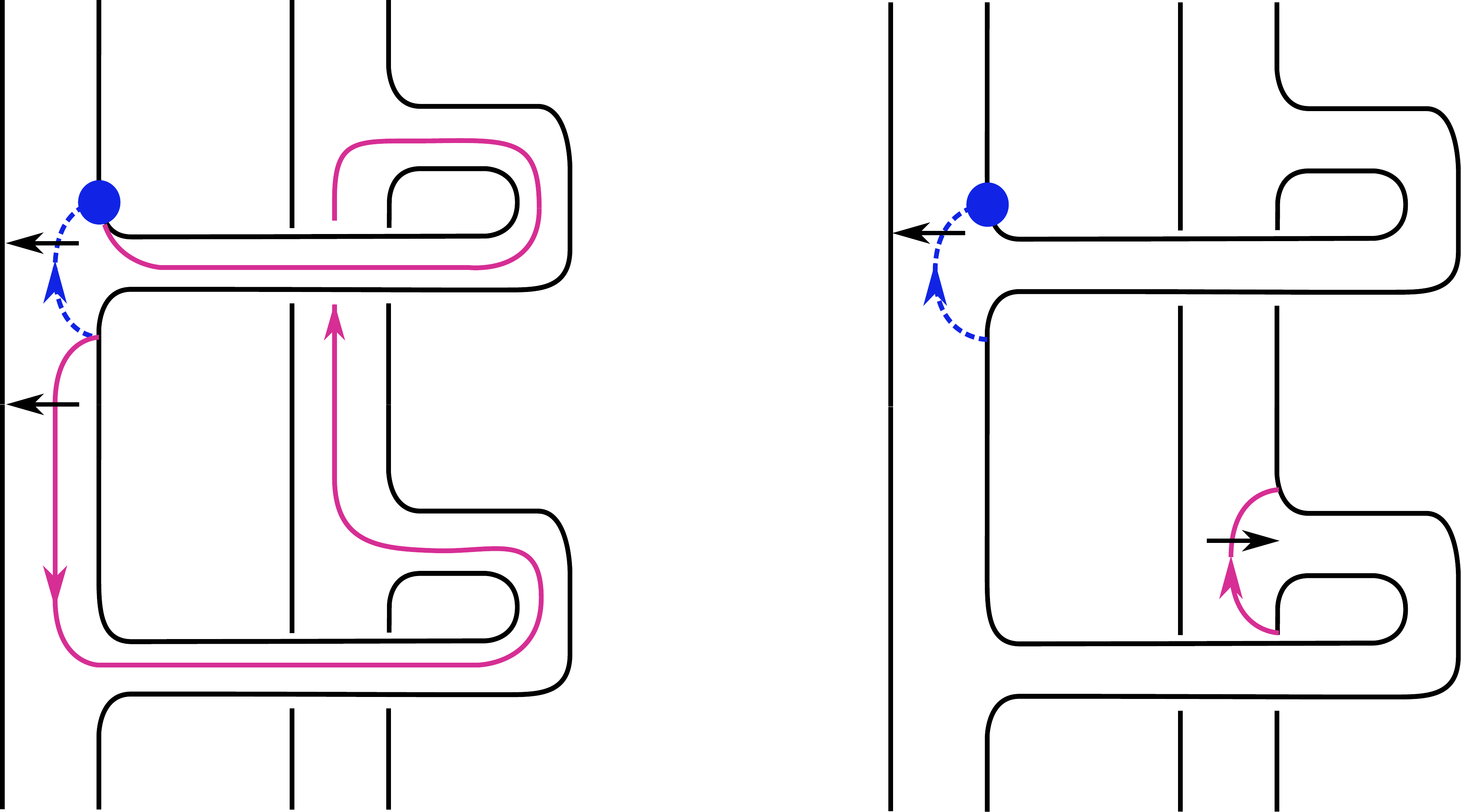}
\caption{After standardizing, $\widehat{\alpha}_j^-$ and $\widehat{\alpha}_j^+$ ``point in opposite directions".} 
\label{fig:global_cusping_product_disk}
\end{figure}
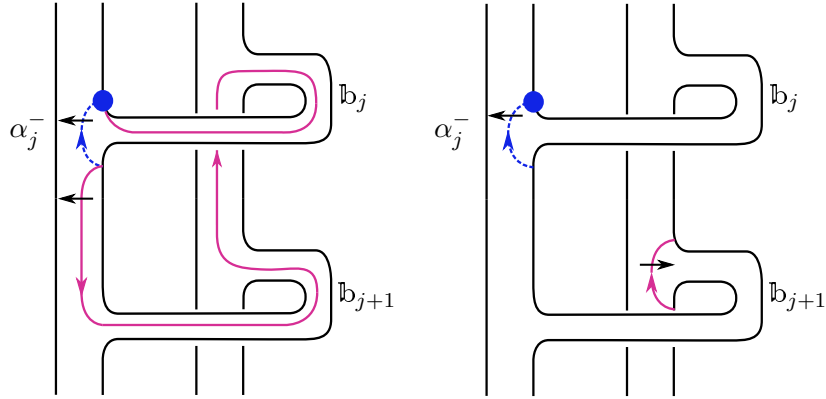

The cusp direction of $\alpha_j^{-}$ determines the co-orientation of $D_j$. Moreover, the upper endpoint of $\alpha_j^-$ is planted above the attachment site of the 1-handle $\mathbbm{b}_j$, which in turn is associated to the $j^{\text{th}}$ letter $\sigma_i$ of $\beta$. Therefore, we can encode the co-orientation of $D_j$ directly to $\mathbbm{b}_j$, via the induced cusp orientation on $\alpha_j^{-}$.

\begin{defn} \label{defn:cuspbraidword} We encode the co-orientation of $D_j$ by recording the cusp direction of $\alpha_i^{-}$ in tandem with $\beta$. For $\sigma$ the $j^{\text{th}}$ letter of $\beta$:
\begin{itemize}
\item Writing $\leftarrow$ below $\sigma$ indicates $\langle \widehat{\alpha}_j^-, \alpha_j^- \rangle = 1$ and $\langle \widehat{\alpha}_j^+, \alpha_j^+ \rangle = -1$. That is, $\alpha_j^-$ is cusped ``to the left", and $\alpha_{j}^{+}$ is cusped ``to the right" when looking at $(F \times \{\frac{1}{2}\})^+$.
\item Writing $\rightarrow$ below $\sigma$ indicates $\langle \widehat{\alpha}_j^-, \alpha_j^- \rangle = -1$ and $\langle \widehat{\alpha}_j^+, \alpha_j^+ \rangle = 1$. That is, $\alpha_j^-$ is cusped ``to the right", and $\alpha_{j}^{+}$ is cusped ``to the left" when looking at $(F \times \{\frac{1}{2}\})^+$
\item Writing $( \ \ )$ below $\sigma$ indicates \textbf{not} choosing the product disk $D_j$ with pre-standardized arc $\alpha_{j}^{+}$ passing through this 1-handle. We say $\sigma$ is \textbf{uncusped}.
\end{itemize}
\end{defn}

$P(-2,3,7)$ is realized as the closure of $\beta = \sigma_1^7 \sigma_2^2 \sigma_1^2 \sigma_2 = \sigma_1^7 \sigma_2 \sigma_2 \sigma_1 \sigma_1 \sigma_2$. The cusping directions in (\ref{cuspdirectionsexample}) below determine a branched surface -- it specifies which product disks to choose when building the spine, and how to co-orient them, as in Figure \ref{fig:fiber_surface_example}.
\begin{align}
&\ \sigma_1^7 \ \ \ \ \sigma_2 \ \ \sigma_2 \ \  \sigma_1 \ \  \sigma_1 \ \  \sigma_2 \nonumber \\
&(\leftarrow)^7 (\rightarrow ) (\leftarrow) (\rightarrow) (\ \ ) (\ \ ) \label{cuspdirectionsexample}
\end{align}

We emphasize: directions, as in (\ref{cuspdirectionsexample}), completely determine a branched surface. In Section \ref{section:buildingbranchedsurfaces}, we assign cusp directions for an arbitrary positive 3-braid closure.

\subsubsection{\textbf{\textup{Check $B$ is sink disk free}}} \label{subsection:step4b}

\begin{lemma} \label{lemma:productdisksinkdisk}
A branch sector arising from an isotoped product disk is never a sink disk.
\end{lemma}

\begin{proof}
Let $D_j$ be any product disk sector. By Lemma \ref{lemma:cuspdirections}, the pairings $\langle \alpha_j^+, \widehat{\alpha}_j^+ \rangle$ and $\langle \alpha_j^-, \widehat{\alpha}_j^- \rangle$ have opposite signs. Therefore, one of $\widehat{\alpha}_j^+$ and $\widehat{\alpha}_j^-$ points out of $(F \times \{ \frac{1}{2}\})^{+}$ and into $D_j$ (and vice-versa for the other). It is impossible for both cusp directions to point into $D_j$.
\end{proof}
In Section \ref{section:3braids}, we develop techniques for determining which cusping directions (as in (\ref{cuspdirectionsexample})) create sink disks. For the branched surface $B$ for $P(-2,3,7)$, we already identified the branch sectors on $F \times \{\frac{1}{2}\}$, so verifying $B$ is sink disk free is straightforward. To show a branch sector is not a half sink disk, we need only check some cusped arc $\widehat{\alpha}_j^\star$ points out of it.

\begin{itemize}
\item The Disk Sectors
	\begin{itemize}
		\item $S_1$ is not a sink disk, because $\widehat{\alpha}_{10}^{-}$ points out of it.
		\item $S_2$ is not a sink disk, because $\widehat{\alpha}_{1}^{+}$ points out of it.
		\item $S_3$ is not a sink disk, because $\widehat{\alpha}_{9}^{+}$ points out of it.
	\end{itemize}
\item The Band Sectors
	\begin{itemize}
		\item The sectors $\mathbbm{b}_{2}, \ldots, \mathbbm{b}_{7}$ have $\widehat{\alpha}_{2}^{-}, \ldots, \widehat{\alpha}_{7}^{-}$ pointing out of the respective regions.
		\item The band sector containing $\mathbbm{b}_{9} \cup \mathbbm{b}_{10}$ in the boundary has $\widehat{\alpha}_{9}^{-}$ pointing out of it.
	\end{itemize}
\item The Polygon Sectors
	\begin{itemize}
		\item The boundary of the \textbf{upper polygon sector} $P_u$ is contained in $\alpha_{7}^{+} \cup \alpha_{8}^{-} \cup \alpha_{9}^{-} \cup \ \partial F$; $\widehat{\alpha}_{8}^{-}$ points out of the sector.
		\item The boundary of the \textbf{lower polygon sector} $P_\ell$ is contained in $\alpha_{7}^{+} \cup \alpha_{9}^{-} \cup \alpha_{10}^{+} \cup \ \partial F$; $\widehat{\alpha}_{9}^{-}$ points out of the sector.	\end{itemize}
\end{itemize}

\subsubsection{\textbf{\textup{$B$ is a laminar branched surface}}} \label{subsection:step4c}

\begin{prop} \label{prop:laminarbranchedsurface}
A sink disk free branched surface $B$, constructed from a copy of the fiber surface and a collection of product disks, is a laminar branched surface.
\end{prop}
\vspace{-.5cm}
\begin{proof} \phantom{\qedhere}
We verify $B$ is laminar by verifying conditions (1) -- (4) of Theorem \ref{thm:taolisinkdisk} hold. Note that the $M$ of Theorem \ref{thm:taolisinkdisk} is $X_K$.

\begin{enumerate}
\item[(1a)] \textbf{$\partial_h(N(B))$ is incompressible and $\partial$-incompressible in $M - \text{int}(N(B))$.}

A sutured manifold $(M, \gamma)$ is \textbf{taut} if $M$ is irreducible and $R(\gamma)$ is norm minimizing in $H_2(M, \gamma)$ \cite{Gabai:FoliationsI}. Each of our product disks appears in a sutured manifold decomposition of $(X_F, K \times I)$ which terminates in $(D^2, S^1 \times I)$. Thus, any sutured manifold appearing in the sequence of product disk decompositions of $(X_F, K \times I)$ is a taut sutured manifold \cite{Gabai:FoliationsI}. In particular, the exterior of the pre-split spine (built from $c_1+c_2-2$ co-oriented product disks), $(M', \gamma_M')$, is a taut product sutured manifold, and $R(\gamma_M')$ is norm minimizing.  

The exterior of the post-split spine also has a product sutured manifold structure; denote this manifold $(N', \gamma_N')$. For $B$ the branched surface whose spine has standardized disks, we have $\gamma_N' \approx \partial_v(N(B)) \cup (\partial X_K -  int (N(B))|_{\partial X_K})$ and $R(\gamma_N')$ is isotopic to $R(\gamma_M') \approx \partial_h(N(B))$. Thus $\partial_h N(B)$ is norm minimizing in $H_2(N',\gamma_N')$, and $\partial_h(N(B))$ is incompressible and $\partial$-incompressible in $M-int(N(B))$. \\

\item[(1b)] \textbf{There is no monogon in $M - \text{int}(N(B))$.} \\
This follows from our construction; the branched surface has a transverse orientation.\\

\item[(1c)] \textbf{No component of $\partial_h N(B)$ is a sphere or a disk properly embedded in $M$.} \\
Every component of $\partial_h N(B) $ meets $\partial X_K$, so no component of $\partial_h(N(B))$ can be a sphere. The horizontal boundary $\partial_h N(B)$ is properly embedded in $X_B$, not $X_K$. \\

\item[(2)] \textbf{$M - \text{int}(N(B))$ is irreducible and $\partial M - \text{int}(N(B))$ is incompressible in $M - \text{int}(N(B))$.} \\
$M - int(N(B))$ is a submanifold of $S^3$ with connected boundary, thus is irreducible. $\partial X_K - int(N(B))$ is a torus with a neighborhood of a train track removed: it is a collection of bigons. In particular, any simple closed curve in $\partial X_K - int(N(B))$ bounds a disk in $\partial X_K - int(N(B))$, and is incompressible in $M - int(N(B))$. \\

\item[(3)] \textbf{$B$ contains no Reeb branched surface (see \cite{GabaiOertel} for more details).} \\
To prove $B$ does not contain a Reeb branched surface, it suffices to show that $B$ cannot carry a torus or fully carry an annulus. 

By construction, every sector of $B$ meets $\partial X_K$. Thus, any compact surface carried by $B$ must also meet $\partial X_K$. Thus, $B$ cannot carry a torus.

We now prove $B$ cannot fully carry an annulus. Suppose, by way of contradiction, that $B$ fully carries a compact surface $S$. Any such $S$ is built as a union of branch sectors, where each branch sector has a positive weight. Since $\beta$ is in the form specified by Section \ref{FiberSurface3Braids}, we can restrict our attention to the first two letters of the braid word, namely $\sigma_1 \sigma_1$; see Figure \ref{fig:compact_surfaces}. We assign weights to the relevant branch sectors:
\begin{itemize}
\item the disk sectors $D_1$ and $D_2$ have weights $w_1$ and $w_2$, respectively
\item the band sector $\mathbbm{b}_2$ has weight $w_3$,
\item the two (isotoped) product disks associated to $\sigma_1^2$ have weights $w_4$ and $w_5$.
\end{itemize}

See Figure \ref{fig:compact_surfaces}. If $B$ carries a compact surface, the switch relations induced by the branch loci induce the following: $w_1 = w_2 + w_4, \ w_3 = w_2 + w_4,$ and $w_1 = w_5 + w_3$.

This implies that $w_1 = w_3 = w_3 + w_5$, thus $w_5=0$. This contradicts that $S$ is fully carried by $B$. We conclude that $B$ cannot carry any compact surface, and therefore does not carry an annulus. \\

\item[(4)] \textbf{$B$ is sink disk free.} \\
This holds by assumption.\hfill $\Box$
\end{enumerate} 
\end{proof}

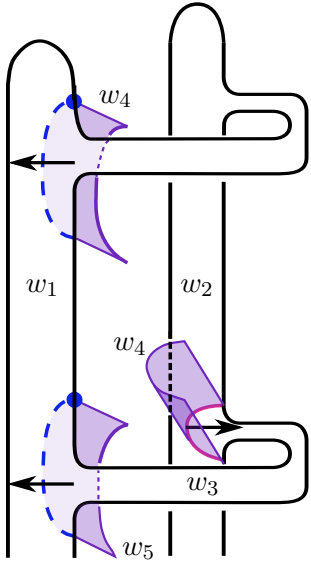
\begin{wrapfigure}{L}{.4\linewidth}
\labellist
\pinlabel {$w_1$} at 35 100
\pinlabel {$w_2$} at 88 100
\pinlabel {$w_3$} at 90 32
\pinlabel {$w_4$} at 60 165
\pinlabel {$w_4$} at 65 80
\pinlabel {$w_5$} at 68 10
\endlabellist
\begin{minipage}[c][.4\paperheight]{60mm}
\begin{center}
\includegraphics[scale=1.1]{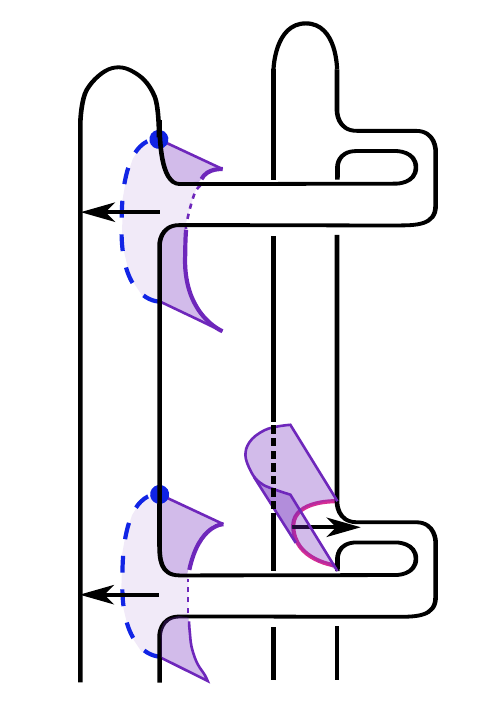}
\captionof{figure}{A local picture of the branched surface near the bands $\mathbbm{b}_1$ and $\mathbbm{b}_2$. For simplicity, the product disk associated to $\alpha_1$ appears broken in our figure; it has weight $w_4$. The standard relations near the branch locus indicate that $w_5=0$, thus $B$ cannot fully carry an annulus.}
\label{fig:compact_surfaces}
\end{center}
\end{minipage}
\end{wrapfigure}

\subsection{Construct taut foliations in $X_K$} \label{subsection:step5}
$B$ is a laminar branched surface. Theorem \ref{thm:taolisinkdisk} guarantees that for every rational slope $r$ carried by the boundary train track $\tau$, there exists an essential lamination $\mathcal{L}_r$ meeting $\partial X_K$ in simple closed curves of slope $r$. To construct taut foliations in $X_K$, we first understand which slopes are carried by $\tau$, apply Theorem \ref{thm:taolisinkdisk} to get a family of essential laminations, and then extend each lamination to a taut foliation in $X_K$.

\subsubsection{\textbf{\textup{Show the train track $\tau$ carries all rational slopes $r < 2g(K)-1$}}} \label{subsection:step5a} 
Since $B$ is formed by $(F \times \{\frac{1}{2}\}) \cup D_1 \cup \ldots \cup D_{c_1+c_2-2}$, the boundary train track $\tau$ carries slope $0$.

\begin{defn}
Each $D_j$ meets $\partial X_K$ in two arcs, each tracing out the path of an endpoint of $\alpha_j^-$ under $\varphi$. These arcs are \textbf{sectors of the train track $\tau$}; $\overline{\tau - \lambda}$ is a collection of sectors.
\end{defn}

We have $c_1+c_2-2$ disks, and therefore $2 \cdot (c_1+c_2-2)$ sectors in the associated train track $\tau$. Consider $\alpha_j^-$ with cusping $\widehat{\alpha}_j^-$. The cusping $\widehat{\alpha}_j^-$ will agree with the orientation of $\lambda$ at one endpoint of $\alpha_j^-$, and disagree at the other endpoint. Thus, for $s_j$ and $s_j'$ the pair of sectors induced by $\alpha_j^-$, the train tracks $\lambda \cup s_j$ and $\lambda \cup s_j'$ carry different slopes, as in Figure \ref{fig:train_track_setup}: $\lambda \cup \text{ (the leftmost sector)}$ carries $[0,1)$, while $\lambda \cup \text{ (the middle sector)}$ carries $(-\infty, 0]$.

\begin{defn}  \label{defn:endpoints}
If the direction of $\widehat{\alpha}_{j}^{-}$ disagrees with the orientation of $\lambda$ at a given endpoint of $\alpha_j^-$, we say this endpoint \textbf{contributes maximally to $\tau$}.
\end{defn}

In our figures, the endpoint of $\alpha_j^{-}$ contributing maximally is \textbf{bolded}.

\begin{figure}[h!]\center
\labellist
\pinlabel {$\lambda$} at 435 40
\pinlabel {$\partial(X_K)$} at 480 10
\endlabellist
\includegraphics[scale=.6]{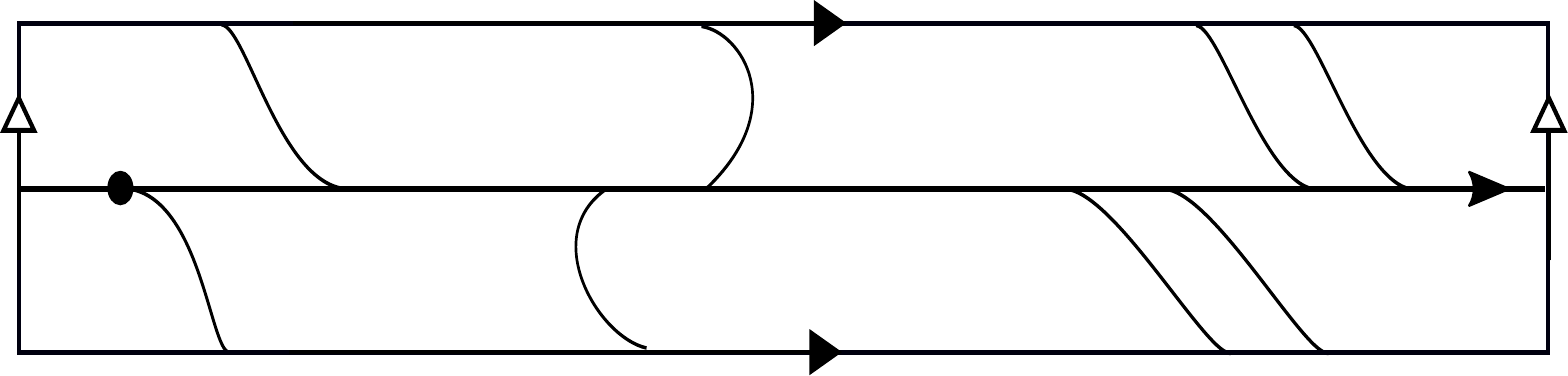}
\caption{A train track $\tau \subset \partial(X_K)$. $\lambda \cup \text{ (the leftmost sector)}$ carries $[0,1)$, while $\lambda \cup \text{ (the middle sector)}$ carries $(-\infty, 0]$. The rightmost sectors are linked.} 
\label{fig:train_track_setup}
\end{figure}

Our goal is to maximize the interval of slopes carried by $\tau$. There are $c_1 + c_2 - 2$ endpoints contributing maximally to $\tau$ -- one for each product disk. It is tempting to claim $\tau$ carries all slopes $[0, c_1 + c_2 - 2)$. However, this is na\"ive: the endpoints of the arcs $\alpha, \alpha'$ could be linked on along $\partial F$, as in the rightmost picture in Figure \ref{fig:train_track_setup}.

\begin{defn} \label{defn:linked}
Let $\alpha_{j}^{-}$ and $\alpha_{\ell}^{-}$ be distinct properly embedded arcs on $F$ such that (1) the first endpoint of each arc contributes maximally to $\tau$ and (2) their endpoints are linked in $\lambda$. Then $\alpha_{j}^{-}$ and $\alpha_{\ell}^{-}$ are called \textbf{linked arcs}. See Figure \ref{fig:train_track_setup}. If $\alpha_j^-$ and $\alpha_\ell^-$ are not linked, they are \textbf{unlinked} or \textbf{not linked}. 
\end{defn}

The train track $\tau$ induced by $B$ will carry all slopes in $(-\infty, k)$, where $k$ is the maximum number of pairwise unlinked arcs contributing maximally to $\tau$. Proving Theorem \ref{thm:main} requires sorting positive 3-braids into three types. For each type, we construct a laminar branched surface $B$ using $c_1+c_2-2$ product disks and a unique pair of linked arcs. Thus, $\tau$ carries all slopes in $[0, (c_1 + c_2 - 2)-1) = [0, 2g(K)-1)$.

\begin{defn} \label{defn: subtraintrack}
A \textbf{sub-train-track} $\tau'$ of $\tau$ is a train track carrying slope $0$, such that $ \{ \text{sectors of } \tau' \} \subseteq  \{ \text{sectors of } \tau \}$. 
\end{defn}

\begin{rmk}
For our purposes, $\tau'$ will include all sectors contributing maximally to $\tau$, and a single sector $s$ with $\lambda \cup s$ carrying $(-\infty, 0]$. 
\end{rmk}

\begin{lemma} \label{lemma:subtraintrackcarries}
Any slope carried by $\tau'$, a sub-train-track of $\tau$, is also carried by $\tau$. \hfill $\Box$
\end{lemma}

For $P(-2,3,7)$, we have $c_1+c_2-2=10$ sectors contributing maximally to $\tau$, and exactly one pair of linked arcs coming from $\alpha_{8}^{-}$ and $\alpha_{9}^{-}$. Let $\tau'$ be the sub-train-track built from the endpoints of $\alpha_1^{-}, \ldots, \alpha_{10}^-$ that contribute maximally to $\tau$. Thus $\tau'$ carries all rational slopes in $[0,9)$. Appending the upper endpoint of $\alpha_8^{-}$ to $\tau'$ ensures $\tau'$ carries all slopes in $(-\infty, 9)$.  Applying Lemma \ref{lemma:subtraintrackcarries}, we conclude $\tau$, the train track induced by $B$, carries slopes in $(-\infty, 9)$.

\subsubsection{\textbf{\textup{Extend essential laminations to taut foliations}}} \label{subsection:step5b}

We now have a laminar branched surface $B$ carrying all rational slopes in $(-\infty, 2g(K)-1)$. By Theorem \ref{thm:taolisinkdisk}, $B$ carries an essential lamination $\mathcal{L}_r$ for every rational $r \in (-\infty, 2g(K)-1)$. We use these laminations to construct taut foliations in $X_K$.

\begin{prop} \label{prop:extendlamination}
Let $\mathcal{L}_r$ be an essential lamination carried by our laminar branched surface $B$, such that $\mathcal{L}_r$ meets $\partial X_K$ in simple closed curves of slope $r$. Then $\mathcal{L}_r$ can be extended to a taut foliation in $X_K$, which foliates $\partial X_K$ in parallel simple closed curves of slope $r$.
\end{prop}
\begin{proof}

In Proposition \ref{prop:laminarbranchedsurface}, we proved the branched surface exterior $X_B \approx \overline{X_K - int(N(B))}$ is isotopic to a product sutured manifold. In particular, $X_B$ has an $I$-bundle structure. $N(B)$ is an $I$-bundle over $B$, thus $\overline{N(B) - \mathcal{L}_r}$ has an $I$-bundle structure. Endowing the lamination exterior $X_{\mathcal{L}_r} \approx \overline{X_K - \mathcal{L}_r}$ with an $I$-bundle structure yields a foliation $\mathcal{F}_r$ for $X_K$ which is induced by $\mathcal{L}_r$.

$\mathcal{L}_r$ meets $\partial X_K$ in simple closed curves of slope $r$, so $\displaystyle \overline{X_{\mathcal{L}_r}}|_{\partial X_K}$ is an $r$-sloped annulus $A_r$. $A_r$ is formed from $X_B|_{\partial X_K}$ and $\overline{N(B)-L_r}|_{\partial X_K}$, which both have $I$-bundle structures. Simultaneously endowing $X_B$ and $\overline{N(B)-L_r}$ with an $I$-bundle structure (as above) foliates $A_r$ by circles of slope $r$; thus $\partial X_K$ is foliated by simple closed curves of slope $r$. 
\end{proof}

\subsubsection{\textbf{\textup{Produce taut foliations in $S^3_r(K)$ via Dehn filling}}} \label{subsection:step5c}

For all rational $r <2g(K)-1$, $X_K$ admits a taut foliation $\mathcal{F}_r$ foliating $\partial X_K$ in simple closed curves of slope $r$. Performing $r$-framed Dehn filling endows $S^3_r(K)$ with a taut foliation. 

To summarize for $P(-2,3,7)$: we constructed a laminar branched surface $B \subset X_K$. The induced train track $\tau$ carries all rational slopes in $(-\infty, 2g(K)-1) = (-\infty, 9)$. Applying Proposition \ref{prop:laminarbranchedsurface}, Theorem \ref{thm:taolisinkdisk} and Proposition \ref{prop:extendlamination}, we deduce $X_K$ admits taut foliations meeting the boundary torus $T$ in simple closed curves of slope $r \in (-\infty, 2g(K)-1)$. Performing $r$-framed Dehn filling yields $S_r^3(K)$ endowed with a taut foliation. These manifolds are non-L-spaces; we have produced the taut foliations predicted by Conjecture \ref{conj:LSpace}. 

\section{Proof of Theorem \ref{thm:main}} \label{section:3braids}

\noindent In this section, we prove our main theorem: 

\noindent \textbf{Theorem \ref{thm:main}. }\textit{Let $K$ be a knot in $S^3$, realized as the closure of a positive 3-braid. Then for every rational $r < 2g(K)-1$, the knot exterior $X_K := S^3 - \accentset{\circ}{\nu}(K)$ admits taut foliations meeting the boundary torus $T$ in parallel simple closed curves of slope $r$. Hence the manifold obtained by $r$-framed Dehn filling, $S^3_r(K)$, admits a taut foliation.} \\

The proof requires generalizing the $P(-2,3,7)$ example of Section \ref{section:example}. In Section \ref{section:lemmas}, we prove a few lemmas. Three families of branched surfaces are constructed in Section \ref{section:buildingbranchedsurfaces}. 

\subsection{Co-orienting Arcs} \label{section:lemmas}

Given an arbitrary positive 3-braid word $\beta$, we choose $c_1+c_2-2$ product disks, as in Section \ref{subsection:step1}. We need a strategy for assigning co-orientations. As in Section \ref{subsection:step4a}, we will provide cusp directions in tandem with $\beta$, and analyze which cusping directions produce sink disks and linked arc pairs. We aim to maximize the slopes carried by $\tau$ while ensuring $B$ is sink disk free.

\begin{lemma} \label{lemma:sigmasinkdisk}
Suppose the subword $\sigma_i \sigma_i$ arises as the $j^{\text{th}}$ and $j+1^{\text{st}}$ letters in $\beta$. The cusping directions $(\leftarrow)^2, (\rightarrow)^2$, and $(\rightarrow \ \leftarrow)$ prevent the band sector $\mathbbm{b}_{j+1}$ from being a half sink disk.
\end{lemma}
\begin{proof}
As in Figures \ref{fig:two_arrows_same_direction} and \ref{fig:two_bands_linked_arcs}, $\alpha_{j}^{+}$ is isotopic to the co-core of $\mathbbm{b}_{j+1}$. If $\widehat{\alpha}_{j}^{-} = (\rightarrow)$, then by Lemma \ref{lemma:cuspdirections}, $\widehat{\alpha}_{j}^{+} = (\leftarrow)$, hence the directions $(\rightarrow)^2$ and $(\rightarrow \ \leftarrow)$ do not make $\mathbbm{b}_{j+1}$ a half sink disk. The cusping directions $(\leftarrow)^2$ have $\widehat{\alpha}_{j+1}^-$ pointing out of $\mathbbm{b}_{j+1}$. 
\end{proof}

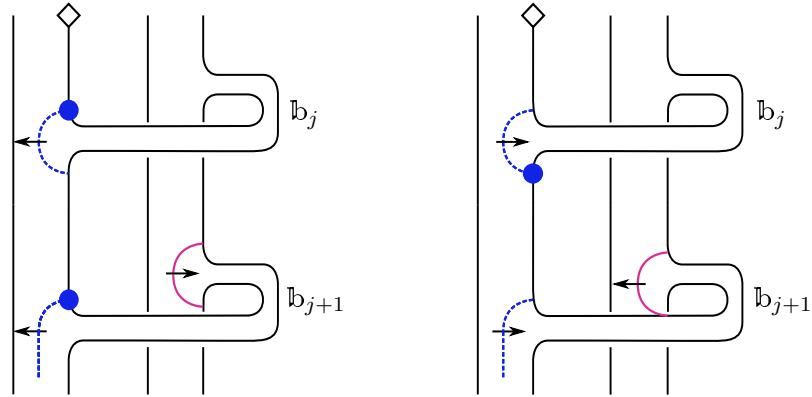
\begin{figure}[h!]\center
\labellist
\pinlabel {$\mathbbm{b}_j$} at 250 240
\pinlabel {$\mathbbm{b}_j$} at 650 240
\pinlabel {$\mathbbm{b}_{j+1}$} at 260 80
\pinlabel {$\mathbbm{b}_{j+1}$} at 660 80
\endlabellist
\includegraphics[scale=.45]{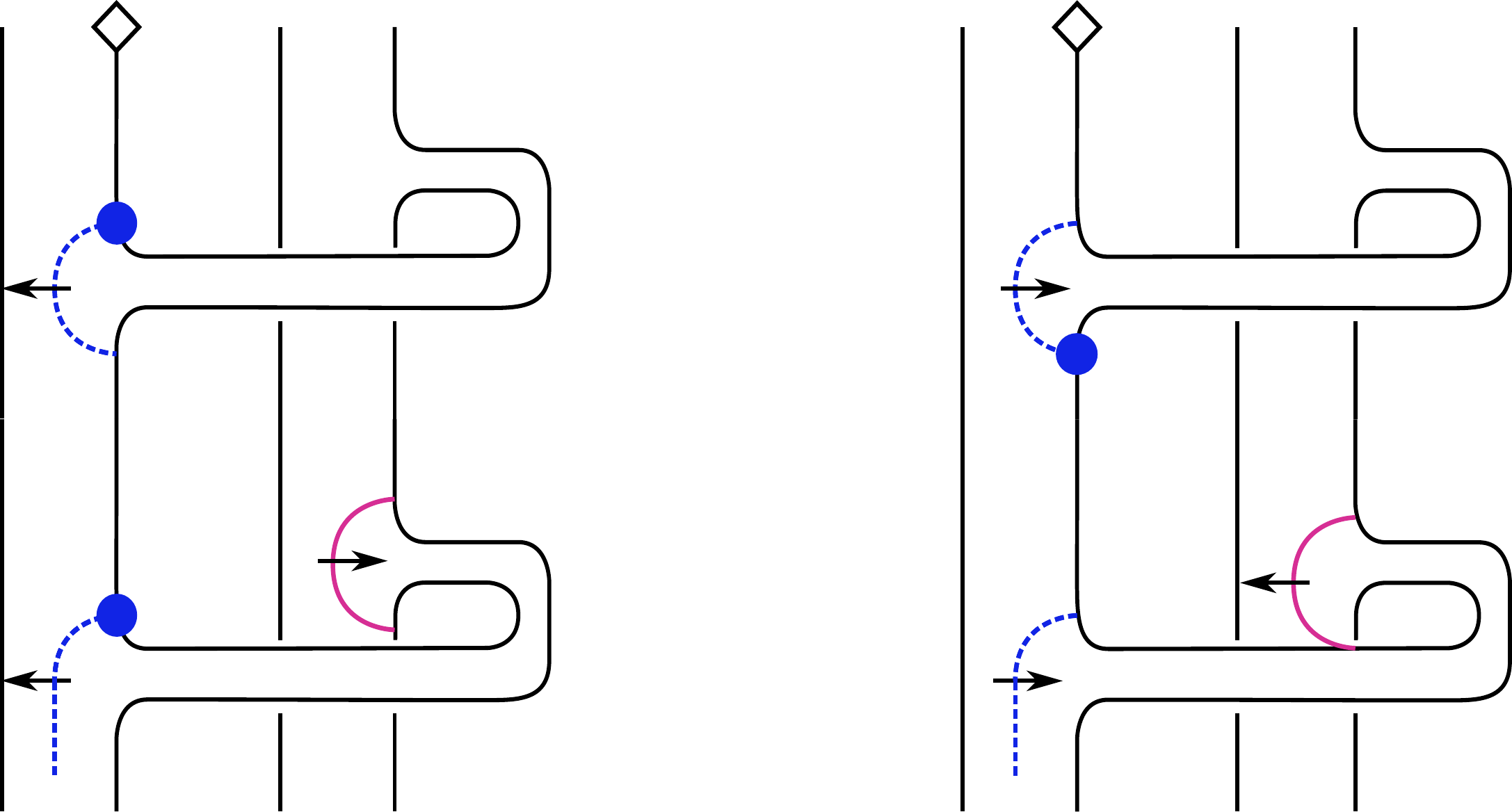}
\caption{The directions $(\rightarrow)^2$ and $(\leftarrow)^2$ do not make $\mathbbm{b}_{j+1}$ a half sink disk.} 
\label{fig:two_arrows_same_direction}
\end{figure}

\begin{figure}[h!]\center
\labellist
\pinlabel {$\mathbbm{b}_{j}$} at 600 1000
\pinlabel {$\mathbbm{b}_{j+1}$} at 640 580
\pinlabel {$\mathbbm{b}_{j+2}$} at 640 200
\endlabellist
\includegraphics[scale=.2]{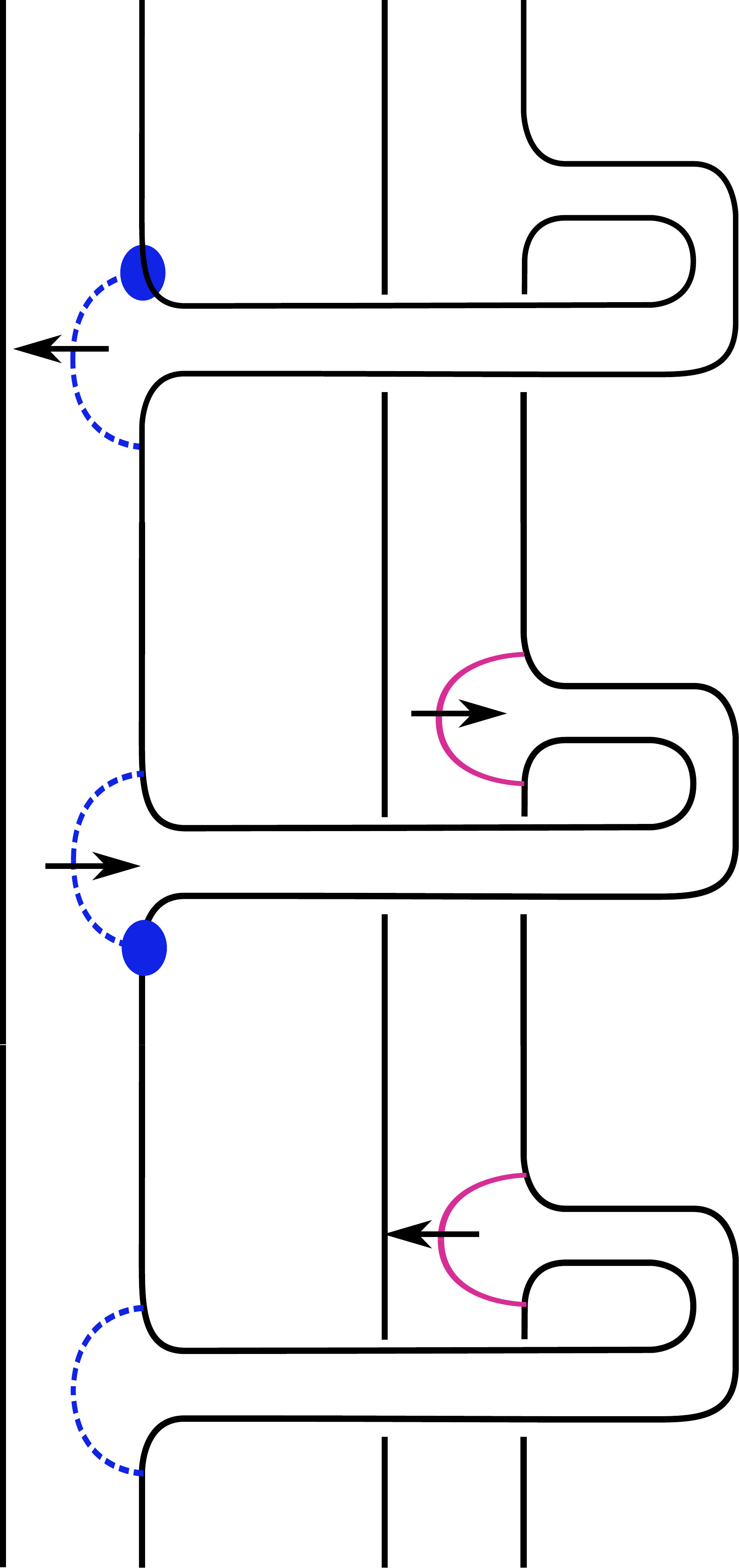}
\caption{The band $\mathbbm{b}_{j+1}$ is a half sink disk.} 
\label{fig:sink_disk_in_band}
\end{figure}

\begin{lemma} \label{lemma:sinkdiskbandsector}
Suppose $\beta$ contains the subword $\sigma_i \sigma_i \sigma_i$, arising as the $j, j+1, j+2$ letters of $\beta$. The cusping directions $(\leftarrow \ \ \rightarrow \ \ \star \ )$, $\star \in \{\rightarrow, \ \ \leftarrow, \ \ \ \ \}$ force $\mathbbm{b}_{j+1}$ to be a half sink disk.
\end{lemma}
\begin{proof}
As in Figure \ref{fig:sink_disk_in_band}, both $\alpha_{j+1}^{-}$  and $\alpha_{j}^{+}$ are isotopic to the co-core of $\mathbbm{b}_{j+1}$. Not only does $\widehat{\alpha}_{j+1}^{-}$ point into $\mathbbm{b}_{j+1}$, but by Lemma \ref{lemma:cuspdirections}, so does $\widehat{\alpha}_{j+1}^{-}$. 
\end{proof}

To produce a sink disk free branched surface, we should avoid the cusp directions $(\leftarrow \ \ \rightarrow)$. 

\begin{figure}[h!]\center
\labellist
\pinlabel {$\mathbbm{\alpha}_{j}^{-}$} at -30 400
\pinlabel {$\mathbbm{\alpha}_{j}^{+}$} at 350 250
\pinlabel {$\mathbbm{\alpha}_{j+1}^{-}$} at -50 100
\pinlabel {$\mathbbm{b}_{j}$} at 450 440
\pinlabel {$\mathbbm{b}_{j+1}$} at 475 150
\endlabellist
\includegraphics[scale=.28]{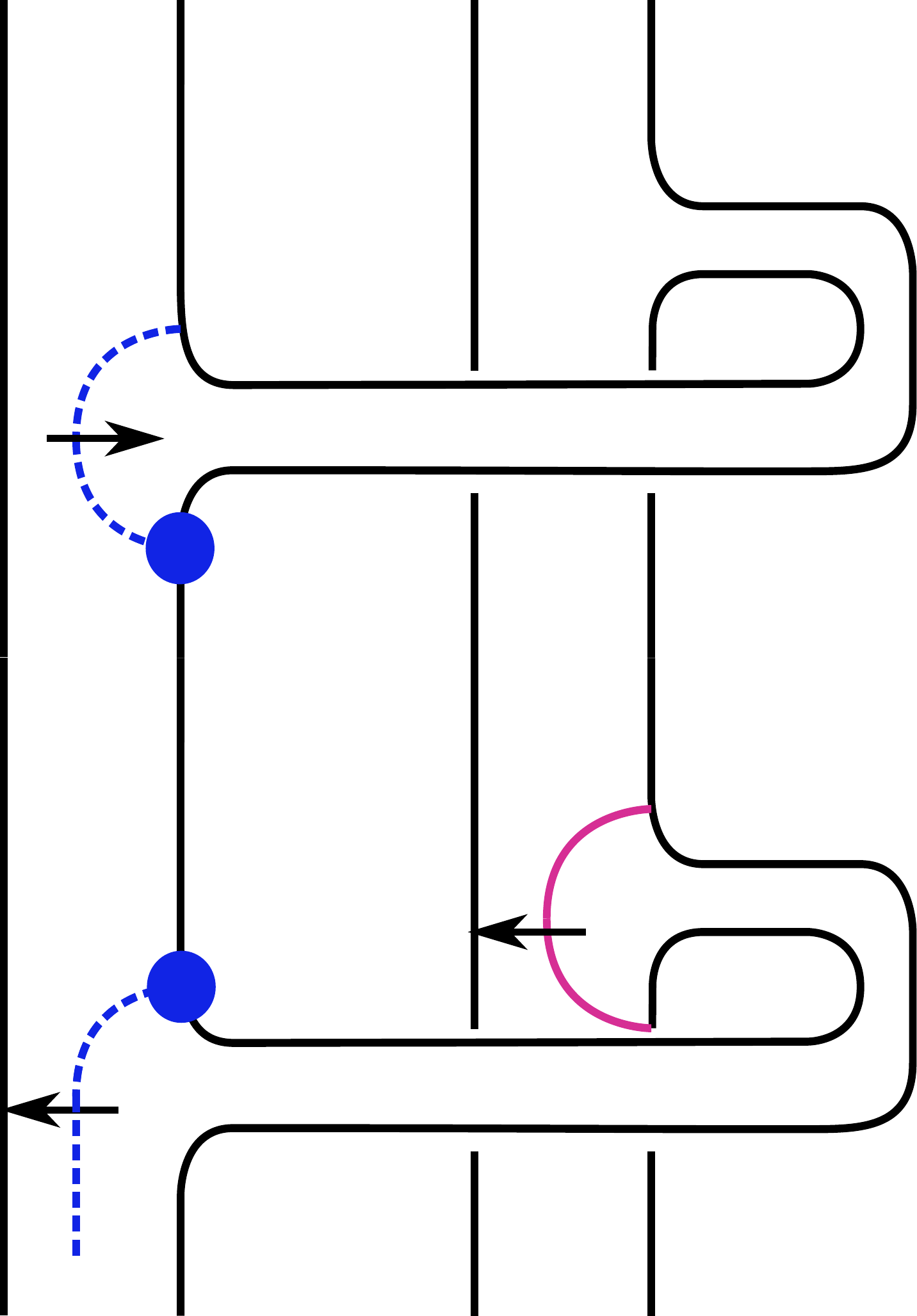}
\caption{The arcs $\alpha_j^{-}$ and $\alpha_{j+1}^{-}$ are linked.} 
\label{fig:two_bands_linked_arcs}
\end{figure}

\begin{figure}[h!]\center
\labellist
\pinlabel {$\mathbbm{\alpha}_{j}^{-}$} at -32 400
\pinlabel {$\mathbbm{\alpha}_{j+1}^{-}$} at 350 250
\pinlabel {$\mathbbm{\alpha}_{j}^{+}$} at 350 162
\pinlabel {$\mathbbm{b}_{j}$} at 450 440
\pinlabel {$\mathbbm{b}_{j+1}$} at 700 150
\endlabellist
\includegraphics[scale=.25]{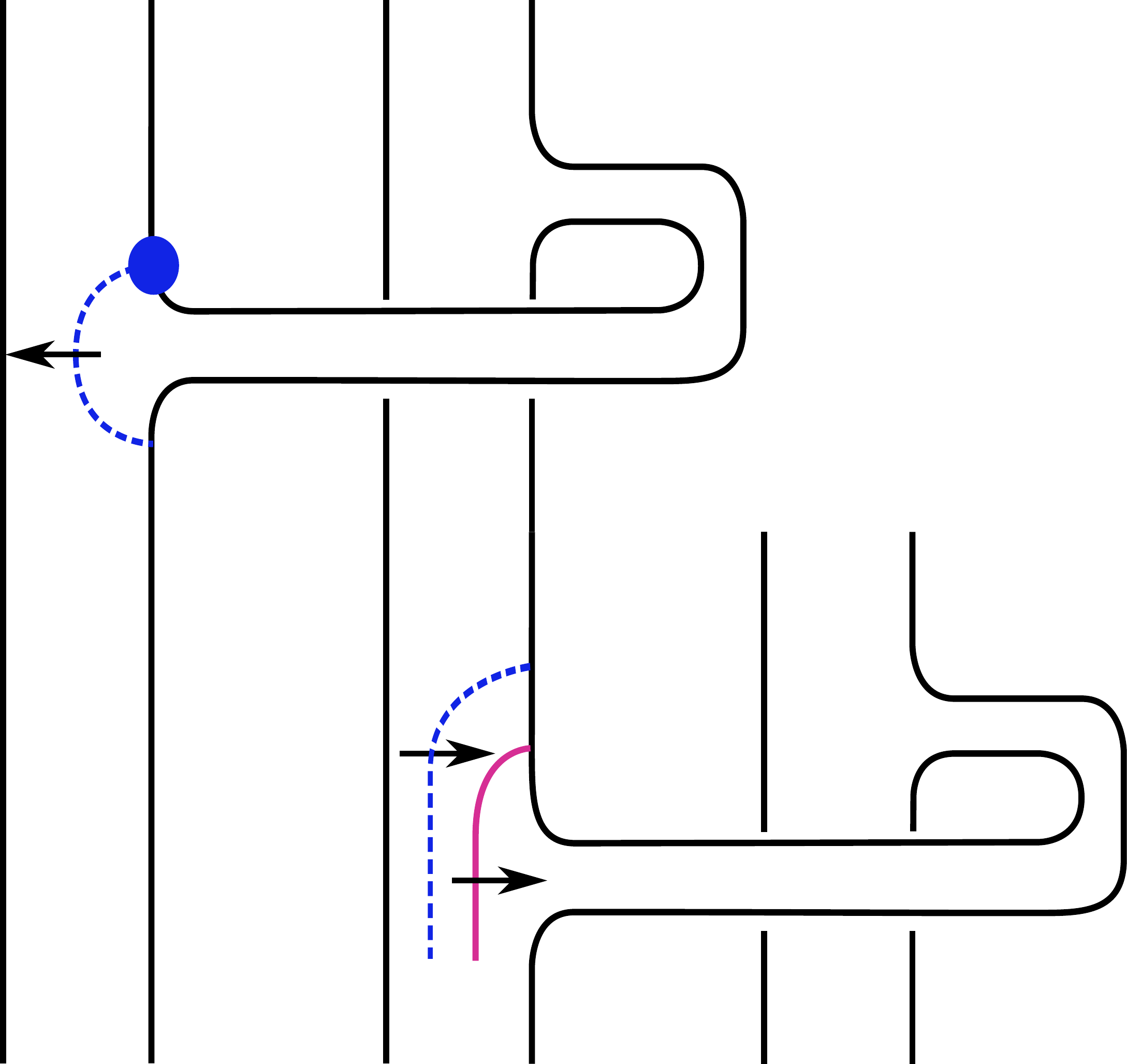}
\caption{The arcs $\alpha_j^-$ and $\alpha_{j+1}^-$ are not linked.} 
\label{fig:sigma1sigma2}
\end{figure}

\begin{lemma} \label{lemma:linkingarcs}
Suppose $\beta$ contains the subword $\sigma_i \sigma_i$ arising as the $j^{\text{th}}$ and $j+1^{\text{st}}$ letters in the braid word $\beta$. The associated cusping directions $(\leftarrow \ \ \leftarrow) \text{\ \ and \ \ } (\rightarrow \ \ \rightarrow)$ create an arc, unlinked from all other arcs, that contributes maximally to $\tau$. The cusping directions $(\rightarrow \ \ \leftarrow)$ create a pair of linked arcs.\end{lemma}
\begin{proof}
First, suppose $(\sigma_i)^2$ is cusped via $(\leftarrow)^2$, as in the left picture in Figure \ref{fig:two_arrows_same_direction}. The bolded endpoints of $\alpha_j^{-}$ and $\alpha_{j+1}^{-}$ contribute maximally to $\tau$. Traversing $K$ from $\Diamond$, we first encounter the upper endpoint of $\alpha_{j}^{-}$, and then its image: no point that contributes maximally to $\tau$ occurs between them. Thus $\alpha_j^{-}$ is unlinked from all other arcs. Analogously, if $(\sigma_i)^2$ is cusped via $(\rightarrow)^2$, $\alpha_j^-$ is unlinked from all other arcs, as in the right picture of Figure \ref{fig:two_arrows_same_direction}. If $(\sigma_i)^2$ is cusped via $(\rightarrow \ \leftarrow)$, $\alpha_j^-$ and $\alpha_{j+1}^-$ are linked, as in Figure \ref{fig:two_bands_linked_arcs}.
\end{proof}

\begin{lemma} \label{lemma:sigma1sigma2}
Suppose the subword $\sigma_1 \sigma_2$ occurs as the $j$ and $j+1$ letters of $\beta$. The arcs $\alpha_j$ and $\alpha_{j+1}$, cusped as $(\leftarrow \ \rightarrow)$, are unlinked.
\end{lemma}
\begin{proof}
As in Figure \ref{fig:sigma1sigma2}, $\alpha_{j}^{-}$ is unlinked from $\alpha_{j+1}$.
\end{proof}

\subsection{Building Branched Surfaces:} \label{section:buildingbranchedsurfaces}

\begin{defn} \label{defn:typesabc}
$\beta$ has the form described in Equation \ref{3braidstandardform}. Then $\beta$ is one of Types A, B, or C described below:

\begin{tabular}{ll}
\textbf{Type A}: & $k=1$, and $\beta = \sigma_1^{a_1} \sigma_2^{b_1}$. For $\hat{\beta}$ to be a knot, $a_1$ and $b_1$ are both odd. \\
& Note: $\hat{\beta} = T(2,a_1) \# T(2,b_1)$. \\
\textbf{Type B}: & $k= 2$, and $b_1 = b_2 = 1$. So, $\beta = \sigma_1^{a_1} \sigma_2 \sigma_1^{a_2} \sigma_2$ \\
\textbf{Type C}: & all other positive 3-braid closures; namely: \\
& $\bullet \ k=2$ and (up to cyclic rotation) $a_1, a_2, b_1 \geq 2, b_2 \geq 1$ \\
& $\bullet \ k \geq 3, a_i \geq 2, b_i \geq 1$ for all $i$. \\
\end{tabular}
\end{defn}

Given a positive 3-braid knot, we construct a branched surface by fusing $c_1 + c_2 - 2$ product disks to $F \times \{\frac{1}{2}\}$, such that we have exactly one linked pair of arcs. Propositions \ref{prop:typea}, \ref{prop:typeb}, \ref{prop:typec} construct the branched surfaces for \textbf{Types A, B, and C} respectively.

\begin{prop} \label{prop:typea} (Building the branched surface for \textbf{Type A}) \\
Suppose $\beta = \sigma_1^{a_1} \sigma_2^{b_1}$ for $a_1, b_1$ odd, and $K = \hat{\beta}$.  There exists a sink-disk free branched surface $B \subset X_K$, for $K = T(2,a_1) \# T(2,b_1)$, with exactly one pair of linked arcs. Moreover, there exists a sub-train-track $\tau'$ of $\tau$ carrying all rational slopes $r < 2g(K)-1$.
\end{prop}

\begin{proof}
First suppose $a_1, b_1 \geq 3$. We identify $c_1 + c_2 - 2 = a_1 + b_1 - 2$ product disks:
\begin{align} \label{case1cuspings}
\beta = \ \sigma_1^{a_1} \sigma_2^{b_1} \nonumber = & \ \sigma_1^{a_1 - 1} \ \ \ \ \  \sigma_1 \ \ \ \ \sigma_2^{b_1-2}\ \ \ \ \  \sigma_2 \ \ \ \  \sigma_2 \nonumber \\
 & (\rightarrow)^{a_1-1}\ \ (\ \ ) \ \ (\rightarrow)^{b_1-2}\ \ (\leftarrow)\ \ (\ \ )
\end{align}

The spine of the branched surface is built from $F \times \{\frac{1}{2}\}$, fused with the product disks specified. Applying Lemma \ref{lemma:isotopy} puts the product disks into standardized position; cusping as== instructed in (\ref{case1cuspings}) yields a branched surface $B$. In this case, all arcs on $F \times \frac{1}{2}$ are pairwise unlinked (see Figure \ref{fig:TypesABC} for an example). Lemma \ref{lemma:productdisksinkdisk} guarantees no product disk sector is a half sink disk, while Lemmas \ref{lemma:sigmasinkdisk} and \ref{lemma:sinkdiskbandsector} guarantee no band sectors are half sink disks. There are no polygon sectors. We check the disk sectors $S_1, S_2,$ and $S_3$ are not half sink disks. 
\begin{itemize}
\item  $\widehat{\alpha}_{1}^{-}$ points out of $S_1$.
\item $\widehat{\alpha}_{a_1+1}^{-}$ points out of the $S_2$ disk sector.
\item $\widehat{\alpha}_{a_1 + b_1 - 1}^{-}$ points into the $S_2$ disk sector, so $\widehat{\alpha}_{a_1 + b_1 - 1}^{+}$ points out of the $S_3$ disk sector. 
\end{itemize}

$B$ is sink disk free. By Lemma \ref{lemma:linkingarcs}, $\alpha_{a_1 + b_1 - 2}^{-}$ and $\alpha_{a_1 + b_1 - 1}^{-}$ are the unique pair of linked arcs.

Now suppose $a_1 \geq 3 \text{ and } b_1 = 1$, $a_1 = 1 \text{ and } b_3 \geq 1$, or $a_1 = b_1 = 1$. Then $\hat{\beta}$ is isotopic to $T(2, a_1)$, $T(2,b_1)$, or the unknot respectively. The canonical fiber surface for $K$ is produced after destabilization. The following instructions specify a construction of a branched surface for $T(2,n), n \geq 3$:
\begin{align*}
\beta = \ \sigma_1^n = & \ \sigma_1^{n-2} \ \ \ \ \ \sigma_1 \ \ \ \  \sigma_1 \\
 & (\rightarrow)^{n-2}\ \ (\leftarrow)\ \ (\ \ )
\end{align*}

Standardize the disks as in Lemma \ref{lemma:isotopy}. Lemmas \ref{lemma:productdisksinkdisk} and \ref{lemma:sinkdiskbandsector}, and \ref{lemma:sigmasinkdisk} guarantee no product disks or band sectors are half sink disks. There are no polygon sectors. $\widehat{\alpha}_1^{-}$ and $\widehat{\alpha}_{n-1}^{+}$ point out of $S_1$ and $S_2$ respectively, ensuring no disk sectors. Finally, Lemma \ref{lemma:linkingarcs} guarantees only $\alpha_{n-2}^{-}$ and $\alpha_{n-1}^{-}$ are linked. 

Thus for any $\beta = \sigma_1^{a_1} \sigma_2^{b_1}, \ a_1, b_1 \geq 1$ and odd, there exists a sink disk free branched surface $B$ with a unique pair of linked arcs. Including both sectors induced by $\alpha_{1}$ to $\tau'$ ensures that $\tau'$ carries all rational $r < 2g(K)-1$. 
\end{proof}

\begin{rmk} \label{rmk:compact_surface_type_A}
\textup{Eventually, we aim to conclude that $B$ is not just sink-disk-free, but that it is laminar. To do so, we need to modify the proof that $B$ does not fully carry an annulus (our proof of this in Proposition \ref{prop:laminarbranchedsurface} relied on a local model that does not apply for braids of Type $A$). However, this is straightforward: consider Figure \ref{fig:compact_surfaces}, and reverse the orientations on each of the cusp directions shown (i.e. $\widehat{\alpha}_1^-$ and $\widehat{\alpha}_2^-$ point out of $S_1$, and $\widehat{\alpha}_1^+$ points into $S_2$). The resulting local model for a branched surface now matches Type A branched surfaces. We preserve the labelling of the weights of each sector. The new cusp directions, combined with the switch relations at branch sectors, induce the following:
\begin{align*}
w_1 +w_4= w_2 \qquad w_3 +w_4= w_2 \qquad w_1 + w_5 = w_3
\end{align*}
Therefore, $w_1+w_4 = w_3+w_4$, which implies $w_1 = w_3$. Again, we conclude that $w_5=0$. We conclude that $B$ does not carry any compact surface, and therefore does not carry an annulus.}
\end{rmk}

\begin{figure}[h!]\center
\begin{minipage}{\textwidth}
\centering
\includegraphics[scale=.64]{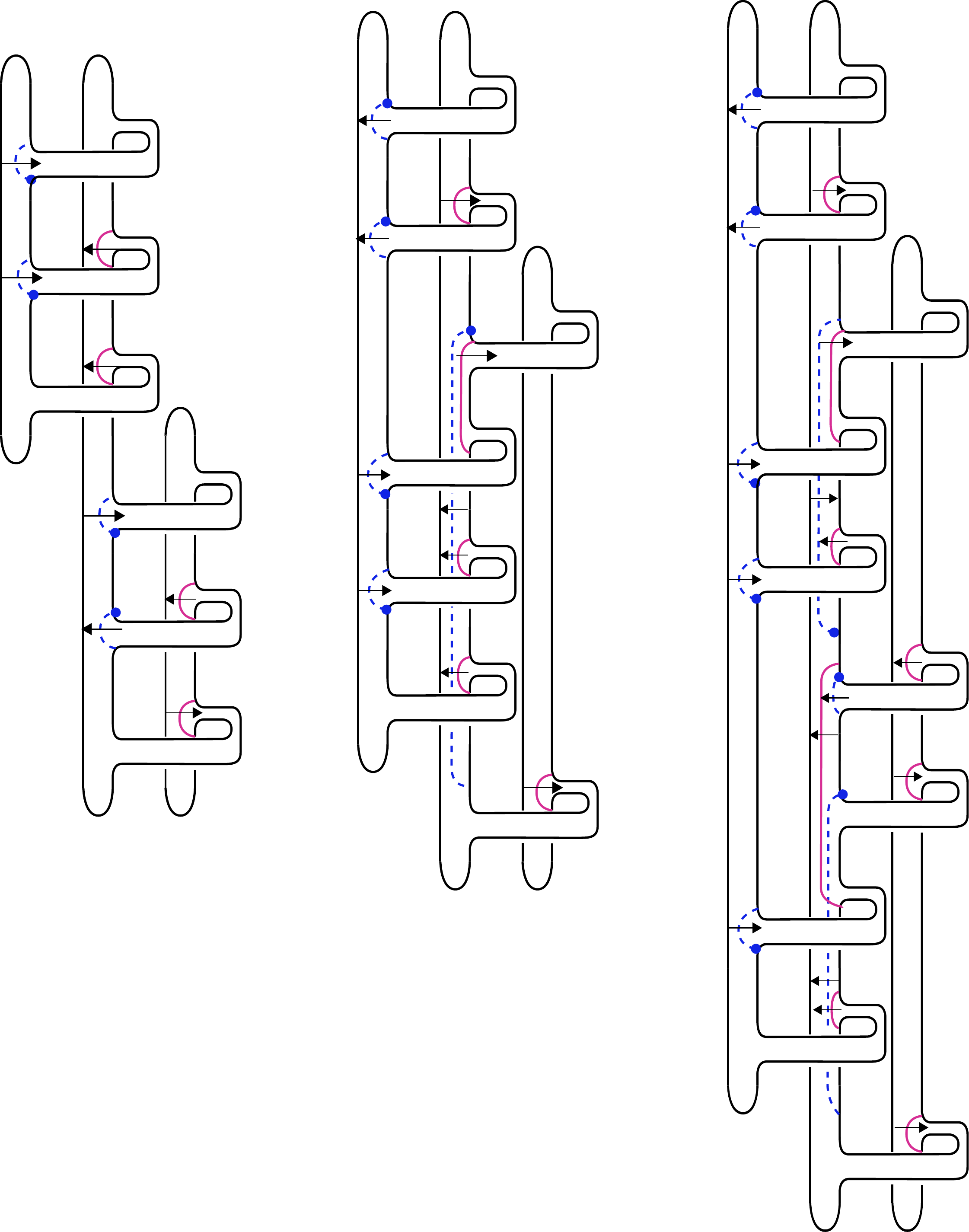}
\caption{From left to right: laminar branched surfaces of Types A, B, and C.} 
\label{fig:TypesABC}
\end{minipage}
\end{figure}

\begin{prop} \label{prop:typeb} (Building the branched surface for \textbf{Type B}) \\
Suppose $\beta = \sigma_1^{a_1} \sigma_2 \sigma_1^{a_2} \sigma_2, \ a_i \geq 2$ and $K = \hat{\beta}$.  There exists a sink-disk free branched surface $B \subset X_K$ with exactly one pair of linked arcs. Moreover, there is a sub-train-track $\tau'$ of $\tau$ carrying all rational slopes $r < 2g(K)-1$.
\end{prop}

\begin{proof}
The spine of the branched surface is built from $F \times \{\frac{1}{2}\}$, fused with the product disks specified below:
\begin{align} \label{case2cuspings}
\beta &= \sigma_1^{a_1} \sigma_2 \sigma_1^{a_2} \sigma_2 \nonumber \\
&= \sigma_1^{a_1} \ \ \ \ \ \sigma_2 \ \  \sigma_1^{a_2-1}\ \ \ \sigma_1 \ \ \sigma_2 \nonumber \\ 
&= (\leftarrow)^{a_1} (\leftarrow) (\rightarrow)^{a_2-1}(\ \ )(\ \ )
\end{align}

Lemma \ref{lemma:isotopy} puts the product disks into standardized position. Cusping the disks as specified in (\ref{case2cuspings}) yields a branched surface, as in Figure  \ref{fig:TypesABC}. By Lemma \ref{lemma:productdisksinkdisk}, no product disk sector is a half sink disk. No disk sectors are half sink disks:
\begin{itemize}
\item $\widehat{\alpha}_{a_1+2}^{-}$ points out of the $S_1$ disk sector
\item $\widehat{\alpha}_{1}^{-}$ points into the $S_1$ disk sector, so $\widehat{\alpha}_{1}^{+}$ points out of the $S_2$ disk sector
\item $\widehat{\alpha}_{a_1+1}^{-}$ points into the $S_2$ disk sector, so $\widehat{\alpha}_{a_1+1}^{+}$ points out of the $S_3$ disk sector
\end{itemize}

Lemmas \ref{lemma:sigmasinkdisk} and \ref{lemma:sinkdiskbandsector} guarantee no band sectors are sink disks. It remains to check the single polygon sector $P$, which lies in Seifert disk $S_2$. The boundary of $P$ meets $\alpha_j^{+}, a_1+2 \leq j \leq c_1+c_2-2, \alpha_{a_1 + 1}^{-}, \alpha_{a_1}^{+}$, and no other arcs $\alpha_j^{\pm}$. Since $\widehat{\alpha}_{a_1+1}^{-}$ points out of $P$, it is not a half sink disk. Thus, our branched surface $B$ is sink disk free. 

We are fusing $c_1 + c_2 - 2$ product disks to $F \times \{ \frac{1}{2} \}$, so there exists a sub-train-track $\tau'$ with $c_1 + c_2 -2$ sectors. By Lemmas \ref{lemma:linkingarcs} and \ref{lemma:sigma1sigma2}, $\alpha_{a_1}^{-}$ and  $\alpha_{a_1+1}^{-}$ are the unique pair of linked arcs. Thus $\tau'$ carries all slopes in $[0, c_1 + c_2 - 3) = [0, 2g(K)-1)$. Including both sectors induced by $\alpha_{a_1 + 1}$ to $\tau'$ ensures that $\tau'$ carries all slopes $r<2g(K)-1$. 
\end{proof}

The most nuanced construction arises in \textbf{Case C}:

\begin{prop} \label{prop:typec} (Building the branched surface for \textbf{Case C}) \\
Let $K = \hat{\beta}$, where $\beta$ is of \textbf{Case C} (see Definition \ref{defn:typesabc}). There exists a sink-disk free branched surface $B \subset X_K$ with a unique pair of linked arcs. Moreover, there is a sub-train-track $\tau'$ of $\tau$ carrying all rational slopes $r < 2g(K)-1$.
\end{prop}

\begin{proof}
The spine of the branched surface is built from $F \times \{\frac{1}{2}\}$, fused with the product disks specified by:
\begin{align} \label{case3cuspings}
\beta &= \sigma_1^{a_1} \sigma_2^{b_1} \sigma_1^{a_2} \sigma_2^{b_2} \ldots \sigma_1^{a_k} \sigma_2^{b_k} \nonumber \\
&= \sigma_1^{a_1} \ \ \ (\sigma_2)\ ( \sigma_2^{b_1-1}) \ \ \sigma_1^{a_2} \ \ \ \sigma_2^{b_2} \ \ \   \sigma_1^{a_3}\ \ \ \  \sigma_2^{b_3}\ \ \ \ldots (\sigma_1^{a_k-1})( \sigma_1) (\sigma_2^{b_k-1})(\sigma_2)  \nonumber \\
&= (\leftarrow)^{a_1} (\rightarrow)(\leftarrow)^{b_1-1} (\rightarrow)^{a_2} (\leftarrow)^{b_2} (\rightarrow)^{a_3} (\leftarrow)^{b_3}\ldots (\rightarrow)^{a_k-1}(\ \ )(\leftarrow)^{b_k-1}(\ \ )
\end{align}

Applying Lemma \ref{lemma:isotopy} puts the product disks into standardized position. Cusping the disks as specified in (\ref{case3cuspings}) yields a branched surface $B$. See Figure \ref{fig:TypesABC} for an example. 

We check for half sink disks: by Lemma \ref{lemma:productdisksinkdisk}, no product disk sector is a half sink disk. No disk sector is a half sink disk:
\begin{itemize}
\item $\widehat{\alpha}_{a_1+b_1+1}^{-}$ points out of the $S_1$ disk sector
\item $\widehat{\alpha}_{1}^{-}$ points into the $S_1$ disk sector, $\widehat{\sigma}_{1}^{+}$ points out of the $S_2$ disk sector
\item whether $k=2$ or $k=3$, there exists a $\sigma_2$ letter in $\beta$ cusped via $(\leftarrow)$. The corresponding image arc will point out of the $S_3$ disk sector
\end{itemize}

Lemmas \ref{lemma:sigmasinkdisk} and \ref{lemma:sinkdiskbandsector} guarantee no band sectors are sink disks. 

It remains to analyze polygon sectors. Unlike the cases analyzed in Propositions \ref{prop:typea} and \ref{prop:typeb}, there may be intersection points between $\alpha^{+}$ and $\alpha^{-}$ arcs. Each intersection point will occur between consecutive blocks.
Moreover, each intersection point indicates the existence of two polygon sectors. Reading from top-to-bottom, we number the intersection points $i_1, \ldots, i_m, \ldots i_{n}$. We note that $n$ is bounded above by $k-1$, where $k$ is the total number of blocks in $\beta$. Moreover, $n=k-1$ if and only if for every $t$, $b_t \geq 2$. In particular, the intersection point $i_m$ does \textit{not} have to occur between the blocks $m$ and $m+1$. For example, in the rightmost diagram in Figure \ref{fig:TypesABC}, the unique intersection point $i = i_1$ occurs between blocks 2 and 3.

As an intersection point indicates the existence of a pair of polygon sectors, we will identify the individual polygon sectors by their relative position. The polygon sectors associated to the intersection point $i_m$ are labelled $P_{u,m}$ and $P_{\ell,m}$, and called \textit{upper polygon} and \textit{lower polygon} sectors respectively. 

We first analyze the behavior of $b_1$. If $b_1 = 1$, then we have a single polygon sector $P$. It is not a half sink disk, as $\widehat{\alpha}_{a_1}^{+}$ points out of the region; see (Figure \ref{fig:polygon_b1}, left). If $b_1 \geq 2$, we have a pair of polygon sectors to analyze; see (Figure \ref{fig:polygon_b1}, right): 

\begin{minipage}[t]{0.5\textwidth}
$\bullet$ The boundary of $P_{u,1}$ meets the arcs
\begin{itemize}
\item[$\circ$] $\alpha_{j}^{-}, a_1+1 \leq j \leq a_1+b_1$
\item[$\circ$] $\alpha_{a_1}^{+}$, 
\end{itemize}
\end{minipage}
\begin{minipage}[t]{0.5\textwidth}
$\bullet$ The boundary of $P_{\ell, 1}$ meets the arcs
\begin{itemize}
\item[$\circ$] $ \alpha_{a_1 + b_1}^{-}$,
\item[$\circ$] $\alpha_{a_1}^{+}$, 
\item[$\circ$] $\alpha_j^{+}, a_1+b_1+1 \leq j \leq a_1+b_1+a_2-1$, 
\end{itemize}
\end{minipage}

\begin{figure}[h!]\center
\labellist
\pinlabel {$\mathbbm{b}_{a_1}$} at 600 1230
\pinlabel {$\mathbbm{b}_{a_1+1} = \mathbbm{b}_{a_1+b_1}$} at 960 920
\pinlabel {$\mathbbm{b}_{a_1+b_1+1}$} at 660 630
\pinlabel {$\mathbbm{b}_{a_1+b_1+2}$} at 660 300
\pinlabel {$\alpha_{a_1}^-$} at 10 1200
\pinlabel {$\alpha_{a_1+b_1+1}^-$} at -55 610
\pinlabel {$\alpha_{a_1+b_1+2}^-$} at -55 300
\pinlabel {$P$} at 450 100
\pinlabel {$\mathbbm{b}_{a_1}$} at 1850 1220
\pinlabel {$\mathbbm{b}_{a_1+1}$} at 2070 925
\pinlabel {$\mathbbm{b}_{a_1+b_1}$} at 2080 550
\pinlabel {$\mathbbm{b}_{a_1+b_1+1}$} at 1900 280
\pinlabel {$P_{u,1}$} at 1715 725
\pinlabel {$P_{\ell,1}$} at 1715 100
\pinlabel {$\alpha_{a_1}^-$} at 1260 1200
\pinlabel {$\alpha_{a_1+b_1+1}^-$} at 1200 250
\endlabellist
\includegraphics[scale=.2]{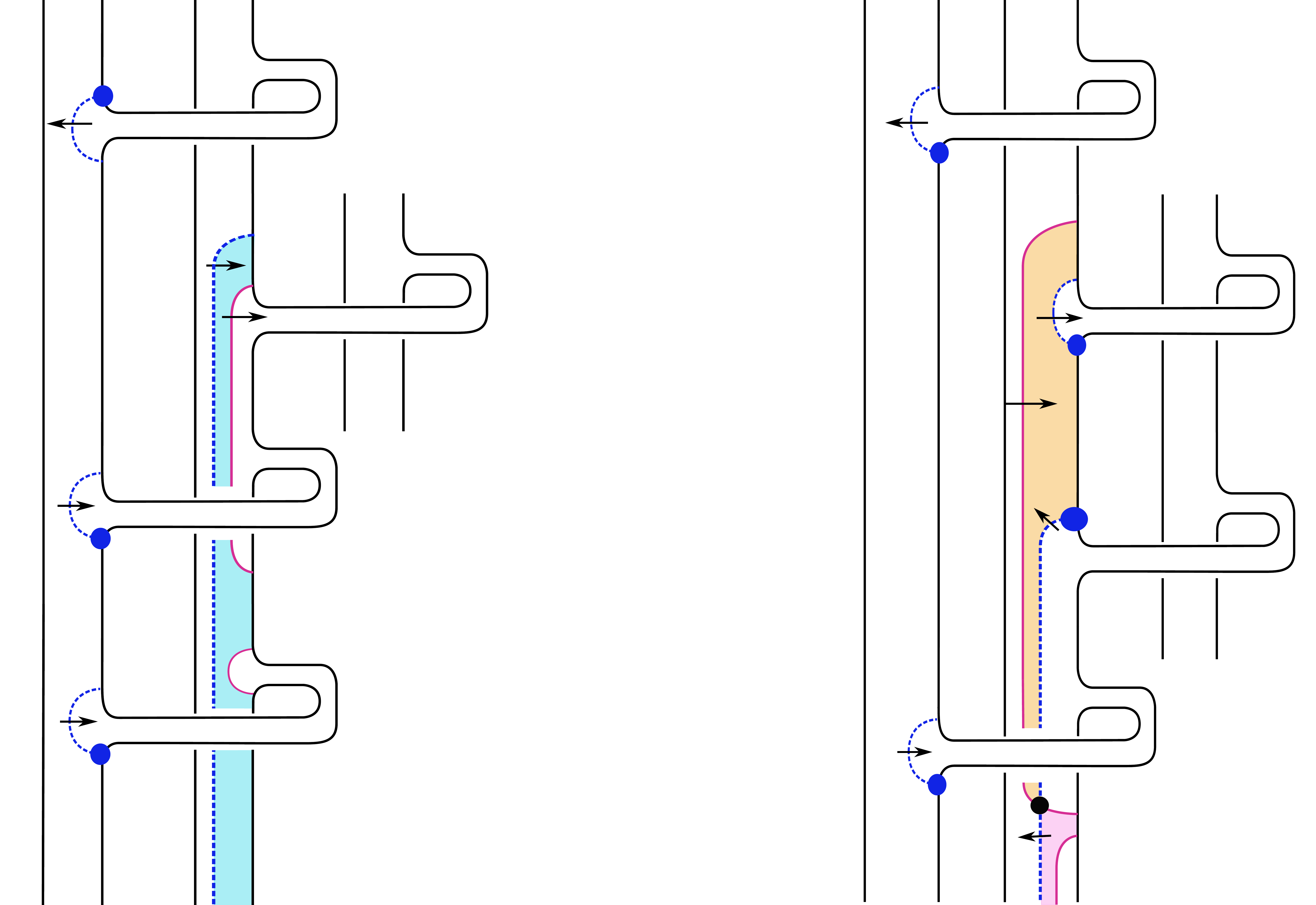}
\caption{Studying $b_1$. On the left, we have $b_1=1$. The shaded region is the single polygon sector $P$, which is not a sink disk, as $\widehat{\alpha}_{a_1}^{+}$ points out of it. On the right, an example with $b_1=2$. The upper and lower polygon sectors, $P_{u,1}$ and $P_{\ell,1}$, are shaded; these polygon sectors meet at the point $i_1$ (not labelled, but indicated in the diagram). $P_{u,1}$ is not a sink disk because $\widehat{\alpha}_{a_1+1}^-$ points out of it. $P_{\ell,1}$ is not a sink disk because $\widehat{\alpha}_{a_1+b_1}^-$ points out of it.} 
\label{fig:polygon_b1}
\end{figure}
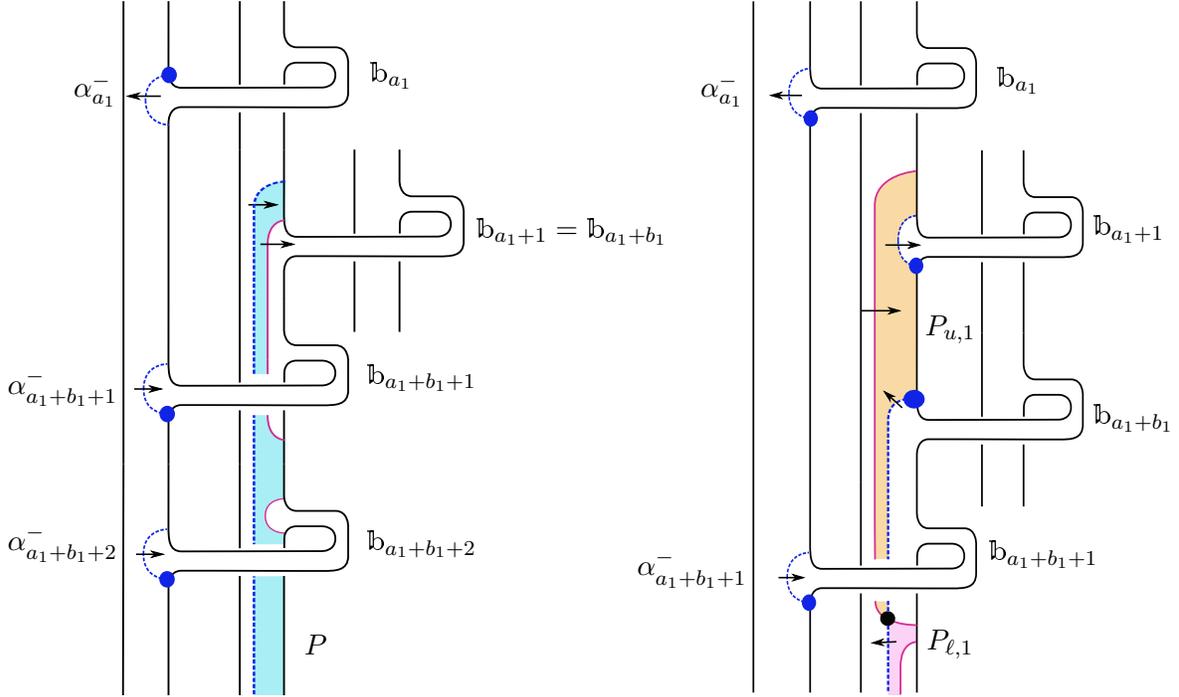

Since $\widehat{\alpha}_{a_1+1}^{-}$ points out of $P_{u,1}$, and $\widehat{\alpha}_{a_1+b_1}^{-}$ points out of $P_{\ell, 1}$, neither are half sink disks. 

We now analyze the remaining polygon sectors. If, for $q \geq 2$, the $q^{\text{th}}$ block has $b_q = 1$, there will be a single polygon region. It is not a half sink disk because $\widehat{\alpha}_{a_1+b_1+\ldots+a_q}^{-}$ points out of it region (see Figure \ref{fig:polygon_b2}, left). If $b_q \geq 2$, then the polygon sectors come in pairs; all such pairs can be analyzed simultaneously (see Figure \ref{fig:polygon_b2}, right). Suppose $i_m$ is the intersection point between $P_{u,m}$ and $P_{\ell,m}$, which occur at the transition from block $t$ to block $t+1$. For the pair $P_{u,m}$ and $P_{\ell,m}$:

\newpage
\begin{itemize}
\item the boundary of $P_{u,m}$ meets the arcs
\begin{itemize}
\item[$\circ$] $\alpha_j^{-}$, where $a_1+b_1 + \ldots a_t + 1 \leq j \leq a_1+b_1 + \ldots + a_t + b_t$
\item[$\circ$] $\alpha_{a_1+b_1 +\ldots + a_t}^{+}$ \\
\end{itemize}

\item the boundary of $P_{\ell, m}$ meets the arcs
\begin{itemize}
\item[$\circ$] $\alpha_{a_1 + b_1+\ldots+a_t+b_t}^{-}$
\item[$\circ$] $\alpha_{a_1+b_1 +\ldots + a_t}^{+}$
\item[$\circ$] $\alpha_{j}^{+}$, where $a_1+b_1+\ldots+b_t+1 \leq j \leq a_1+b_1+\ldots+b_t+a_{t+1} -1$
\end{itemize}
\end{itemize}

For each $2 \leq m \leq n$, $P_{u,m}$ is not a sink disk: $\widehat{\alpha}_{a_1+b_1+\ldots+a_t}^{+}$ points out of it. Furthermore, $P_{\ell,m}$ has $\widehat{\alpha}_{a_1+b_1+\ldots+a_t+b_t}^{-}$ pointing out of it. Thus $B$ is sink disk free. 

\begin{figure}[h!]\center
\labellist
\pinlabel {$\mathbbm{b}_{a_1+b_1+\ldots+a_q}$} at 710 1220
\pinlabel {$\mathbbm{b}_{a_1+b_1+\ldots+a_q+b_q}$} at 960 920
\pinlabel {$\mathbbm{b}_{a_1+b_1+\ldots+a_q+b_q+1}$} at 780 630
\pinlabel {$\mathbbm{b}_{a_1+b_1+\ldots+a_q+b_q+2}$} at 780 300
\pinlabel {$\alpha_{a_1+b_1+\ldots+a_q}^-$} at -110 1200
\pinlabel {$\mathbbm{b}_{a_1+b_1+\ldots+a_t}$} at 1950 1220
\pinlabel {$\mathbbm{b}_{a_1+\ldots+a_t+1}$} at 2150 925
\pinlabel {$\mathbbm{b}_{a_1+\ldots+a_t+b_t}$} at 2170 550
\pinlabel {$\mathbbm{b}_{a_1+\ldots+b_t+1}$} at 1950 270
\pinlabel {$P_{u,m}$} at 1715 725
\pinlabel {$P_{\ell,m}$} at 1715 100
\pinlabel {$\alpha_{a_1+\ldots+a_t}^-$} at 1180 1200
\endlabellist
\includegraphics[scale=.185]{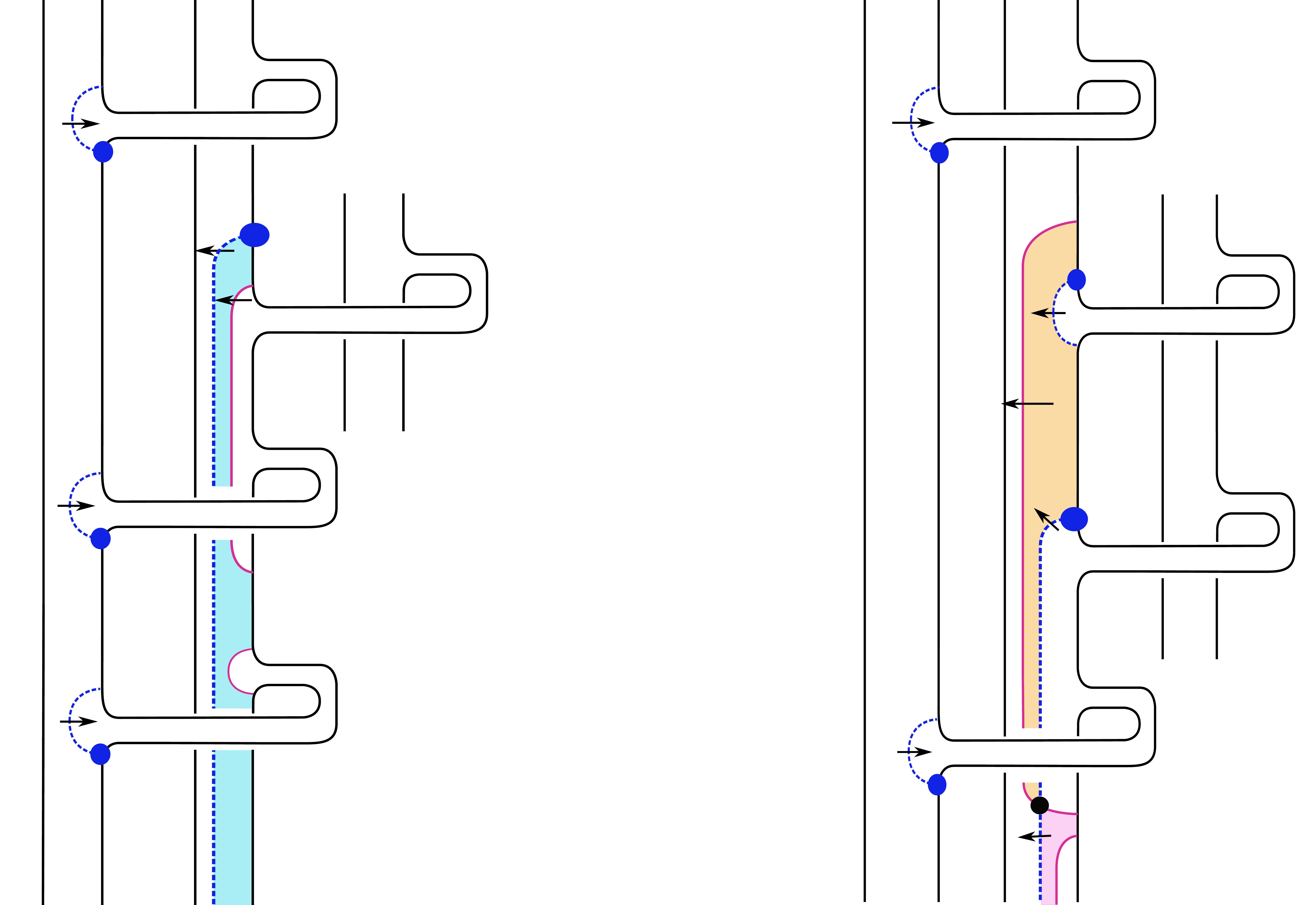}
\caption{On the left: we have $b_q=1$. The shaded region is the single polygon sector, which is not a sink disk because $\widehat{\alpha}_{a_1+b_1+\ldots+a_q}^{+}$ points out of it. On the right: an example with $b_m=2$. The upper and lower polygon sectors, $P_{u,m}$ and $P_{\ell,m}$, are shaded; they meet at the point $i_m$ (not labelled, but indicated in the diagram). $P_{u,m}$ is not a sink disk because $\widehat{\alpha}_{a_1+\ldots+a_t}^+$ points out of it. $P_{\ell,m}$ is not a sink disk because $\widehat{\alpha}_{a_1+\ldots+a_t+b_t}^-$ points out of it.} 
\label{fig:polygon_b2}
\end{figure}

We cusped $(c_1 - 1) + (c_2 - 1)$ arcs. By Lemma \ref{lemma:linkingarcs}, there exists a single linked pair, arising from the arcs associated to the first two occurrences of $\sigma_2$ in $\beta$. Thus, there exists a sub-train-track $\tau'$ carrying all slopes in $[0, c_1 + c_2 - 3) = [0, 2g(K)-1)$.  Including the sectors induced by $\alpha_{a_1 + 1}$ to $\tau'$ ensures that $\tau'$ carries all rational $r < 2g(K)-1$. 
\end{proof}

\subsection{Finale}

We conclude this section with the proof of the main theorem. 

\begin{proof}[Proof of Theorem \ref{thm:main}]
Let $K \subset S^3$ be the closure of a positive 3-braid $\beta$. After isotopy, $\beta$ has the form specified by Equation \ref{3braidstandardform}, and by Definition \ref{defn:typesabc} is Type A, B or C. By Propositions \ref{prop:typea}, \ref{prop:typeb}, \ref{prop:typec}, there exists a branched surface $B \subset X_K$
inducing a sub-train-track $\tau'$ carrying all rational slopes in the interval $(-\infty, \ 2g(K)-1)$. $B$ is laminar Proposition \ref{prop:laminarbranchedsurface} (we note that if $\beta$ is Type A, then we additionally apply Remark \ref{rmk:compact_surface_type_A}). Applying Theorem \ref{thm:taolisinkdisk} yields a family of essential laminations $\{ \mathcal{L}_r  \ | \ r \in (-\infty, 2g(K)-1) \cap \Q \}$, where $\mathcal{L}_r$ meets $\partial X_K$ in simple closed curves of slope $r$. Proposition \ref{prop:extendlamination} extends the essential lamination $\mathcal{L}_r$ to a taut foliation $\mathcal{F}_r$ in $X_K$, foliating $\partial X_K$ by simple closed curves of slope $r$. Performing $r$-framed Dehn filling yields $S_r^3(K)$ endowed with a taut foliation.
\end{proof}

\section{Proof of Theorem \ref{thm:1bridgebraids}} \label{section:1bridgebraids}

We generalize the techniques developed in Sections \ref{section:example} and \ref{section:3braids} to produce taut foliations in 1-bridge braid exteriors. Gabai defines a 1-bridge braid $K(w, b, t)$ in $D^2 \times S^1$ to be a knot, realized as the closure of a positive braid $\beta$, which is specified by three parameters: $w$, the braid index; $b$, the bridge width; and $t$, the twist number:
$\beta = (\sigma_b \sigma_{b-1} \ldots \sigma_2 \sigma_1) (\sigma_{w-1} \sigma_{w-2} \ldots \sigma_2 \sigma_1)^t$
where $1 \leq b \leq w-2, \ 1 \leq t \leq w-2$ \cite{Gabai:1BridgeBraids}. We consider a slightly more general definition:

\begin{defn} \label{defn:1bridgebraid}
A (positive) 1-bridge braid $K$ in $S^3$ is a knot realized as the closure of a braid $\beta$ on $w$-strands, where $$\beta = \underbrace{(\sigma_b \sigma_{b-1} \ldots \sigma_2 \sigma_1)}_{\textup{bridge subword}} (\sigma_{w-1} \sigma_{w-2} \ldots \sigma_2 \sigma_1)^t$$ for $ w \geq 3, \ 1 \leq b \leq w-2, \ t \geq 1$. We call the first $b$ letters of $\beta$ the \textbf{bridge subword}.
\end{defn}

In particular, we allow a 1-bridge braid in $S^3$ to have arbitrarily large twist number.

\begin{rmk}
There are no 1-bridge braids with $w=3$; we may assume $w \geq 4$.
\end{rmk}

\noindent \textbf{Theorem \ref{thm:1bridgebraids}.}
\textit{Let $K$ be a (positive) 1-bridge braid in $S^3$. Then for every $r \in (-\infty, g(K)) \cap \Q$, the knot exterior $X_K := S^3 - \accentset{\circ}{\nu}(K)$ admits taut foliations meeting the boundary torus $T$ in parallel simple closed curves of slope $r$. Moreover, the manifold obtained by $r$-framed Dehn filling, $S^3_r(K)$, admits a taut foliation.}

Every 1-bridge braid $K$ is a fibered knot in $S^3$. As in Theorem \ref{thm:main}, proving Theorem \ref{thm:1bridgebraids} requires building a laminar branched surface $B$ from a copy of the fiber surface $F$ and a collection of product disks. 

\begin{defn}
Let $\mathcal{B}_w$ denote the braid group on $w$ strands. Suppose $\beta' \in \mathcal{B}_w$ such that $\beta' = \sigma_m \sigma_{m-1} \sigma_{m-2} \ldots \sigma_2 \sigma_1$, with $1 \leq m \leq w-1$. We call the canonical fiber surface $F'$ for $\beta'$, built from $w$ disks and $m$ 1-handles, a \textbf{horizontal slice}. 
\end{defn}

We can view the canonical fiber surface $F$ for a 1-bridge braid $K(w,b,t)$ as built by vertically stacking $t+1$ horizontal slices, $\mathbbm{h}_0, \mathbbm{h}_1, \ldots, \mathbbm{h}_{t+1}$: numbering the horizontal slices from top-to-bottom, the horizontal slice $\mathbbm{h}_0$ comes from the bridge subword; the remaining $t$ horizontal slices $\mathbbm{h}_1, \ldots \mathbbm{h}_{t}$ come from the $t$ occurrences of the subword $\sigma_{w-1} \sigma_{w-2} \ldots \sigma_{2} \sigma_1$ in $\beta$; see (Figure \ref{fig:1BB_horizontal_both}, upper) for an example.

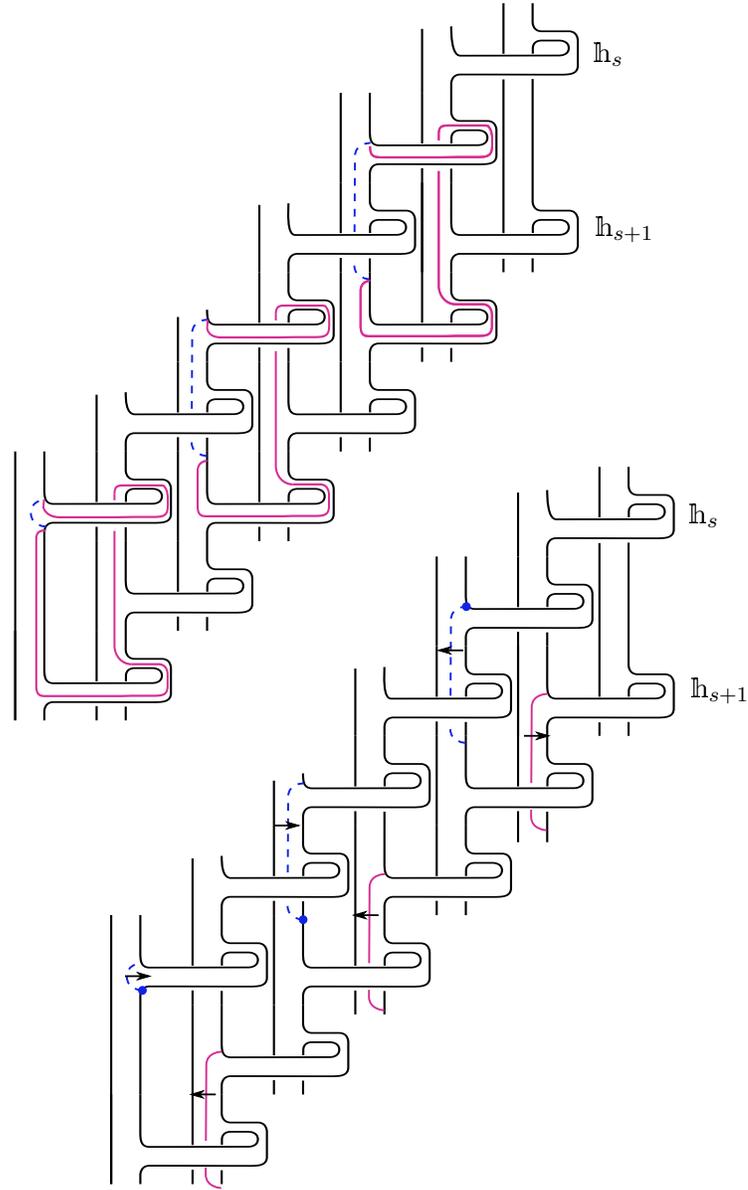
\begin{figure}[h!]\center
\labellist
\pinlabel {$\mathbbm{h}_s$} at 400 730
\pinlabel {$\mathbbm{h}_{s+1}$} at 410 620
\pinlabel {$\mathbbm{h}_s$} at 460 440
\pinlabel {$\mathbbm{h}_{s+1}$} at 470 330
\endlabellist
\includegraphics[scale=.605]{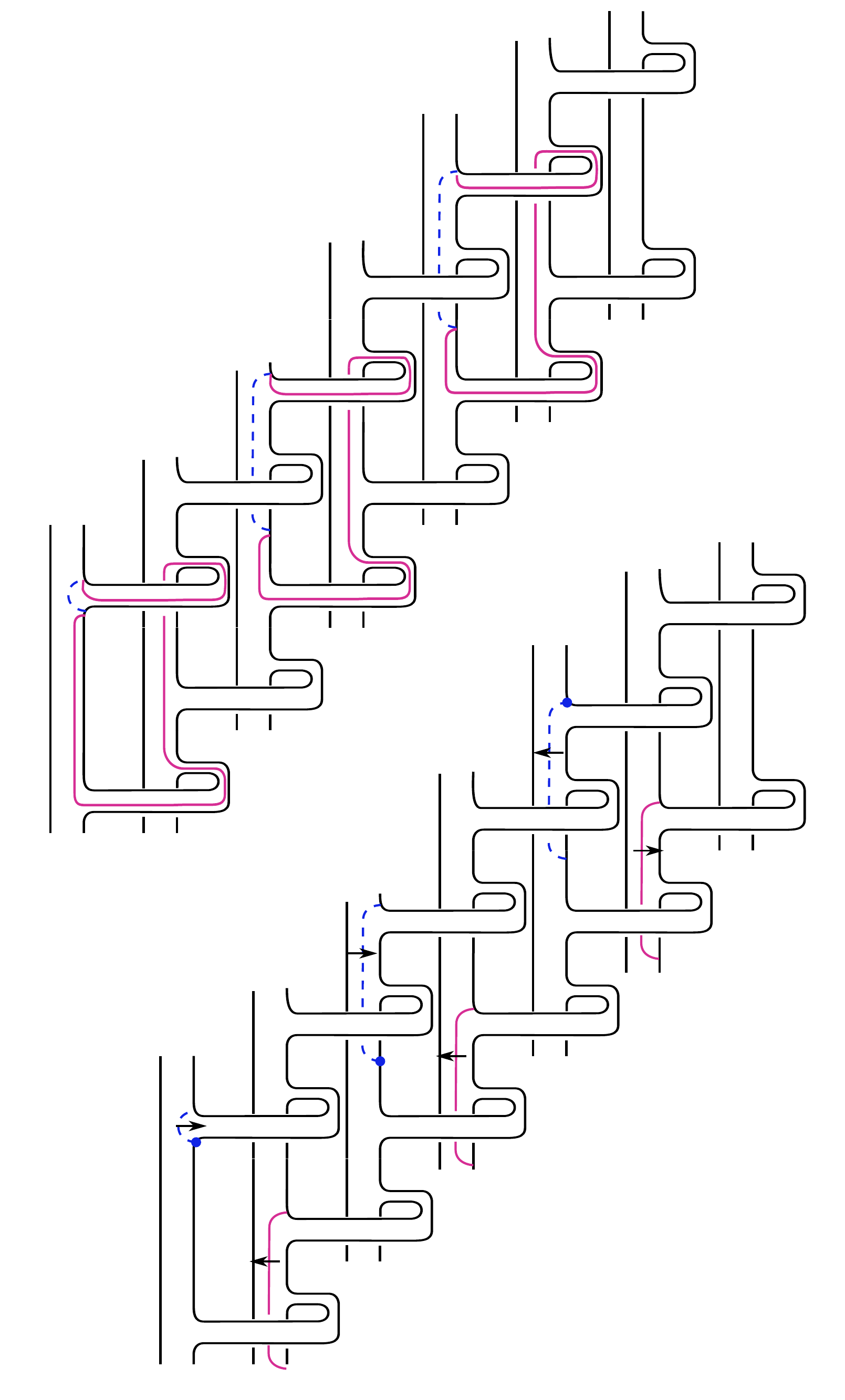}
\caption{Upper: the consecutive horizontal slices $\mathbbm{h}_s$ and $\mathbbm{h}_{s+1}$ of a 1-bridge braid fiber surface. As in Figure \ref{fig:2productdisks}, we have identified three product disks, by indicating where the disks meet $F^{-} \cup F^{+}$. The disks look like those in Figure \ref{fig:eg_productdisk}. Lower: we first performed a spinal isotopy so that the $\alpha^{+}$ arcs lie only in Seifert disks, and then co--oriented the disks as shown. The result is a portion of a branched surface.} 
\label{fig:1BB_horizontal_both}
\end{figure}

\begin{defn}
A Seifert disk $S_i$ is \textbf{odd (even)} if $i$ is odd (even). 
\end{defn}

As in Sections \ref{section:example} and \ref{section:3braids}, we provide cusping directions alongside $\beta$. That is, given a 1-bridge braid $K$ with braid word presented as in Definition \ref{defn:1bridgebraid}, we will choose disjoint product disks as in Section \ref{subsection:step1}. The boundaries of these disks will lie entirely in consecutive Seifert disks and consecutive bands of the same type; see (Figure \ref{fig:1BB_horizontal_both}, upper). As in Definition \ref{defn:cuspbraidword} (with the paragraph preceding it) and Section \ref{section:3braids}, the data of the disks and their co-orientations are recorded in tandem with the braid word; see (Figure \ref{fig:1BB_horizontal_both}, lower) for example of a portion of the resulting branched surface.

\begin{prop} \label{prop:1BB_branchedsurface}
For $K$ a 1-bridge braid in $S^3$, the following cusping directions specify a sink disk free branched surface:
\begin{itemize}
\item $\sigma_i$ is cusped via $( \ \ ) \iff i$ is even, or $i$ is odd and $\sigma_i$ is associated to a 1-handle used to build $\mathbbm{h}_{t}$.
\item Otherwise, $\sigma_i$ is cusped via $( \leftarrow )$ or $( \rightarrow )$, as specified below:
\begin{itemize}
\item[$\circ$]  The first occurrence of $\sigma_i$ in $\beta$ is cusped $(\leftarrow)$.
\item[$\circ$] All other occurrences of $\sigma_i$ in $\beta$ are cusped via $(\rightarrow)$.
\end{itemize}
\end{itemize}
\end{prop}

\begin{proof}
Following Sections \ref{section:example} and \ref{section:3braids}, the directions above specify arcs $\alpha_{j}^{-}$; applying the monodromy to these arcs produces the product disks $\{D_j\}$. Build the spine for a branched surface from $F \times \{\frac{1}{2}\}$ and $\{D_j\}$. Applying the proof of Lemma \ref{lemma:isotopy} splits the spine of $B$, putting the disks in standard position. After standardizing, all $\alpha_j^{-}$ lie in odd Seifert disks $S_i$, and all $\alpha_{j}^{+}$ lie in even Seifert disks. Choosing co-orientations for $\{D_j\}$ as specified by the instructions provided yields a branched surface $B$. See Figure \ref{fig:1BB_example} for an example of such a branched surface.

\begin{figure}[h!]\center
\begin{minipage}{\textwidth}
\centering
\includegraphics[scale=.6]{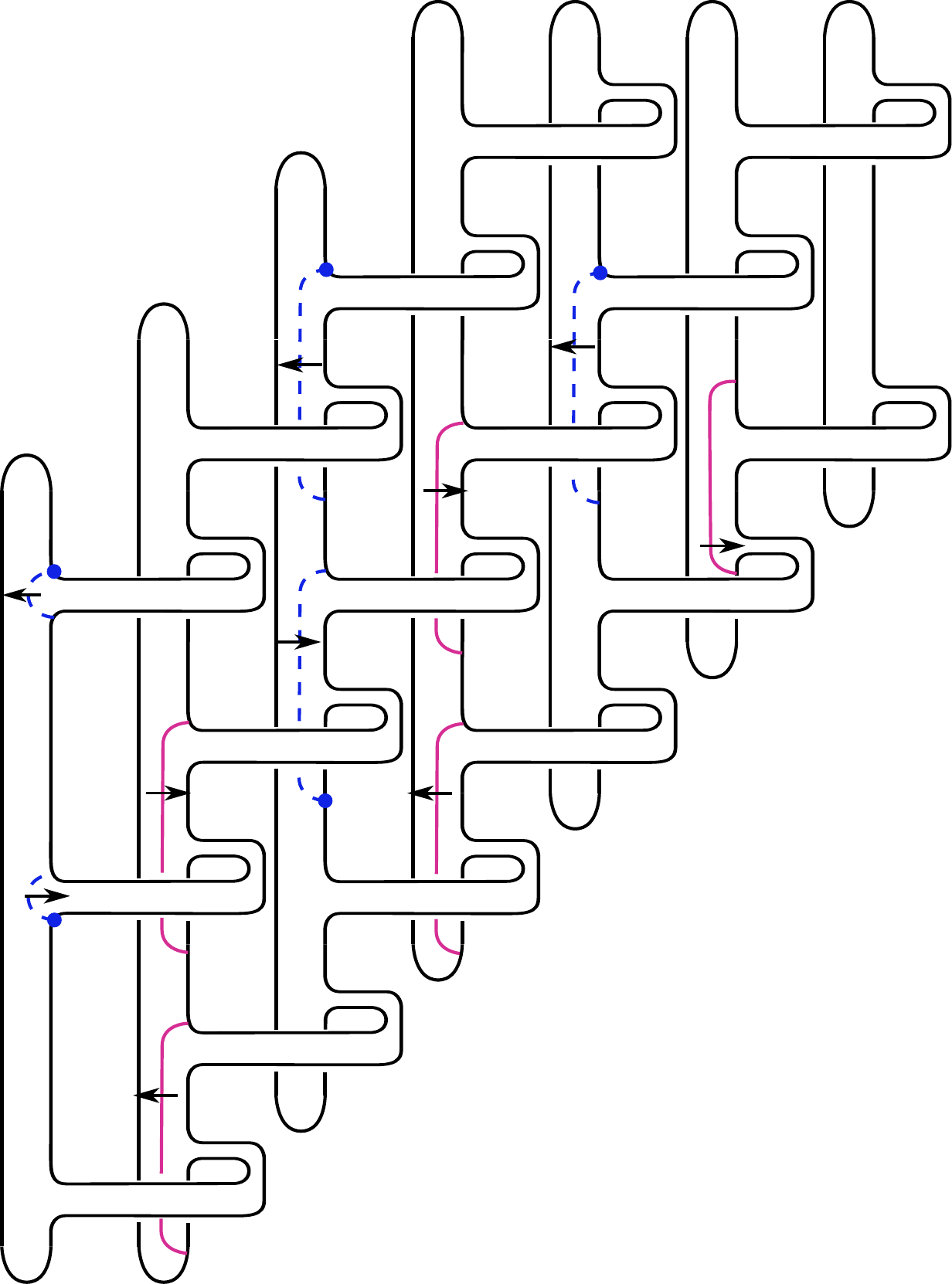}
\caption{A laminar branched surface for the 1-bridge braid $K(7,4,2)$} 
\label{fig:1BB_example}
\end{minipage}
\end{figure}

We check $B$ has no sink disks. No Seifert disk $S_i$ contains both $\alpha_j^{-}$ and $\alpha_\ell^{+}$ arcs, thus there are no polygon sectors. It suffices to check that no disk and band sectors are sink disks. There are at most $t+1$ band sectors: one for each horizontal slice $\mathbbm{h}_0, \mathbbm{h}_1, \ldots, \mathbbm{h}_{t}$.

\begin{defn}
The branch sector containing the bands in $\mathbbm{h}_{i}$ is the \textbf{$i^{\text{th}}$ band sector}, and denoted $\mathbbm{B}_i$.
\end{defn}

We consider 3 cases: $t = 1$, $t = 2$, and $t \geq 3$. 

If $t=1$, then after destabilizing, $K = K(w,b,1) \approx T(b+1,2) \approx T(2,b+1)$ as knots in $S^3$. In Proposition \ref{prop:typea}, we constructed a laminar branched surface $B$ for any knot $K = T(2,n)$, where the induced train track $\tau$ carried all slopes $(-\infty, 2g(K)-1)$. Appealing to Theorem \ref{thm:main} yields a stronger result than the one we seek for Theorem \ref{thm:1bridgebraids}.

Before treating the $t=2$ and $t \geq 3$ cases, we prove:

\begin{lemma} \label{lemma:1bbdisksectors}
Let $B$ be the branch surface described above, for $K(w,b,t)$ with $t \geq 2$. If $b$ is odd (resp. even), the disk sectors $S_1, \ldots S_{b+1}$ (resp. $S_1, \ldots S_b$) are not half sink disks.
\end{lemma}
\vspace{-.6cm}
\begin{proof} \phantom{\qedhere}
If $b$ is odd (resp. even), then every odd Seifert disk among $S_1, \ldots, S_b$ (resp. $S_1, \ldots, S_{b-1}$) contains arcs $\alpha_j^{-}$ cusped via both $(\leftarrow)$ and $(\rightarrow)$ (this is guarenteed since $t \geq 2$). Lemma \ref{lemma:cuspdirections} guarantees all Seifert disks $S_1, S_2, \ldots, S_{b+1}$ (resp. $S_1, S_2, \ldots, S_{b}$) contain arcs cusped via both $(\leftarrow)$ and $(\rightarrow)$. Each of these disks contains an outward pointing cusped arc, hence they are not half sink disks. This completes the proof of Lemma \ref{lemma:1bbdisksectors}.
\end{proof}

We return to the proof of Proposition \ref{prop:1BB_branchedsurface}. 

If $t=2$, we have a three subcases:
\begin{itemize}
\item $b=w-2$, $b \equiv w \equiv 0 \mod 2$ 

No band sectors are half sink disks: $\widehat{\alpha}_b^-$, $\widehat{\alpha}_{b+1}^{-}$, and $\widehat{\alpha}_{b+w-1}^{+}$ point out $\mathbbm{B}_0, \mathbbm{B}_1$ and $\mathbbm{B}_2$ respectively.

By Lemma \ref{lemma:1bbdisksectors}, the disk sectors $S_1, S_2, \ldots, S_{w-2}$ are not half sink disks. $S_{w-1}$, $S_{w-2}$, and $\mathbbm{B}_{1}$ are part of the same branch sector; we already determined $\mathbbm{B}_{1}$ is not a half sink disk. Finally, $\widehat{\alpha}_{b+1}^{+}$ points out of $S_w$, and $B$ is sink disk free. \\

\newpage
\item $b=w-2$, $b \equiv w \equiv 1 \mod 2$

No band sectors are half sink disks: $\widehat{\alpha}_b^{-}$ and $\widehat{\alpha}_{b+w-1}^{+}$ point out of $\mathbbm{B}_0$ and $\mathbbm{B}_2$ respectively. $\mathbbm{B}_1$ and $\mathbbm{B}_{2}$ are in the same branch sector, so $\mathbbm{B}_2$ is not a half sink disk.

By Lemma \ref{lemma:1bbdisksectors}, the Seifert disks $S_1, S_2, \ldots, S_{w-1}$ are not half sink disks. $S_w$ and $\mathbbm{B}_{2}$ are in the same branch sector. $B$ is sink disk free.\\

\item $b < w-2$

$\widehat{\alpha}_b^{-}$ and $\widehat{\alpha}_{b+w-1}^{+}$ point out of $\mathbbm{B}_0$ and $\mathbbm{B}_{2}$ respectively. Either $\widehat{\alpha}_{b+1}^{-}$ (if $w \equiv 0 \mod 2$) or $\widehat{\alpha}_{b+2}^{-}$ (if $w \equiv 1 \mod 2$) points out of  $\mathbbm{B}_{1}$. No band sectors are half sink disks.

If $b \equiv 0 \mod 2$, then by Lemma \ref{lemma:1bbdisksectors}, $S_1, S_2, \ldots, S_{b}$ are not half sink disks. Every even Seifert disk $S_{i}$ with $i \geq b+2$ contains an image arc cusped via $(\rightarrow)$. $S_{b+1}$ is in the same branch sector as $S_1$. All other Seifert disks $S_i, i \geq b+3$ are in the same branch sector as $\mathbbm{B}_1$, which we know has an outwardly cusped arc. B is sink disk free. 

Alternatively, if $b \equiv 1 \mod 2$, then by Lemma \ref{lemma:1bbdisksectors}, $S_1, S_2, \ldots, S_{b+1}$ are not half sink disks. Every even Seifert disk $S_{i}, i \geq b+3$ contains an image arc cusped via $(\rightarrow)$. Every odd Seifert disk $S_{i}, i \geq b+2$ is in the same branch sector as $S_1$. $B$ is sink disk free.
\end{itemize}

Consider a 1-bridge braid with $t \geq 3$. Every odd Seifert disk $S_i$ contains arcs cusped via both $(\leftarrow)$ and $(\rightarrow)$. If $w$ is even (resp. odd), the proof of Lemma \ref{lemma:1bbdisksectors} guarantees $S_1, \ldots, S_w$ (resp. $S_1, S_2, \ldots, S_{w-1}$) are not half sink disks. If $w$ is odd, $S_{1}$ and $S_{w}$ will be in the same disk sector. We conclude no disks sectors are half sink disks. 

Finally, we verify no band sectors are sink disks: $\widehat{\alpha}_b^-$ points out of $\mathbbm{B}_0$. For each $2 \leq i \leq t$, $\widehat{\alpha}_{b+(i-1)(w-1)}^-$ points out of $\mathbbm{B}_{i}$. We need only confirm $\mathbbm{B}_1$ is not a half sink disk. If $b < w-2$, $\widehat{\alpha}_{b+2}^-$ points out of $\mathbbm{B}_1$ (if $w$ is odd) or $\widehat{\alpha}_{b+1}^-$ does (if $w$ is even). If $b=w-2$ and $w \equiv 1 \mod 2$, then $\mathbbm{B}_1$ and $S_w$ are in the same branch sector; we know $\mathbbm{B}_1$ is not a half sink disk. If $b=w-2$ and $w \equiv 0 \mod 2$, then $\widehat{\alpha}_{b+1}^-$ points out of $\mathbbm{B}_1$. We conclude $B$ is sink disk free. 
\end{proof}

\begin{figure}[h!]
\includegraphics[scale=.6]{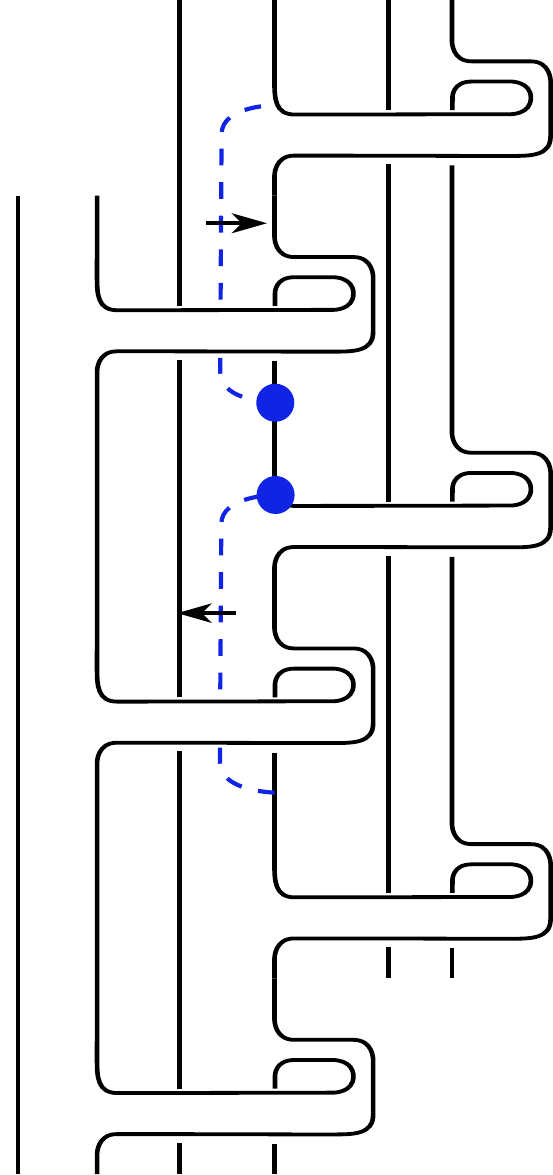}
\caption{These cusping instructions for $\alpha_{m}^{-}$ and $\alpha_{m+w-1}^{-}$ yield a linked pair.}
\label{fig:1BB_NoLinkedArcs}
\end{figure}

\begin{lemma} \label{lemma:1BB_nolinkedarcs}
The train track $\tau$, induced by $B$, admits no linked pairs of arcs.
\end{lemma}

\begin{proof}
All arcs $\alpha_{j}^{-}$ contributing maximally to $\tau$ lie in odd Seifert disks $S_i$. Therefore, the only way to produce a linked pair of arcs is if $\sigma_{m}^{-}$ and $\sigma_{m+w-1}^{-}$ are cusped via $(\rightarrow)$ and $(\leftarrow)$ respectively, as in Figure \ref{fig:1BB_NoLinkedArcs}. Our cusping directions avoid these instructions.
\end{proof}

\begin{defn} \label{defn:NumberCuspedArcs}
Let $K$ be a 1-bridge braid, and $B$ the sink disk free branched surface built in Proposition \ref{prop:1BB_branchedsurface}. Define $\Gamma$ to be the number of product disks used to build $B$. 
\end{defn}

\begin{lemma} \label{lemma:1BB_tau}
The induced train track $\tau$ carries all rational slopes in $(-\infty, g(K))$. 
\end{lemma}

\begin{proof}
By Lemma \ref{lemma:1BB_nolinkedarcs}, we have no linked arcs; therefore, we need only count the total number of product disks $\Gamma$ used to build $B$, and verify $\Gamma \geq g(K)$. It is straightforward to compute the genus of any 1-bridge braid $K$:
\begin{align*}
\chi(F) &= w - ((w-1)t + b) \implies g(K)= \frac{-\chi(F)+1}{2} = \frac{wt-w-t+b+1}{2}
\end{align*}

The value of $\Gamma$ depends on the parity of $w$ and $b$; we analyze the 4 possible cases below:
\begin{center}
\begin{tabular}{c|c|ll}
parity of $w$ & parity of $b$ & \qquad \qquad \qquad \qquad \qquad $\Gamma$  \\
\hline
\hline 
& & \\
even & even & \multirow{2}{*}{$\displaystyle \frac{(t-1)w}{2} + \frac{b}{2}= \frac{wt + b - w}{2}$} \\ 
& & \\
& & \\
\hline 
& & \\
even & odd &  \multirow{2}{*}{$\displaystyle \frac{(t-1)w}{2} + \frac{b+1}{2} = \frac{wt-w+b+1}{2}$} \\ 
& & \\
& & \\
\hline 
& & \\
odd & even & $\displaystyle \frac{(w-1)(t-1)}{2} + \frac{b}{2} = \frac{wt-w-t+b+1}{2}$ \\
& & \\
\hline
& & \\
odd & odd & $\displaystyle \frac{(w-1)(t-1)}{2} + \frac{b+1}{2} = \frac{wt-w-t+b+2}{2}$ \\
& & \\
\end{tabular}
\end{center}

In each case above, $\Gamma \geq g(K)$. Including both sectors of $\tau$ induced by $\alpha_b$ yields a sub-train track $\tau'$ carrying all slopes in $(-\infty, g(K))$. Therefore, for any $K$, the train track $\tau$ induced by the branched surface $B$ carries all rational slopes $r < g(K)$.
\end{proof}

\noindent \textit{Proof of Theorem \ref{thm:1bridgebraids}}. 
By Proposition \ref{prop:1BB_branchedsurface}, for any 1-bridge braid $K \subset S^3$, there exists a sink disk free branched surface $B \subset X_K$. We want to prove that $B$ is laminar by applying Proposition \ref{prop:laminarbranchedsurface}. However, we need to modify the proof of said proposition to show that $B$ cannot fully carry an annulus (our proof of (3) in Proposition \ref{prop:laminarbranchedsurface} relied on a local model that does not apply to 1-bridge braids).

This is straightforward. The cusping directions provided in Proposition \ref{prop:1BB_branchedsurface} focuses our attention to the first Seifert disk; see Figure \ref{fig:1BB_compact_surfaces}. Let the weights of the disk sectors $S_1$ and $S_2$ be $w_1$ and $w_2$ respectively, the weight for the horizontal slice $\mathbbm{h}_1$ is $w_3$, and the weight of the isotoped product disks associated to the first two occurrences of $\alpha_1$ are $w_4$ and $w_5$ respectively. The switch relations for a branched surface to carry a compact surface imply the following:
\begin{align*}
w_1 = w_2 + w_4 \qquad w_3 = w_1 + w_5 \qquad w_3 = w_2 + w_4
\end{align*}
This implies that $w_3= w_1$, thus $w_5=0$. This contradicts that $S$ is fully carried by $B$. We conclude that $B$ cannot carry any compact surface, and therefore does not carry an annulus. Thus, $B$ is laminar.

\begin{figure}[h!]
\labellist
\pinlabel {$w_1$} at 22 200
\pinlabel {$w_2$} at 74 200
\pinlabel {$w_3$} at 75 51
\pinlabel {$w_4$} at 75 140
\pinlabel {$w_4$} at 50 182
\pinlabel {$w_5$} at 45 82
\endlabellist
\includegraphics[scale=1]{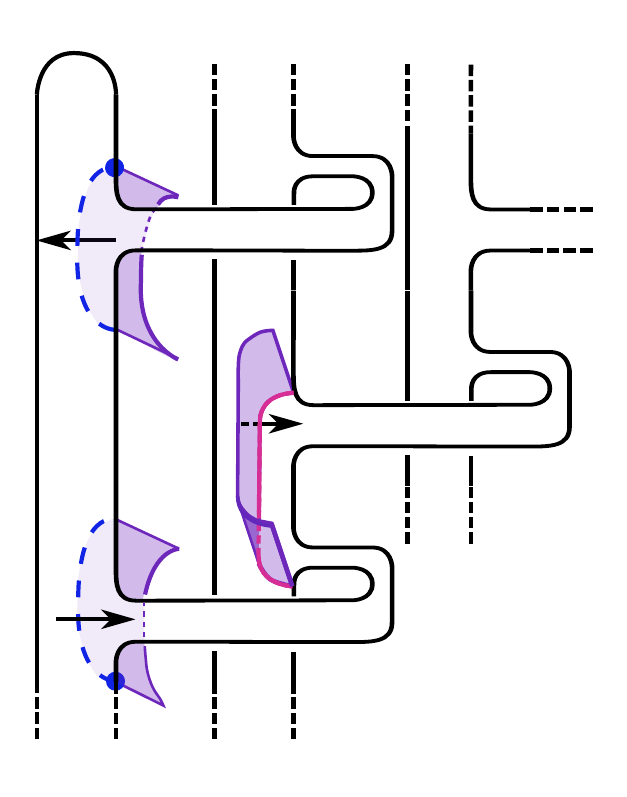}
\caption{A local picture of the branched surface near the first Seifert disk $S_1$. For simplicity, the first product disk appears broken in our figure; it has weight $w_4$. We do not see all of the second product disk, which has weight $w_5$. The standard switch relations induced by the branch loci indicate that $w_5=0$. Thus $B$ cannot carry a compact surface.}
\label{fig:1BB_compact_surfaces}
\end{figure}
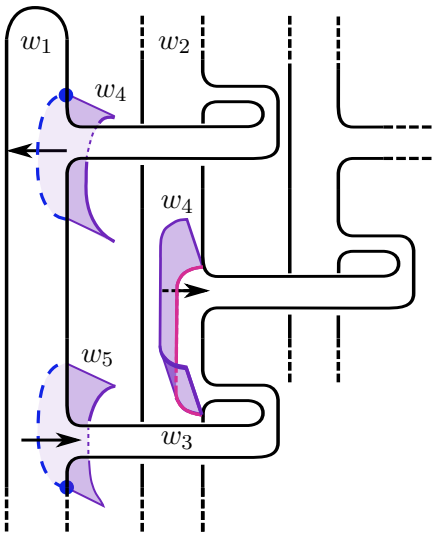

By Lemma \ref{lemma:1BB_tau}, the boundary train track $\tau$ carries all rational slopes $r<g(K)$. Applying Theorem \ref{thm:taolisinkdisk} yields a family of essential laminations $\mathcal{L}_r$ carried by $B$, where $r < g(K)$. Proposition \ref{prop:extendlamination} extends each essential lamination $\mathcal{L}_r$ to a taut foliation $\mathcal{F}_r$ meeting $\partial X_K$ in simple closed curves of slope $r$. Performing $r$-framed Dehn filling produces $S_r^3(K)$ endowed with a taut foliation. \hfill $\Box$

\newpage
\bibliographystyle{amsalpha2}
\bibliography{../../masterbiblio}

\providecommand{\bysame}{\leavevmode\hbox to3em{\hrulefill}\thinspace}
\providecommand{\MR}{\relax\ifhmode\unskip\space\fi MR }
\providecommand{\MRhref}[2]{%
  \href{http://www.ams.org/mathscinet-getitem?mr=#1}{#2}
}
\providecommand{\href}[2]{#2}
\begin{thebibliography}{HRRW15}

\bibitem[Baa13]{Baader:PositiveBraidsSignature}
Sebastian Baader, \emph{Positive braids of maximal signature}, Enseign. Math.
  \textbf{59} (2013), no.~3-4, 351--358. \MR{3189041}

\bibitem[BM18]{BakerMoore:Montesinos}
Kenneth~L. Baker and Allison~H. Moore, \emph{Montesinos knots, {H}opf
  plumbings, and {L}-space surgeries}, J. Math. Soc. Japan \textbf{70} (2018),
  no.~1, 95--110. \MR{3750269}

\bibitem[Ber18]{Berge}
John Berge, \emph{Some knots with surgeries yielding lens spaces},
  https://arxiv.org/abs/1802.09722 (2018).

\bibitem[Bow16]{Bowden:Approx}
Jonathan Bowden, \emph{Approximating {$C^0$}-foliations by contact structures},
  Geom. Funct. Anal. \textbf{26} (2016), no.~5, 1255--1296.

\bibitem[BC15]{BoyerClay2}
Steven Boyer and Adam Clay, \emph{Slope detection, foliations in graph
  manifolds, and {L}-spaces}, https://arxiv.org/abs/1510.02378 (2015).

\bibitem[BC17]{BoyerClay}
Steven Boyer and Adam Clay, \emph{Foliations, orders, representations,
  {L}-spaces and graph manifolds}, Adv. Math. \textbf{310} (2017), 159--234.

\bibitem[BGW13]{BoyerGordonWatson}
Steven Boyer, Cameron~McA. Gordon, and Liam Watson, \emph{On {L}-spaces and
  left-orderable fundamental groups}, Math. Ann. \textbf{356} (2013), no.~4,
  1213--1245.

\bibitem[BNR97]{BrittenhamNaimiRoberts}
Mark Brittenham, Ramin Naimi, and Rachel Roberts, \emph{Graph manifolds and
  taut foliations}, J. Differential Geom. \textbf{45} (1997), no.~3, 446--470.

\bibitem[CLW13]{ClayLidmanWatson}
Adam Clay, Tye Lidman, and Liam Watson, \emph{Graph manifolds,
  left-orderability and amalgamation}, Algebr. Geom. Topol. \textbf{13} (2013),
  no.~4, 2347--2368.

\bibitem[DR]{DelmanRoberts}
Charles Delman and Rachel Roberts, \emph{Personal communication}.

\bibitem[EHN81]{EisenbudHirschNeumann}
David Eisenbud, Ulrich Hirsch, and Walter Neumann, \emph{Transverse foliations
  of {S}eifert bundles and self-homeomorphism of the circle}, Comment. Math.
  Helv. \textbf{56} (1981), no.~4, 638--660.

\bibitem[FS80]{FintushelStern}
Ronald Fintushel and Ronald~J. Stern, \emph{Constructing lens spaces by surgery
  on knots}, Math. Z. \textbf{175} (1980), no.~1, 33--51.

\bibitem[FO84]{FloydOertel}
W.~Floyd and U.~Oertel, \emph{Incompressible surfaces via branched surfaces},
  Topology \textbf{23} (1984), no.~1, 117--125.

\bibitem[Gab83]{Gabai:FoliationsI}
David Gabai, \emph{Foliations and the topology of {$3$}-manifolds}, J.
  Differential Geom. \textbf{18} (1983), no.~3, 445--503.

\bibitem[Gab85]{Gabai:MurasugiSumII}
\bysame, \emph{The {M}urasugi sum is a natural geometric operation. {II}},
  Combinatorial methods in topology and algebraic geometry ({R}ochester,
  {N}.{Y}., 1982), Contemp. Math., vol.~44, Amer. Math. Soc., Providence, RI,
  1985, pp.~93--100.

\bibitem[Gab86]{Gabai:Fibered}
\bysame, \emph{Detecting fibred links in {$S^3$}}, Comment. Math. Helv.
  \textbf{61} (1986), no.~4, 519--555.

\bibitem[Gab90]{Gabai:1BridgeBraids}
\bysame, \emph{{$1$}-bridge braids in solid tori}, Topology Appl. \textbf{37}
  (1990), no.~3, 221--235.

\bibitem[GO89]{GabaiOertel}
David Gabai and Ulrich Oertel, \emph{Essential laminations in {$3$}-manifolds},
  Ann. of Math. (2) \textbf{130} (1989), no.~1, 41--73.

\bibitem[GLV18]{GLV:11LSpace}
Joshua~Evan Greene, Sam Lewallen, and Faramarz Vafaee, \emph{{$(1,1)$}
  {L}-space knots}, Compos. Math. \textbf{154} (2018), no.~5, 918--933.

\bibitem[HRRW15]{HRRW}
Jonathan Hanselman, Jacob Rasmussen, Sarah~Dean Rasmussen, and Liam Watson,
  \emph{Taut foliations on graph manifolds}, https://arxiv.org/abs/1508.05911
  (2015).

\bibitem[Juh15]{Juhasz:Survey}
Andr\'as Juh\'asz, \emph{A survey of {H}eegaard {F}loer homology}, New ideas in
  low dimensional topology, Ser. Knots Everything, vol.~56, World Sci. Publ.,
  Hackensack, NJ, 2015, pp.~237--296.

\bibitem[KR17]{KazezRoberts}
William~H. Kazez and Rachel Roberts, \emph{{$C^0$} approximations of
  foliations}, Geom. Topol. \textbf{21} (2017), no.~6, 3601--3657.

\bibitem[KMOS07]{KMOSz}
P.~Kronheimer, T.~Mrowka, P.~Ozsv\'ath, and Z.~Szab\'o, \emph{Monopoles and
  lens space surgeries}, Ann. of Math. (2) \textbf{165} (2007), no.~2,
  457--546.

\bibitem[LV19]{LeeVafaee}
Cristine Lee and Faramarz Vafaee, \emph{On 3-braids and {L}-space knots},
  https://arxiv.org/abs/1911.01289 (2019).

\bibitem[Li02]{TaoLi:SinkDisk}
Tao Li, \emph{Laminar branched surfaces in 3-manifolds}, Geom. Topol.
  \textbf{6} (2002), 153--194.

\bibitem[Li03]{TaoLi:BoundarySinkDisk}
\bysame, \emph{Boundary train tracks of laminar branched surfaces}, Topology
  and geometry of manifolds ({A}thens, {GA}, 2001), Proc. Sympos. Pure Math.,
  vol.~71, Amer. Math. Soc., Providence, RI, 2003, pp.~269--285.

\bibitem[LM16]{LidmanMoore:PretzelKnots}
Tye Lidman and Allison~H. Moore, \emph{Pretzel knots with {$L$}-space
  surgeries}, Michigan Math. J. \textbf{65} (2016), no.~1, 105--130.

\bibitem[LS09]{LiscaStipsicz}
Paolo Lisca and Andr\'as~I. Stipsicz, \emph{On the existence of tight contact
  structures on {S}eifert fibered 3-manifolds}, Duke Math. J. \textbf{148}
  (2009), no.~2, 175--209.

\bibitem[Liv04]{Livingston:TauInvariant}
Charles Livingston, \emph{Computations of the {O}zsv\'{a}th-{S}zab\'{o} knot
  concordance invariant}, Geom. Topol. \textbf{8} (2004), 735--742.
  \MR{2057779}

\bibitem[Nie19]{Nie:LOPretzelKnots}
Zipei Nie, \emph{Left-orderablity for surgeries on {$(-2,3,2s+1)$}-pretzel
  knots}, Topology Appl. \textbf{261} (2019), 1--6. \MR{3946327}

\bibitem[OS04]{OSz:HolDisks}
Peter Ozsv\'ath and Zolt\'an Szab\'o, \emph{Holomorphic disks and genus
  bounds}, Geom. Topol. \textbf{8} (2004), 311--334.

\bibitem[OS05]{OSz:LensSpaceSurgeries}
\bysame, \emph{On knot {F}loer homology and lens space surgeries}, Topology
  \textbf{44} (2005), no.~6, 1281--1300.

\bibitem[RR17]{Rasmussen2}
Jacob Rasmussen and Sarah~Dean Rasmussen, \emph{Floer simple manifolds and
  {L}-space intervals}, Adv. Math. \textbf{322} (2017), 738--805.

\bibitem[Rob01a]{Roberts:Part1}
Rachel Roberts, \emph{Taut foliations in punctured surface bundles. {I}}, Proc.
  London Math. Soc. (3) \textbf{82} (2001), no.~3, 747--768.

\bibitem[Rob01b]{Roberts:Part2}
\bysame, \emph{Taut foliations in punctured surface bundles. {II}}, Proc.
  London Math. Soc. (3) \textbf{83} (2001), no.~2, 443--471.

\bibitem[Rud93]{Rudolph:QPsliceness}
Lee Rudolph, \emph{Quasipositivity as an obstruction to sliceness}, Bull. Amer.
  Math. Soc. (N.S.) \textbf{29} (1993), no.~1, 51--59.

\bibitem[Sta78]{Stallings:Fibered}
John~R. Stallings, \emph{Constructions of fibred knots and links}, Algebraic
  and geometric topology ({P}roc. {S}ympos. {P}ure {M}ath., {S}tanford {U}niv.,
  {S}tanford, {C}alif., 1976), {P}art 2, Proc. Sympos. Pure Math., XXXII, Amer.
  Math. Soc., Providence, R.I., 1978, pp.~55--60.

\bibitem[Tra19]{Tran}
Anh~Tuan Tran, \emph{Left-orderability for surgeries on twisted torus knots},
  Proc. Japan Acad. Ser. A Math. Sci. \textbf{95} (2019), no.~1, 6--10.
  \MR{3896142}

\bibitem[Vaf15]{Vafaee:TwistedTorusKnots}
Faramarz Vafaee, \emph{On the knot {F}loer homology of twisted torus knots},
  Int. Math. Res. Not. IMRN (2015), no.~15, 6516--6537.

\end{thebibliography}

\end{document}